\documentclass[11pt]{article}
\usepackage{amsthm}
\usepackage{mathrsfs}

\usepackage[show]{ed}

\makeatletter
\def\th@plain{%
  \thm@notefont{}
  \itshape 
}
\def\th@definition{%
  \thm@notefont{}
  \normalfont 
}
\makeatother
\usepackage{amsthm}
\usepackage{mathrsfs}
\usepackage{amsmath}
\usepackage{amsfonts}
\usepackage{latexsym}
\usepackage{amssymb}
\usepackage{graphicx}
\usepackage{tikz}
\usepackage{srcltx,graphicx}
\theoremstyle{plain}
\usepackage{multirow}
\usepackage{soul}
\usepackage{tabularx,ragged2e}
\newcolumntype{C}{>{\Centering\arraybackslash}X} 
\usepackage{caption}
\usepackage{subcaption}
\usepackage{float}
\usepackage{xcolor}
\usepackage{pdfpages}

\newtheorem*{theorem*}{Theorem}
\newtheorem{theorem}{Theorem}[section]

\newtheorem{corollary}[theorem]{Corollary}
\newtheorem{definition}[theorem]{Definition}

\newtheorem{example}[theorem]{Example}
\newtheorem{lemma}[theorem]{Lemma}
\newtheorem*{lemma*}{Lemma}
\newtheorem*{definition*}{Definition}

\newtheorem{remark}[theorem]{Remark}

\numberwithin{equation}{section}
\numberwithin{figure}{section}
\usepackage[a4paper]{geometry}
\newcommand{\beq}{\begin{equation}}
\newcommand{\eeq}{\end{equation}}

\newcommand{\cA}{{\mathcal A}}

\newcommand{\cF}{{\cal F}}

\newcommand{\cR}{\mathcal{R}}

\newcommand{\cK}{{\mathcal K}}

\newcommand{\cV}{{\mathcal V}}

\newcommand{\cB}{{\mathcal B}}

\newcommand{\bx}{\mathbf{x}}
\newcommand{\bfx}{\mathbf{x}}

\newcommand{\br}{\mathbf{r}}

\newcommand{\bfr}{\mathbf{r}}
\newcommand{\bfd}{\mathbf{d}}

\newcommand{\supp}{\mathrm{supp}}

\newcommand{\Rea}{\mathbb{R}}

\usepackage{color}
\usepackage{hyperref}
\definecolor{myblue}{rgb}{0,0,0.6}
\definecolor{darkgreen}{rgb}{0,0.5,0}
\hypersetup{colorlinks=true,
linkcolor=myblue,citecolor=myblue,filecolor=myblue,urlcolor=myblue}

\definecolor{escol}{rgb}{0,0,0.6}
\definecolor{sgcol}{rgb}{0,0,0.7}
\definecolor{estcol}{rgb}{0.5,0,0}
\definecolor{esnewcol}{rgb}{0,0.5,0}
\definecolor{lightgrayl}{RGB}{198,198,198}
\newcommand{\red}[1]{{\color{black}{#1}}}

\newcommand{\beqs}{\begin{equation*}}
\newcommand{\eeqs}{\end{equation*}}
\newcommand{\bit}{\begin{itemize}}
\newcommand{\eit}{\end{itemize}}
\newcommand{\ben}{\begin{enumerate}}
\newcommand{\een}{\end{enumerate}}
\newcommand{\bal}{\begin{align}}
\newcommand{\eal}{\end{align}}
\newcommand{\bals}{\begin{align*}}
\newcommand{\eals}{\end{align*}}
\newcommand{\bre}{\begin{remark}}
\newcommand{\ere}{\end{remark}}
\newcommand{\bpf}{\begin{proof}}
\newcommand{\epf}{\end{proof}}
\newcommand{\ble}{\begin{lemma}}
\newcommand{\ele}{\end{lemma}}
\newcommand{\bco}{\begin{corollary}}
\newcommand{\eco}{\end{corollary}}
\newcommand{\bex}{\begin{example}}
\newcommand{\eex}{\end{example}}
\newcommand{\bth}{\begin{theorem}}
\newcommand{\enth}{\end{theorem}}

\newcommand{\tfa}{\text{ for all }}
\newcommand{\tfor}{\text{ for }}

\newcommand{\tin}{\text{ in }}
\newcommand{\ton}{\text{ on }}
\newcommand{\tas}{\text{ as }}
\newcommand{\tand}{\text{ and }}

\newcommand*{\N}[1]{\left\|#1\right\|}

\newcommand{\noi}{\noindent}

\newcommand{\hsc}{{\hbar}}

\newcommand{\mythmname}[1]{\textbf{\emph{(#1.)}}}

\newcommand{\diff}[2]{\frac{d #1}{d #2}}

\newcommand{\fPML}{f_{\newtheta}}
\newcommand{\DeltaPML}{\Delta_{\newtheta}}
\newcommand{\DeltaPMLj}{\Delta_{\newtheta, j}}

\newcommand{\coeffc}{c}

\newcommand{\newtheta}{{\rm s}}
\DeclareMathOperator{\RC}{\mathsf{RC}}

\newcommand{\Tchi}{\chi^{>}}
\newcommand{\TTchi}{\Tchi}

\newcommand{\parallele}{e_+}
\newcommand{\bparallelepsilon}{\boldsymbol \epsilon_+}
\newcommand{\parallelepsilon}{\epsilon_+}
\newcommand{\parallelejn}{e_{+,j}^n}
\newcommand{\paralleleelln}{e_{+,\ell}^n}
\newcommand{\parallelejnpo}{e_{+,j}^{n+1}}
\newcommand{\parallelu}{u_+}
\newcommand{\parallelujn}{u_{j}^n}
\newcommand{\sequentialu}{u_\times}
\newcommand{\parallelepsilonjn}{\epsilon_{+,j}^n}

\newcommand{\width}{\kappa}

\newcommand{\vertiii}[1]{{\left\vert\kern-0.25ex\left\vert\kern-0.25ex\left\vert #1 
    \right\vert\kern-0.25ex\right\vert\kern-0.25ex\right\vert}}

\newcommand{\matrixC}{{\mathcal{C}}}

\newcommand{\newa}{\alpha}

\usepackage[outdir=./]{epstopdf}
\usepackage{verbatim}

\usepackage{enumitem}

\title{Convergence of overlapping domain decomposition methods 
with PML transmission conditions
applied to nontrapping Helmholtz problems}

\author{J.~Galkowski\thanks{Department of Mathematics, University College London, 25 Gordon Street, London, WC1H 0AY, UK,   \tt J.Galkowski@ucl.ac.uk}
, S.~Gong\thanks{School of Science and Engineering, The Chinese University of Hong Kong, Shenzhen, Guangdong 518172, China, {\tt gongshihua@cuhk.edu.cn} and Shenzhen International Center for Industrial and Applied Mathematics, Shenzhen Research Institute of Big Data, Guangdong 518172, China.}
, I.~G.~Graham\thanks{Department of Mathematical Sciences, University of Bath, UK, {\tt i.g.graham@bath.ac.uk}}
, D.~Lafontaine\thanks{CNRS and Institut de Math\'ematiques de Toulouse, UMR5219; Universit\'e de Toulouse, CNRS; UPS, F-31062 Toulouse Cedex 9, France; \tt david.lafontaine@math.univ-toulouse.fr}
, 
E.~A.~Spence\thanks{Department of Mathematical Sciences, University of Bath, UK, {\tt e.a.spence@bath.ac.uk}}
}

\date{\today}

\begin{document}
\maketitle

\begin{abstract}
We study overlapping Schwarz methods for the Helmholtz equation posed in any dimension with large, real wavenumber and smooth variable wave speed. The radiation condition is approximated by a Cartesian perfectly-matched layer (PML). The domain-decomposition subdomains are overlapping hyperrectangles with Cartesian PMLs at their boundaries. 
The overlaps of the subdomains and the widths of the PMLs are all taken to be independent of the wavenumber. 

For both parallel (i.e., additive) and sequential (i.e., multiplicative) methods, we show that after a specified number of iterations
-- depending on the behaviour of the geometric-optic rays  -- 
the error is smooth and smaller than any negative power of the wavenumber.
For the parallel method, 
the specified number of iterations is less than 
 the maximum number of subdomains, counted with their multiplicity, that a geometric-optic ray 
can intersect. 

These results, which are illustrated by numerical experiments, 
are the first wavenumber-explicit results about convergence of overlapping Schwarz methods for the Helmholtz equation, and the first wavenumber-explicit results about convergence of \emph{any} domain-decomposition method for the Helmholtz equation with a non-trivial scatterer (here a variable wave speed). 
\end{abstract}

\section{Introduction}

\subsection{Informal description of the problem}

We consider the following Helmholtz problem:~for arbitrary $d\geq 1$, given $f\in L^2_{\rm comp}(\Rea^d)$, (normalised) strictly-positive wavespeed $c\in C^\infty(\Rea^d)$ with $\supp (1-\coeffc)$ compact, and wavenumber $k\gg 1$, find $u\in H^1_{\rm loc}(\Rea^d)$ satisfying the 
Helmholtz equation and Sommerfeld radiation condition:
\beq\label{eq:Helmholtz}
k^{-2}\Delta u + \coeffc^{-2} u =f  \quad \tin \Rea^d \quad\tand\quad (\partial_r - ik)u=o(r^{\frac{1-d}{2}}) \tas r:=|x|\to \infty.
\eeq
Since $\mathbb{R}^d$ is unbounded, a standard approximation to this Helmholtz problem is to truncate the domain at an artificial boundary chosen so that the truncated domain contains both 
$\supp (1-\coeffc)$ and $\supp f$ (i.e., the scatterer and the data) and approximate the radiation condition by a \emph{perfectly-matched layer} (PML) \cite{Be:94}. In this paper we assume that the artificial boundary is a hyperrectangle, and that the PML is a Cartesian PML (i.e., the scaling is active in Cartesian coordinate directions; see \S\ref{sec:scaled} below for a precise definition).

Helmholtz solutions oscillate on a length scale of $k^{-1}$, and approximating an arbitrary function oscillating on this scale requires $\sim(kL)^d$ degrees of freedom, where $L$ is the length scale of the domain. For piecewise polynomials of fixed degree, the number of degrees of freedom required is $\gg (kL)^d$ because of the \emph{pollution effect} \cite{BaSa:00, GS3}. The linear systems resulting from finite-element discretisations of \eqref{eq:Helmholtz} are therefore very large. 
Furthermore, since the standard variational formulation of the Helmholtz problem above is not self-adjoint and not coercive, methods that work well for self-adjoint coercive problems, such as Poisson's equation, usually perform very badly for Helmholtz problems (see, e.g., the review \cite{ErGa:12}).

Domain-decomposition (DD) methods approximate the solution of the Helmholtz problem on the computational domain by solving Helmholtz problems on subdomains (either overlapping or non-overlapping) of the original domain, with each subdomain problem involving fewer degrees of freedom than the original problem.
Each iteration of a parallel (a.k.a.~additive) method involves solving decoupled problems on each subdomain (i.e., subdomains do not communicate with each other at this stage). In contrast, each iteration of a sequential (a.k.a.~multiplicative) method involves communication between subdomains at each solve phase.

\subsection{Context:~DD methods using PML}\label{sec:context}

The design and analysis of DD methods for solving the Helmholtz equation 
is a very active area; see, e.g., the reviews \cite{Er:08, ErGa:12, LaTaVu:17, GaZh:19, GaZh:22}. 

A key question is:~what boundary conditions should be imposed on the DD subdomains? It is known that the optimal boundary condition on a particular DD subdomain 
is the Dirichlet-to-Neumann map for the Helmholtz equation posed in the exterior of that subdomain (as a subset of the whole domain)
 \cite[\S2]{NaRoSt:94} \cite[\S2]{EnZh:98}, \cite[\S2.4]{DoJoNa:15}. However, complete knowledge of these maps is 
equivalent to knowing the solution operator for the original problem.
 The design of good, practical subdomain boundary conditions is then the goal of \emph{optimised Schwarz methods}; see, e.g., \cite{Ga:06}.

Using PML as a subdomain boundary condition for Helmholtz DD was first advocated for in \cite{To:98}, and there are now many DD methods for Helmholtz using PML on subdomain boundaries with impressive empirical performance; see, e.g.,  
\cite{EnYi:11c, St:13, ChXi:13, ZeDe:16, TaZeHeDe:19, LeJu:19, RoGeBeMo:22, BoNaTo:23} and the reviews \cite{GaZh:19, GaZh:22}.

It is now understood in a $k$-explicit way how PML approximates the Dirichlet-to-Neumann map associated with the Sommerfeld radiation condition. Indeed, when $c \equiv 1$ and $d=2$, the error between the true Helmholtz solution and the Cartesian PML approximation decreases exponentially in $k$, the width of the PML, and the strength of the scaling by \cite[Lemma 3.4]{ChXi:13} (this result also holds for general scatterers in $d\geq 2$ with a radial PML by \cite{GLS2}). For general smooth $c$ and any $d\geq 2$, this error is smaller than any negative power of $k$ by Theorem \ref{thm:outgoing_approx} below.

However, there is no rigorous understanding of how well PML 
approximates the (more complicated) Dirichlet-to-Neumann maps corresponding to the optimal DD boundary conditions for general decompositions and non-trivial scatterers; i.e., there are no rigorous $k$-explicit results about the convergence of DD methods using PML applied to Helmholtz problems with $k\gg 1$ and non-trivial scatterers.

The only existing $k$-explicit rigorous convergence results are for the sequential ``source transfer" DD methods (which \cite{GaZh:19} showed can be considered as a particular type of optimised Schwarz method) applied to \eqref{eq:Helmholtz} with $c\equiv 1$ (i.e., no scattering). 
These methods involve subdomains that only overlap via the perfectly-matched layers (i.e., the physically-relevant parts of the subdomains do not overlap). For these methods applied to \eqref{eq:Helmholtz} with $c\equiv 1$ (i.e., no scattering), $k$-explicit convergence of the DD method at the continuous level can be obtained from a $k$-explicit result about how well (Cartesian) PML approximates the Sommerfeld radiation condition -- this is precisely because the optimal boundary conditions on the subdomains in this case are the Dirichlet-to-Neumann maps associated with the Sommerfeld radiation condition. Therefore, accuracy of Cartesian PML  when $c\equiv 1$ \cite[Lemma 3.4]{ChXi:13} is then the heart of the $k$-explicit convergence proofs of the source-transfer-type methods in \cite{ChXi:13, LeJu:19, DuWu:20, LeJu:21, LeJu:22} for $c\equiv 1$.

\subsection{Informal description of the main results}

The main results of the present paper, Theorems \ref{thm:strip}-\ref{thm:sweep} and \ref{thm:gen} below, 
concern both parallel and sequential overlapping Schwarz methods at the PDE level (i.e., before discretisation).

\begin{theorem*}[Informal summary of Theorems \ref{thm:strip}, \ref{thm:sweep_strip}, \ref{thm:check}, \ref{thm:sweep}, and \ref{thm:gen}]
For both parallel and sequential overlapping Schwarz methods applied 
to the Cartesian PML approximation of \eqref{eq:Helmholtz}, where the subdomains are hyperrectangles with Cartesian PMLs at their boundaries, the following is true.
After a number of iterations depending on the behaviour of the geometric-optic rays, 
given any $M>0$ there exists $C>0$ such that the error is smooth and smaller than $C k^{-M}$ for all $k>0$. 
\end{theorem*}

In particular, this implies that the fixed-point iterations converge exponentially quickly in the number of iterations  for sufficiently-large $k$.

These results are the first $k$-explicit results about convergence of overlapping Schwarz methods for the Helmholtz equation, and the first $k$-explicit results about convergence of \emph{any} DD method for the Helmholtz equation with a non-trivial scatterer 
(here a variable wave speed).
We highlight the following.
\bit
\item These results are valid on fixed domains for sufficiently-large $k$, i.e., the PML widths and DD overlaps are arbitrary, but independent of $k$. Obtaining results that are also explicit in these geometric parameters of the decomposition will require more technical arguments than those used here. 
\item We make clear exactly what properties of PML are required to obtain these results, and for what other complex absorption operators these results also hold (see Appendix \ref{sec:general}).
\item These results are the 
PML analogues of the results in  \cite{GGGLS1,LS1} 
for parallel overlapping Schwarz methods with impedance boundary conditions, 
although the results of the present paper are much stronger because of the superiority of PML over impedance boundary conditions in this context.
\eit

\subsection{Definition of the overlapping Schwarz methods considered in this paper}\label{sec:def}

\subsubsection{Definition of the subdomains}
Let the $d$-dimensional hyperrectangular domain $\Omega_{\rm int}$ be given by the Cartesian product
$$
\Omega_{\rm int} := \prod_{1 \leq \ell \leq d} (0, L_\ell),
$$
Let $\{\Omega_{{\rm int}, j}\}_{j=1}^N$ be an overlapping decomposition of $\Omega_{\rm int}$ in hyperrectangular subdomains:~
\beq\label{eq:Omegaintj}
\Omega_{{\rm int}, j}  := \prod_{1 \leq \ell \leq d} (a^j_\ell, b^j_\ell)
\eeq

We extend $\Omega_{\rm int}$ and each $\Omega_{{\rm int}, j}$, $1\leq j \leq N$, by adding a PML layer to each, to form the domains $\Omega$ and $\Omega_{j}$.
Namely, let $\width > 0$ and $\width_0 > 0$ denote respectively the PML width on $\Omega$ and the interior PML width and let
\begin{equation}
\label{e:everyoneKnows}
\Omega:= \prod_{1 \leq \ell \leq d} (-\width, L_\ell+\width),
\quad
\Omega_{j}  := \prod_{1 \leq \ell \leq d} (a^j_\ell - \width^{j}_{\ell}, b^j_\ell + \width^{j}_{\ell}),
\end{equation}
where $\width^{j}_{\ell} = \width_0$ if the corresponding edge of $\Omega_{{\rm int},j}$ belongs to the interior of $\Omega$, and $\width$ otherwise; i.e., edges of subdomains that touch the boundary have the same PML width as $\Omega$ (i.e., $\width$), and interior subdomain edges have PML width $\width_0$ which can be different than $\width$.

\subsubsection{The scaled operators}\label{sec:scaled}

We now define standard Cartesian PMLs of width $\width$ at the boundary of $\Omega$ and width $\width_0$ at the boundary of each 
$\Omega_j$. For simplicity, we assume that the same PML scaling function is used inside every PML; however,  
this assumption can be easily removed, with, say, one function used in the PMLs of width $\width$ and another used in the PMLs of width $\width_0$, at the cost of introducing more notation.

Let $\fPML \in C^\infty(\mathbb R)$ (with subscript $\newtheta$ standing for ``scaling") be such that
\begin{equation}\label{e:fLinear}
\{ x:\fPML(x) = 0 \} = \{ x:\fPML'(x) = 0 \} = \{x: x \leq 0\},
\,\, \fPML'(x) > 0 \tfor x>0, 
\tand \fPML''(x)= 0 \tfor x\geq \width_{\rm lin}
\end{equation}
for some $\width_{\rm lin}< \width$ (observe that $\fPML$ is then linear for $x\geq \width_{\rm lin}$). 
(This assumption is to avoid technical issues about propagation of singularities -- see Remark \ref{rem:PoS} below.)

For any $1\leq \ell \leq d$, 
we define the following scaling functions in the $\ell$-direction $g_{\ell}\in C^\infty(\mathbb R^d)$  by
\beqs
g_{\ell} (x_\ell) :=
\begin{cases}
\fPML(x_\ell - L_{\ell}) &\text{if }x_\ell \geq L_\ell, \\
0&\text{if }x_\ell \in (0, L_\ell), \\
- \fPML( - x_\ell) &\text{if }x_\ell \leq 0,
\end{cases}
\eeqs
and similarly, the subdomain scaling functions in the $\ell$-direction $g_{\ell, j}\in C^\infty(\mathbb R^d)$ for $1 \leq j \leq N$ by
\beqs
g_{\ell, j} (x_\ell) :=
\begin{cases}
\fPML(x_\ell - b_{\ell}^j) &\text{if }x_\ell \geq b_{\ell}^j, \\
0&\text{if }x_\ell \in (a_{\ell}^j, b_{\ell}^j), \\
- \fPML( a_{\ell}^j - x_\ell) &\text{if }x_\ell \leq a_{\ell}^j.
\end{cases}
\eeqs
We now define the scaled operators $\DeltaPML$ and $\DeltaPMLj$ by
\beqs
\DeltaPML := \sum_{\ell=1}^d \Big(\frac{1}{1+ig_{\ell}'(x_\ell)} \partial_{x_\ell} \Big)^2,
\eeqs
and
\beq\label{eq:DeltaPMLj}
\DeltaPMLj := \sum_{\ell=1}^d \Big(\frac{1}{1+ig'_{\ell,j}(x_\ell)} \partial_{x_\ell} \Big)^2 \hspace{0.3cm}\tfa 1 \leq j \leq N.
\eeq
Given $c\in C^\infty(\Rea^d)$ that is strictly positive and such that 
\begin{equation}
\label{e:whereIsTheC}
\supp(1-c) \subset \Omega_{\rm int},
\end{equation} let 
\beq\label{eq:PPj}
P_{\newtheta} := - k^{-2} \DeltaPML - \coeffc^{-2}, \qquad P_{\newtheta}^j := - k^{-2} \DeltaPMLj - \coeffc^{-2} \; \tfa 1 \leq j \leq N.
\eeq
These operators are defined on $H^1(\Rea^d)$, but we consider $P_\newtheta$ as a operator $H^1(\Omega)\to H^{-1}(\Omega)(= (H^1_0(\Omega))^*)$, 
and $P^j_{\newtheta}$ as an operator $H^1(\Omega_j)\to H^{-1}(\Omega_j)(= (H^1_0(\Omega_j))^*)$. If Dirichlet boundary data in $H^{1/2}$ is prescribed on the corresponding boundaries, these operators are then invertible.

In the proofs of our main results, a key region is the following 
subset of $\Omega_j$: 
\beqs
\supp(P^j_\newtheta-P_\newtheta):=\Omega_j \cap \bigg(\bigcup_{\ell=1}^d \overline{\big\{ x \in \Rea^d\,:\, g_{\ell,j}(x_\ell) \neq g_{\ell}(x_\ell)\big\}}\bigg);
\eeqs
$\supp(P^j_\newtheta-P_\newtheta)$ is the part of the PML region of $\Omega_j$ where the PML in $\Omega_j$ differs from the PML in $\Omega$ (i.e., where $P^j_\newtheta \neq P_\newtheta$). 
Note that when $\Omega_j$ is an interior subdomain, $\supp(P^j_\newtheta-P_\newtheta) = \Omega_j\setminus\Omega_{\rm int, j}$.

\subsubsection{The Helmholtz problem}
Given $f\in H^{-1}(\Omega)$, 
find $u\in H^1_0(\Omega)$ such that 
\beq\label{eq:PDE}
P_{\newtheta} u=f.
\eeq
The solution $u$ exists and is unique \emph{either} for fixed $k$ when the PML width $\kappa$ is sufficiently large  \cite[Theorem 5.5]{KiPa:10}, \cite[Theorem 5.7]{BrPa:13} \emph{or} for fixed $\kappa$ and scaling function $\fPML$ when $k$ is sufficiently large (by Lemma \ref{lem:res_est} below). 
When $\supp f \subset \Omega_{\rm int}$ and $c\equiv 1$, \cite[Lemma 3.4]{ChXi:13} shows that the error between $u$ and the solution of the true Helmholtz problem in $\Rea^d$ satisfying the Sommerfeld radiation condition decays exponentially in $k$ and $\width$. Theorem \ref{thm:outgoing_approx} below shows that, for 
$f\in H^{-1}(\Omega_{\rm int})$ and $c\not\equiv 1$, for fixed $\width$, this error is $O(k^{-\infty})$; i.e., smaller than any negative power of $k$.

\subsubsection{The partition of unity}

Let $\{\chi_j\}_{j=1}^{N}$ be a partition of unity subordinate to 
$\{ \Omega_{j}\setminus \supp(P_\newtheta^j - P_\newtheta)\}_{1\leq j \leq N}$, i.e.,
 $\{\chi_j\}_{1\leq j \leq N}$ is a family of non-negative elements of $C^\infty(\mathbb R^d)$ such that
\begin{align}
&\tfa\,  x \in \mathbb R^d, \quad \sum_{j=1}^{N} \chi_j(x) = 1,\qquad
\operatorname{supp}\chi_j \cap \Omega \subset \big(\Omega_{ j}\setminus {\supp(P_\newtheta^j - P_\newtheta)}\big)  \text{ for all $j$},\label{eq:PoU}
\end{align}
and, additionally, in a neighbourhood of $\partial\Omega$, each $\chi_j$ does not vary in the normal direction to the boundary of $\Omega$ (this last assumption is for technical reasons, and can easily be achieved in practice).

\subsubsection{The parallel overlapping Schwarz method}\label{sec:parallel}

Given $\parallelu^n \in H^1_0(\Omega)$, for $n\geq 0$ and any $1\leq j \leq N$, let $u_j^{n+1}\in H^1(\Omega_j)$ be the solution to
\begin{equation}\label{eq:localprob}
\begin{cases}
P_{\newtheta}^j
u_j^{n+1} = 
P_{\newtheta}^j (\parallelu^n|_{\Omega_j})
-(P_{\newtheta}\parallelu^n )|_{\Omega_j}
+ f|_{\Omega_j}\in H^{-1}(\Omega_j)\\
u_j^{n+1} = \parallelu^n \in H^{1/2}(\partial \Omega_j),
\end{cases}
\end{equation}
and then set
\beq\label{eq:parallel}
\parallelu^{n+1} := \sum_{j=1}^N \chi_j u_j^{n+1}
\eeq
(where the subscript $+$ indicates that these are the iterates for the parallel/additive method).

Another way to write this is to introduce the \emph{local corrector}
$\mathfrak c_j^n :=u_j^{n+1}-\parallelu^n|_{\Omega_j}\in H^1_0(\Omega_j)$, which satisfies
\beq\label{eq:local_corrector}
P_{\newtheta}^j
\mathfrak c_j^n = 
(f- P_{\newtheta}\parallelu^n )|_{\Omega_j} \in H^{-1}(\Omega_j)
\quad\tand\quad
\parallelu^{n+1} = \parallelu^n + \sum_{j=1}^N \chi_j \mathfrak c_j^n.
\eeq
From this we see that the solution $u$ of \eqref{eq:PDE} is a fixed point of \eqref{eq:localprob}-\eqref{eq:parallel}. Indeed, if $\parallelu^n=u$ then $\mathfrak c^n_j=0$ by uniqueness of the PML problem on $\Omega_j$, and thus $\parallelu^{n+1}=\parallelu^n=u$.
(Note that the subdomain problems for $\mathfrak c_j^n$ are wellposed by the same results that ensure wellposedness of \eqref{eq:PDE}.)

In \S\ref{sec:num1} below, we show that, on the discrete level, this parallel method can be understood as a natural PML-variant of the well-known RAS (restricted additive Schwarz) method.

\subsubsection{The sequential overlapping Schwarz methods}\label{sec:sequential}

The following sequential methods are designed to be used when the subdomains have certain additional structure (see the definitions of ``strips" and ``checkerboards" in \S\ref{sec:strip_checker} below), but can be written abstractly for any set of subdomains as follows.

\

\emph{Forward-backward sweeping.} 
 Given $\sequentialu^0 \in H^1_0(\Omega)$, let $u^0_j:= \sequentialu^0|_{\Omega_j}.$ 
  Then, for $n\geq 0$, 

\begin{enumerate}
\item \label{it:for}(Forward sweeping) For $j=1,\ldots,N$, let $u_j^{2n+1}\in H^1(\Omega_j)$ be the solution to
\beq\label{eq:forward_sweeping}
\begin{cases}
P_{\newtheta}^j
u_j^{2n+1} = 
P_{\newtheta}^j(u_{j,n}^{\rightarrow}|_{\Omega_j})
-(P_{\newtheta}
u_{j,n}^{\rightarrow})|_{\Omega_j} + f|_{\Omega_j} \in H^{-1}(\Omega_j)
\\
u_j^{2n+1} = u_{j,n}^{\rightarrow}\in H^{1/2}(\partial \Omega_j),
\end{cases}
\eeq
where
$$
u_{j,n}^{\rightarrow} := \sum_{\ell<j} \chi_\ell u_\ell^{2n+1} + \sum_{j\leq \ell} \chi_\ell u_\ell^{2n}\,\in H^1_0(\Omega).
$$
Then set
\beq\label{eq:combine1}
\sequentialu^{2n+1} := u^{\rightarrow}_{N+1,n} = \sum_{\ell=1}^N \chi_\ell u_\ell^{2n+1}\,\in H^1_0(\Omega)
\eeq
(where the subscript $\times$ indicates that these are the iterates for a sequential/multiplicative method). 
\item \label{it:back} (Backward sweeping)
For $j=N,\ldots,1$, let $u_j^{2n+2} \in H^1(\Omega_j)$ be the solution to
\beqs
\begin{cases}
P_{\newtheta}^j
u_j^{2n+2} = 
P_{\newtheta}^j(u_{j,n}^{\leftarrow}|_{\Omega_j} ) 
-(P_{\newtheta}
u_{j,n}^{\leftarrow} )|_{\Omega_j} + f|_{\Omega_j}\in H^{-1}(\Omega_j)
\\
u_j^{2n+2} = u_{j,n}^{\leftarrow}\in H^{1/2}(\partial\Omega_j), 
\end{cases}
\eeqs
where
$$
u_{j,n}^{\leftarrow} := \sum_{\ell\leq j} \chi_\ell u_\ell^{2n+1} + \sum_{j< \ell} \chi_\ell u_\ell^{2n + 2}\, \in H^1_0(\Omega).
$$
Then set
\beq\label{eq:combine2}
\sequentialu^{2n+2} := u^{\leftarrow}_{N+1,n} = \sum_{\ell=1}^N \chi_\ell u_\ell^{2n+ 2}\, \in H^1_0(\Omega).
\eeq
\end{enumerate}
Observe that 
the algorithm only uses the local approximations $\{ u^{m}_j\}_{j=1}^N$; i.e., $\sequentialu^{m-1}$ is not directly used to define $\{u^m_j\}_{j=1}^N$.
However, the approximation to the solution of \eqref{eq:PDE} at each step is $\sequentialu^m$, obtained by combining the local approximations $u^m_j$ via the partition of unity in \eqref{eq:combine1}/\eqref{eq:combine2}.

\

In forward-backward sweeping, the subdomains are always visited in the same order and its reverse. 
Some of our results about the sequential method involve changing the order in which the subdomains are visited from iteration to iteration, 
as described as follows.

\

\emph{General sequential method.}
Let $\{\sigma_n\}_n$ be a sequence of permutations $\{ 1, \dots, N \} \mapsto \{ 1, \dots, N \}$.
For each $n$, 
$\sigma_n$ corresponds to the order that that subdomains are visited in the $n$th sweep; e.g., 
$\sigma_n(1)$ is the first subdomain visited on the $n$th sweep, $\sigma_n(2)$ is the second subdomain visited on the $n$th sweep, etc. 
We therefore call $\{\sigma_n\}_n$ a sequence of orderings.

We highlight immediately that forward-backward sweeping is defined by $\sigma_{2n+1}(j) = j, \sigma_{2n}(j) = N-j+1$ (i.e., odd-numbered sweeps visit the subdomains in order, and then even-numbered sweeps visit the subdomains in reverse order).

Given $\sequentialu^0 \in H^1_0(\Omega)$, 
 let $u^0_j:= \sequentialu^0|_{\Omega_j}$. 
Then for $n\geq 1$, 
for $j = 1, \dots, N$, let $u_{\sigma_n(j)}^{n+ 1}\in H^1(\Omega_{\sigma_n(j)})$ be the solution to
\beq\label{eq:general_sweeping}
\begin{cases}
P_{\newtheta}^{\sigma_{n}(j)}
u_{\sigma_n(j)}^{n+1} = 
P_{\newtheta}^{\sigma_n(j)}(u_{\sigma_n(j),n}^{*}|_{\Omega_{\sigma_n(j)}}) 
-(P_{\newtheta}u_{\sigma_n(j),n}^{*})|_{\Omega_{\sigma_n(j)}} 
+ f|_{\Omega_{\sigma_n(j)}} \in H^{-1}(\Omega_{\sigma_n(j)}),
\\
u_{\sigma_n(j)}^{n+ 1} = u_{\sigma_n(j),n}^{*} \in H^{1/2}(\partial\Omega_j),
\end{cases}
\eeq
where
$$
u^*_{\sigma_n(j),n} := \sum_{\ell<j}
 \chi_{\sigma_n(\ell)} u_{\sigma_n(\ell)}^{n+1} + \sum _{\ell \geq j} 
 \chi_{\sigma_n(\ell)} u_{\sigma_n(\ell)}^n\, \in H^1_0(\Omega).
$$
Then set
$$
\sequentialu^{n+1} := u^*_{N+1,n} = \sum_{\ell=1}^N \chi_\ell u_\ell^{n+1} \,\in H^1_0(\Omega).
$$

\subsubsection{Strips and checkerboards}\label{sec:strip_checker}

For $N_\ell \in \mathbb{Z}^+$, $\ell=1,\ldots, d$, 
we say that $\{\Omega_{j}\}_{j=1}^N$ is a \emph{checkerboard} of size $N_1\times \ldots \times N_\ell$ if the following hold.

\bit
\item[(i)] 
For $\ell=1,\ldots, d$, there exist 
$$
0=y^\ell_0 < \ldots < y^\ell_{N_\ell}=L_\ell
$$
such that 
\beqs
\overline{\Omega_{\rm int}}= \bigcup_{j=1}^N \overline{ U_j}, \quad U_j \cap U_i =\emptyset \text{ if } j\neq i,
\eeqs
with each $U_j$ of the form $\prod_{\ell=1}^d (y_{m_\ell-1}^\ell, y_{m_{\ell}}^\ell)$ 
for some $m_\ell\in \{1,\ldots, N_\ell\}$, $\ell=1,\ldots, d$. 

\item[(ii)] The overlapping decomposition $\{\Omega_{{\rm int}, j}\}_{j=1}^N$ of $\Omega_{\rm int}$ comes from extending each $U_j$ 
in each coordinate direction (apart from at $\partial \Omega$).
\item[(iii)] The extensions in (ii) are such that, for all $j$, $\Omega_j$ overlaps only with $\Omega_{j'}$, where $\Omega_{j}$ and $\Omega_{j'}$ are extensions of adjacent nonoverlapping subdomains.
\eit

We say that  $\{\Omega_j\}$ is a $\mathfrak d$-checkerboard if $\mathfrak d := |\{ \ell, \; N_\ell \neq 1 \}|$ i.e. $\mathfrak d$ is the effective dimensionality of the checkerboard.

We say that $\{\Omega_{j}\}_{j=1}^N$ is a \emph{strip} if it is a $1$-checkerboard. To simplify notation, we then assume that the subdomains of a strip are ordered monotonically (i.e., so that $\Omega_i$ only overlaps $\Omega_{i-1}$ and $\Omega_{i+1}$).

Figure \ref{fig:AI} shows an example of the grid formed by the points $\{y_m^\ell\}$ that is used in the definition of a particular $2$-checkerboard of size $4\times 5$, along with a particular subdomain $\Omega_{{\rm int},j}$.

\begin{figure}[h!]
\begin{center}
\begin{tikzpicture}[scale=1]
\def \w{.2};
\def \j{2};
\def \i{1};
   \draw[fill=lightgray, lightgray] ({\i-\w},{(1+\j*(0.05))*(0.75)*\j-\w})  rectangle    ({\i+1+\w},{(1+(\j+1)*(0.05))*(0.75)*(\j+1)+\w});
   \node  at({\i+.5}, {((1+\j*(0.05))*(0.75)*\j+(1+(\j+1)*(0.05))*(0.75)*(\j+1) )/2}){\small{$\Omega_{{\rm int},j}$}};
    \foreach \y in {0,1,2,3,4,5} {
        \draw (0,{(1+\y*(0.05))*(0.75)*\y}) -- (4*1.07,{(1+\y*.05)*.75*\y});
    }

    \foreach \x in {0,1,2,3} {
        \draw (\x,0) -- (\x,{5*.75*(1+5*.05)});
    }

    \draw (4*1.07,0) -- (4*1.07,{5*.75*(1+5*.05)});

    \foreach \x [count=\xi] in {0,1,2,3,4,5} {
        \node[left] at (0,{.75*(\xi-1)*(1+(\xi-1)*.05)}) {$y^2_{{\x}}$};
    }
    \foreach \x [count=\xi] in {0,1,2,3} {
        \node[below] at (\xi-1,0) {$y^1_{{\x}}$};
    }
    \node[below] at (4*1.07,0) {$y^1_{{4}}$};
     
\end{tikzpicture}
\end{center}
\caption{The grid formed by the points $\{y_m^\ell\}$ that is used in the definition of a particular $2$-checkerboard of size $4\times 5$, and the subdomain $\Omega_{{\rm int}, j}$ formed by extending $(y_1^1,y_2^1)\times (y^2_2, y^2_3)$}
\label{fig:AI}
\end{figure}

\subsection{Statement of the results about strips with $c\equiv 1$}

\begin{theorem}[$\coeffc\equiv 1$, strip, parallel method] \label{thm:strip}
Assume that $\coeffc\equiv 1$ and that $\{\Omega_{j}\}_{j=1}^N$ is a strip. Then there exists $k_0>0$ such that 
for any $M,s>0$ there exists $C>0$ 
such that 
the following is true for any 
 $f \in H^{-1}(\Omega)$ and $k \geq k_0$. 
If $u$ is the solution to (\ref{eq:PDE}), and $\{\parallelu^n\}_{n \geq 0}$
is a sequence of iterates for the  parallel overlapping Schwarz method, then
$$
\Vert u - \parallelu^N \Vert_{H^s_k(\Omega)} \leq  Ck^{-M} \Vert u - \parallelu^0 \Vert_{H^1_k(\Omega)}.
$$
\end{theorem}
The norm $\|\cdot\|_{H^s_k(\Omega)}$ appearing in Theorem \ref{thm:strip} is defined for $s\in \mathbb{N}$ by
\beq\label{eq:weighted_norm}
\N{v}^2_{H^s_k(\Omega)} := \sum_{|\alpha|\leq s} \N{ (k^{-1}\partial)^\alpha v}^2_{L^2(\Omega)},
\eeq
and for $s\in (0,\infty)\setminus \mathbb{Z}^+$ by interpolation; i.e., $\|\cdot\|_{H^s_k(\Omega)}$
is the standard norm on $H^s(\Omega)$ except that every derivative is weighted by $k^{-1}$ (using this norm is motivated by the fact that Helmholtz solutions oscillate at frequency approximately $k$).

\begin{theorem}[$\coeffc\equiv 1$, strip, forward-backward sweeping] \label{thm:sweep_strip}
Assume that $\coeffc\equiv 1$ and that $\{\Omega_{j}\}_{j=1}^N$ is a strip. 
Then there exists $k_0>0$ such that
for all $M,s>0$ there exists $C>0$ such that the following is true for all 
 $f \in H^{-1}(\Omega)$, and $k \geq k_0$. 
If $u$ is the solution to (\ref{eq:PDE}), and $\{\sequentialu^n\}_{n \geq 0}$
is a sequence of iterates for the  forward-backward sweeping overlapping Schwarz method, then
$$
\Vert u - \sequentialu^2 \Vert_{H^s_k(\Omega)} \leq C  k^{-M} \Vert u - \sequentialu^0 \Vert_{H^1_k(\Omega)}
$$
(i.e., after one forward and one backward sweep, the error is $O(k^{-\infty})$).
\end{theorem}

\begin{remark}
\red{The requirement that $k_0$ be sufficiently large in Theorems \ref{thm:strip} and \ref{thm:sweep_strip} 
(as well as Theorems \ref{thm:check}, \ref{thm:sweep}, and \ref{thm:gen} below) 
 is used to guarantee that the solution to all subdomain problems exists (see Lemma \ref{lem:res_est} below). Alternatively,
one could allow $k_0>0$ arbitrary provided that we exclude an open neighborhood
of the discrete set of $k$s where $P_{\newtheta}^j$ may not be invertible.
}
\end{remark}

\subsection{Statement of the results about checkerboards with $c\equiv 1$}\label{sec:statement_check}

\begin{theorem}[$\coeffc\equiv 1$, checkerboard, parallel method] \label{thm:check}
Assume that $\coeffc\equiv 1$ and that $\{\Omega_{j}\}_{j=1}^N$ is a $\mathfrak d$-checkerboard of size $N_1 \times \cdots \times N_{\mathfrak d}$ (with $N_i \geq 2$ for all $1\leq i\leq \mathfrak d$). Then there exists $k_0>0$ 
 such that for any $M,s>0$ there exists $C>0$ such that  the following is true for all 
 $f \in H^{-1}(\Omega)$ and $k \geq k_0$. 
 If $u$ is the solution to (\ref{eq:PDE}), and $\{\parallelu^n\}_n \geq 0$
is a sequence of iterates for the  parallel overlapping Schwarz method, then
$$
\Vert u - \parallelu^{N_1 + \cdots + N_{\mathfrak d} - ({\mathfrak d}-1)} \Vert_{H^s_k(\Omega)} \leq Ck^{-M} \Vert u - \parallelu^0 \Vert_{H^1_k(\Omega)}.
$$
\end{theorem}
\noindent Theorem \ref{thm:check} contains Theorem \ref{thm:strip} by 
recalling that a strip is a $1$-checkerboard (with $N_1$ subdomains), and setting ${\mathfrak d}=1$ 
in Theorem \ref{thm:check}.

\

To state the result for the sequential method on the checkerboard, we need the following notion.
Informally, we say that a sequence of orderings $\{\sigma_n\}$ is \emph{exhaustive} if 
(i) given any ordering in the sequence, the sequence also contains the reverse ordering (i.e., where the subdomains are visited in reverse order), and (ii)
given any directed straight line, there exists $n_*$ such that, in  the order of subdomains corresponding to the ordering $\sigma_{n^*}$, the subdomains intersected by the straight line are visited in that order (but not necessarily one after the other) in that sweep.
For the precise definition of exhaustive, see Definition \ref{def:exhaust} below.

For example, one forward sweep and one backward sweep on a strip corresponds to the ordering 
$\{\sigma_1, \sigma_2 \}$ with $\sigma_1(j) = j$, $\sigma_2(j) = N- j +1$; this ordering is exhaustive, since 
(i) the backward sweep is the reverse of the forward sweep (and vice versa), and (ii) any straight line that intersects more than one subdomain intersects them in the order they appear \emph{either} in the forward sweep \emph{or} in the backward sweep.

Examples \ref{ex:exhaust2} and \ref{ex:exhaust4} below describe how to easily construct exhaustive orderings 
of size $2^{\mathfrak d}$
for a $\mathfrak d$-checkerboard. 
 In particular, one can take $\{\sigma_n\}_{1\leq n \leq 2^{\mathfrak d}}$ corresponding to lexicographic order with respect to every vertex; see Figure \ref{fig:lex}.
Lemma \ref{lem:ex_size} shows that an exhaustive sequence of orderings for a $2$-checkerboard must contain at least $2^2=4$ orderings.

\begin{figure}[h]
\begin{center}
\begin{tabular}{c|c|c}
7 & 8 & 9 \\ \hline
4 & 5 & 6 \\ \hline
1 & 2 & 3
\end{tabular} 
\hspace{0.5cm}
\begin{tabular}{c|c|c}
9 & 8 & 7 \\ \hline
6 & 5 & 4 \\ \hline
3 & 2 & 1
\end{tabular}\hspace{0.5cm}
\begin{tabular}{c|c|c}
1 & 2 & 3 \\ \hline
4 & 5 & 6 \\ \hline
7 & 8 & 9
\end{tabular}\hspace{0.5cm}
\begin{tabular}{c|c|c}
3 & 2 & 1 \\ \hline
6 & 5 & 4 \\ \hline
9 & 8 & 7
\end{tabular}
\caption{A sequence of lexicographic orderings of a $3 \times 3$ checkerboard that is exhaustive in the sense of Definition \ref{def:exhaust}.}\label{fig:lex} 
\end{center}
\end{figure}

\begin{theorem}[$\coeffc\equiv 1$, checkerboard, sequential method] \label{thm:sweep}
Assume that $\coeffc \equiv 1$ and that $\{\Omega_{j}\}_{j=1}^N$ 
is a $\mathfrak d$-checkerboard. Then there exists $k_0>0$ such that for any $M,s>0$ there exists $C>0$ such that the following is true for all
 $f \in H^{-1}(\Omega)$ and $k \geq k_0$. 
 If $u$ is the solution to (\ref{eq:PDE}), and $\{\sequentialu^n\}_{n \geq 0}$ is
a sequence of iterates for the  general sequential overlapping Schwarz method following an exhaustive sequence of orderings of size $S$, 
then
$$
\Vert u - \sequentialu^{S}\Vert_{H^s_k(\Omega)} \leq Ck^{-M} \Vert u - \sequentialu^0 \Vert_{H^1_k(\Omega)}.
$$
In particular, we can construct orderings (via Examples \ref{ex:exhaust2} and \ref{ex:exhaust4}) so that
$$
\Vert u - \sequentialu^{2^{\mathfrak d} }\Vert_{H^s_k(\Omega)} \leq k^{-M} \Vert u - \sequentialu^0 \Vert_{H^1_k(\Omega)}.
$$
\end{theorem}

\bre[Comparing the number of solves for the parallel and sequential methods on a $N_1\times \dots\times N_1$ checkerboard]
For a  $N_1^{d}=N_1\times \dots\times N_1$ checkerboard ($N_1\geq 2$)
Theorem \ref{thm:check} ensures smooth and $O(k^{-\infty})$ error after 
$ (d N_1 - 1)$ iterations each containing 
$N_1^{d}$ parallel subdomain solves, whereas 
Theorem \ref{thm:sweep} ensures smooth and $O(k^{-\infty})$ error after 
$2^d$ iterations each containing 
$N_1^{d}$ sequential subdomain solves.
\ere

\subsection{Statement of the most general result about the parallel method}\label{sec:general_statement}

We first recall the definition of the geometric-optic rays associated with the Helmholtz equation \eqref{eq:Helmholtz}. 
Let $P:= -k^{-2}\Delta - c^{-2}$ and let  
\beqs
p(x,\xi) := |\xi|^2 - \coeffc(x)^{-2}, \quad x,\xi\in \Rea^d.
\eeqs
Given initial data $x_0, \xi_0\in \Rea^d$, let $\Phi_t (x_0,\xi_0):= (x(t),\xi(t))$ be the solution of Hamilton's equations
\beq\label{eq:Hamilton}
\diff{x}{t}(t) = \nabla_\xi p\big(x(t),\xi(t)\big),
\quad
\diff{\xi}{t}(t) = -\nabla_x p\big(x(t),\xi(t)\big)
\eeq
with initial conditions $x(0)= x_0$ and $\xi(0)=\xi_0$;
we call such solutions \emph{trajectories of the flow associated to $P$}. The geometric-optic rays are then the spatial projections of the trajectories, i.e., their projections in the $x$-variable. 
Observe that, when $\coeffc\equiv 1$, the solution of \eqref{eq:Hamilton} is just straight line motion $x(t) = x_0 + 2t \xi_0$, $\xi(t) = \xi_0$.
Observe also the \emph{time reversibility} property that if $(x(t),\xi(t))$ is a solution to \eqref{eq:Hamilton} then so is $(x(-t),-\xi(-t))$.

The operator $P$ is \emph{non-trapping} if, given $R>0$ there exists $T>0$ such that, the $x$-projection of any trajectory starting in $\{ x : |x|<R\} $ has left this set by time $T(R)$; see, e.g., \cite[Page 278]{LaPh:89}.
(Strictly this is the definition of being non-trapping forward in time, but, due to time reversibility, this is equivalent to being non-trapping backward in time.) 

Theorem \ref{thm:gen} below depends on a quantity $\mathcal N$ that depends on the 
properties of the trajectories. 
A precise definition of $\mathcal{N}$ is given in \S\ref{ss:proof_gen} below (see also the informal discussion in \S\ref{sec:powers}), but, informally, 
we consider trajectories for the flow associated to $P$ travelling from the PML region of one subdomain to the PML region of another subdomain, and so on. 
$\mathcal{N}$ is then the maximum number of subdomains, counted with their multiplicity, that any such trajectory can travel between.
Importantly, $\mathcal N<\infty$ when $P$ is nontrapping. 

\begin{theorem}[General $c$, arbitrary hyperrectangular subdomains, parallel method]\label{thm:gen}
Assume 
that $P$ is non-trapping.
Then there exists $k_0>0$ 
such that for any $M,s>0$ there exists $C>0$ such that  the following is true for all 
 $f \in H^{-1}(\Omega)$ and $k \geq k_0$. 
If $u$ is the solution to (\ref{eq:PDE}), and $\{\parallelu^n\}_{n \geq 0}$
is a sequence of iterates for the  parallel overlapping Schwarz method, then 
$$
\Vert u - \parallelu^{\mathcal N} \Vert_{H^s_k(\Omega)} \leq  Ck^{-M} \Vert u - \parallelu^0 \Vert_{H^1_k(\Omega)}.
$$
\end{theorem}

When $c\equiv 1$, for a strip, Theorem \ref{thm:strip} follows from Theorem \ref{thm:gen} by showing that $\mathcal{N}= N$, and, for a checkerboard, Theorem \ref{thm:check} follows from Theorem \ref{thm:gen} by showing that $\mathcal{N}= N_1 + \ldots +N_{\mathfrak d} -({\mathfrak d}-1)$; see \S\ref{sec:allowed_constant_coeff} below.

When $c \not \equiv 1$, the quantity $\mathcal{N}$ can, in principle, be calculated from its definition (Definition \ref{def:Ncurly}) using ray tracing.

\bre[The relationship of PML to the optimal subdomain boundary conditions]
Recall from \S\ref{sec:context} that the optimal boundary condition on a particular DD subdomain is the Dirichlet-to-Neumann map for the Helmholtz equation posed in the exterior of that subdomain (as a subset of the whole domain); in particular, the parallel overlapping Schwarz method with a strip decomposition with $N$ subdomains converges in $N$ iterations with these boundary conditions by \cite[Proposition 2.4]{NaRoSt:94} and \cite[Lemma 2.2]{EnZh:98}.

Theorem \ref{thm:strip} shows that when $c\equiv 1$ the parallel method with a strip decomposition with PML at the subdomain boundaries has $O(k^{-\infty})$ error after $N$ iterations; i.e., PML approximates well the optimal boundary condition in this case. 

When $c\not\equiv 1$, Theorem \ref{thm:gen} proves that the parallel method with a strip decomposition has $O(k^{-\infty})$ error after $\mathcal N$ iterations; we highlight that $\mathcal N$ can be much larger than $N$ when $c \not\equiv 1$ and the rays are not straight lines (e.g., if the rays ``turn around" at one end of the domain), and PML does not approximate well the optimal boundary condition in this case. The fact that PML as a subdomain boundary condition cannot take into account scattering occurring outside that subdomain was highlighted in 
\cite[\S4]{GaZh:18} and \cite[\S10]{GaZh:19}.
\ere

\section{Outline of how the main results are proved}\label{sec:ideas} 

\S\ref{sec:idea_error}-\S\ref{sec:3_ingredients}
explain the ideas behind the results for the parallel method (i.e., Theorems \ref{thm:strip}, \ref{thm:check}, and \ref{thm:gen}). \S\ref{sec:idea_sequential} and \S\ref{sec:idea_sequential2} then explain the results for the sequential methods (Theorem \ref{thm:sweep_strip} and \ref{thm:sweep}).

\subsection{The error propagation matrix $\mathbf{T}$ for the parallel method}\label{sec:idea_error}

For any $n \geq 1$ and $1 \leq j \leq N$,  we define the global and local errors
\beq\label{eq:errors}
\parallele ^n := u -  \parallelu^n \in H^1_0 (\Omega)\qquad\text{and } \qquad \parallelejn := u|_{\Omega_j} - \parallelujn \in H^1(\Omega_j).
\eeq
These definitions, the definition of the iterate $\parallelu^n$ \eqref{eq:parallel}, and the fact that $\{\chi_j\}_{j=1}^N$ is a partition of unity \eqref{eq:PoU} imply that
\beq\label{eq:error_decomp}
\parallele^n = \sum_{j=1}^N   \chi_j  \parallelejn.
\eeq
We introduce the notation 
\beq\label{eq:local_phys_error}
\parallelepsilonjn := \chi_j  \parallelejn
\quad\text{ so that }\quad
 \parallele^n = \sum_{j=1}^N  \parallelepsilonjn.
\eeq
We can interpret $\parallelepsilonjn =\chi_j \parallelejn$, after extension by zero from $\Omega_j$, as an element of $H^1_0(\Omega)$, 
since $\parallelejn$ is zero on $\partial \Omega_j\cap \partial \Omega$ and $\chi_j = 0$ on $\partial \Omega_j\setminus \partial \Omega$. 
We call $\parallelepsilonjn $ the \emph{local physical error}. We highlight that 
the fact that $\parallele^n$ is expressed in \eqref{eq:local_phys_error} in terms of $\parallelepsilonjn=\chi_j  \parallelejn$ and not the whole local error $\parallelejn$ is expected, since $\parallelejn$ contains contributions from the PML layers of $\Omega_j$ that are not PML layers of $\Omega$, and hence not part of the original PDE problem \eqref{eq:PDE} (and hence ``unphysical'').

For $1 \leq i  \leq N$ and $v \in H^1_0(\Omega)$ let $\mathcal T_{i} v \in H^1(\Omega_j)$ be the solution to
\beq\label{eq:Tjl}
\begin{cases}
P^i_{\newtheta} \mathcal T_{i} v = P^i_{\newtheta}(v|_{\Omega_i}) - (P_{\newtheta}v)|_{\Omega_i}\qquad \text{in }{\Omega_i}, \\
\mathcal T_{i} v =v \qquad \text{on }\partial\Omega_i.
\end{cases}
\eeq
For technical reasons, we need a set of cut-off functions that have bigger support than the $\chi_j$s. 
Let $\{\Tchi_j\}_{j=1}^N$ be such that, for each $j$, 
$\Tchi_j\in C^\infty(\Rea^d)$, $\supp \Tchi_j\cap \Omega \subset (\Omega_j\setminus \supp (P_{\newtheta}^j- P_{\newtheta}))$ and $\Tchi_j \equiv 1$ on $\supp \chi_j$
(recall that $\Omega_j\setminus \supp (P_{\newtheta}^j- P_{\newtheta})= \Omega_{\rm int, j}$ for an interior subdomain); such $\{\Tchi_j\}_{j=1}^N$ exist since the distance between $\supp \chi_j$ and $\Omega_j\setminus \supp (P_{\newtheta}^j- P_{\newtheta})$ is $>0$ as a consequence of $\supp \chi_j$ and $\supp (P_{\newtheta}^j- P_{\newtheta})$ being closed.
The functions $\{\Tchi_j\}_{j=1}^N$ and $\{\chi_j\}_{j=1}^N$ are illustrated in the simple case of two subdomains (i.e, $N=2$) in Figure \ref{fig:pou}.

\begin{figure}[h!]
\begin{center}
\includegraphics[scale=0.55]{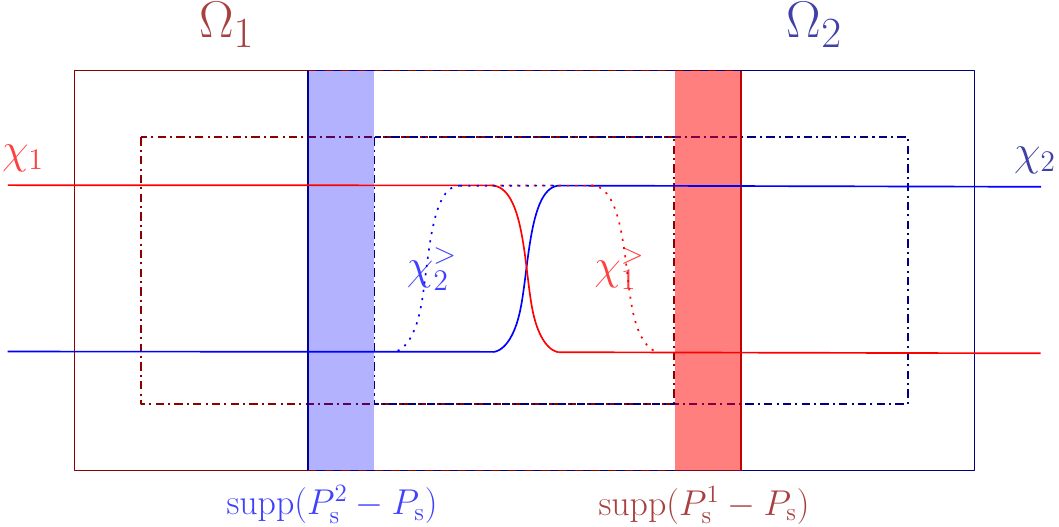}
\caption{The functions $\{\Tchi_j\}$ and $\{\chi_j\}$ for two subdomains}\label{fig:pou}
\end{center}
\end{figure}

We define the \emph{physical error propagation matrix} $\mathbf T$ as
\beq\label{eq:bold_T}
\mathbf T := (\chi_i \mathcal T_{i}\Tchi_j)_{1\leq i,j \leq N} : (H^1_0(\Omega))^N \mapsto  (H^1(\Omega))^N
\eeq
and, for any $n\geq 0$, the physical errors vector ${\bparallelepsilon^n}\in (H^1_0(\Omega))^N$ as

$$
(\bparallelepsilon^n)_j :=\parallelepsilonjn, \, 1\leq j\leq N.
$$

\begin{lemma}[Physical error propagation for the parallel method]\label{lem:model_err_prop}
$\mathbf T:(H^1_0(\Omega))^N \mapsto  (H^1_0(\Omega))^N$ and 
\beq\label{eq:Monday1}
\bparallelepsilon^{n+1} = \mathbf T \bparallelepsilon^{n} \quad\tfa n\geq 0.
\eeq
\end{lemma}

Lemma \ref{lem:model_err_prop} is proved in \S\ref{sec:5error} below.
The analogue of $\mathbf T$ when impedance boundary conditions are imposed on the subdomains was first studied in \cite[\S3]{GGGLS1}.

\subsection{Relating powers of $\mathbf{T}$ to trajectories of the flow defined by $P$}\label{sec:powers}

By \eqref{eq:Monday1}, $\bparallelepsilon^{n}= \mathbf{T}^n \bparallelepsilon^{0}$.
By the definition of $\mathbf T$ \eqref{eq:bold_T}, powers of $\mathbf T$ involve compositions of the maps $\chi_i \mathcal T_{i}\Tchi_j$, and these compositions 
applied to $\parallelepsilon^{0}$
can be interpreted as the error travelling from subdomain to subdomain through the iterations.

The key ingredient to the proof of Theorem \ref{thm:gen} is Lemma \ref{lem:notallowed} below, which relates compositions of $\chi_i \mathcal T_{i}\Tchi_j$ to properties of the 
trajectories of the flow 
travelling between the relevant subdomains. 

We now introduce notation to describe trajectories travelling from subdomain to subdomain. 
Let $\mathcal W$ be the set of finite sequences of elements of $\{ 1, \ldots, N \}$ (i.e., the indices of subdomains) such that, for any $w = \{w_i\}_{1\leq i \leq n} \in \mathcal W$, $w_{i} \neq w_{i+1}$ for all $1\leq i < n$ (i.e., the $(i+1)$th component of $w$ corresponds to a different subdomain than the $i$th component). We call the elements of $\mathcal W$ \emph{words}.

\begin{definition*}[Informal statement of Definition \ref{def:follow} (following a word)]
A trajectory $\gamma$ for the flow associated to $P$ \emph{follows} a word $w \in \mathcal W$ of length $\geq 2$ if it passes through, in order, $\Omega_{w_2}, \ldots, \Omega_{w_n}$ by going through the following sets 
\beq\label{eq:follow_domains}
\supp \Tchi_{w_1} \cap \supp(P^{w_2}_s-P_s), 
\quad
\ldots
\quad, 
\quad
\supp \Tchi_{w_{n-1}} \cap \supp(P^{w_{n}}_s-P_s),
\quad
 \supp \Tchi_{w_n}.
\eeq
\end{definition*}

We make the following immediate remarks:
\bit
\item Figure \ref{drw:word} illustrates a trajectory following a word involving three subdomains of $\Omega$ that do not touch the boundary.
\item To see where the sets in \eqref{eq:follow_domains} come from, observe that, 
by the definition of $\mathcal T_i$ \eqref{eq:Tjl}, 
the functions $\mathcal T_{i}\Tchi_j v$ appearing in the action of $\mathbf T$ \eqref{eq:bold_T} satisfy 
a PDE on $\Omega_i$ with the operator $P_\newtheta^i$ 
and right-hand side supported in $\supp \Tchi_{j} \cap \supp(P^{i}_\newtheta-P_\newtheta)$. 
\item $\supp \chi_{w_i} \cap \supp(P^{w_i+1}_s-P_s)$ is the part of the PML of $\Omega_{w_{i+1}}$ that intersects $\supp \chi_{w_i}$; because $\chi_j$ is zero on the support of $P^j_\newtheta-P_\newtheta$, the sets in \eqref{eq:follow_domains} are disjoint.
\item A trajectory $\gamma$ can follow more than one word because of the overlap of the subdomains. 
\eit

The flow associated to $P$ only allows passage between certain sequences of subdomains. For example, if the $\coeffc\equiv 1$, the trajectories of the flow are straight lines, and thus, since the subdomains are hyperrectangles, in a strip domain there is no trajectory following the word $(1,2,1)$. To see this, consider the analogue of Figure \ref{drw:word} with the subdomain $\Omega_1$ in the same horizontal strip as $\Omega_2$ and $\Omega_3$; in this modified diagram, a straight-line trajectory cannot leave the blue hatched area, go to the purple hatched area, and then come back to the blue hatched area.

\begin{definition} \label{def:allowed}
A word $w\in \mathcal{W}$ is \emph{allowed} if there exists a trajectory for $P$ that follows $w$.
\end{definition}

\begin{figure}[h!]
\begin{center}
\includegraphics[scale=0.9]{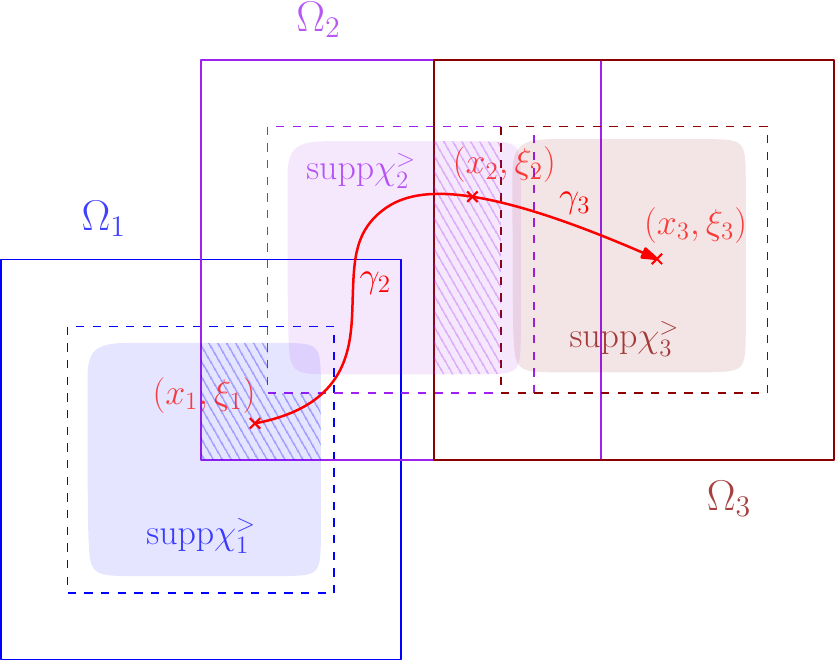}
\caption{A trajectory (in red) following the word $(1,2,3)$ (here $\Omega_1, \Omega_2, \Omega_3$ are interior subdomains of $\Omega$).
The blue hatched shading indicates the domain $\supp \TTchi_1 \cap \supp(P^2_\newtheta-P_\newtheta)$, and the purple hatched shading 
indicates the domain $\supp \TTchi_2 \cap \supp(P^3_\newtheta-P_\newtheta)$. 
(The points $(x_j,\xi_j)$, $j=1,2,3$, and sub-trajectories $\gamma_2,\gamma_3$ are used in the precise definition of  \emph{follow} in Definition \ref{def:follow} below.)
}\label{drw:word}
\end{center}
\end{figure}

\begin{definition} \label{def:Ncurly}
\beq\label{eq:mathcalN}
\mathcal N := \sup_{w \text{ allowed}} |w|.
\eeq
\end{definition}

Since the sets in \eqref{eq:follow_domains} are compact and disjoint, the assumption that $P$ is nontrapping implies that the length of allowed words is bounded above, i.e.,~$\mathcal N < \infty$.

For any $w \in \mathcal W$ of length $|w| :=n\geq 2$, we define the composite map
\beq\label{eq:Tw}
\mathscr T_{w} := 
(\mathbf T)_{w_{n},w_{n-1}} (\mathbf T)_{w_{n-1}, w_{n-2}} \cdots (\mathbf T)_{w_2,w_1}.
\eeq
The presence of the cut-offs in the definition of $(\mathbf T)_{i,j}$ \eqref{eq:bold_T} means that 
$\mathscr T_{w}$ is zero unless $\Omega_{w_j}$ and $\Omega_{w_{j+1}}$ overlap for all $j$.

The heart of the proof of Theorem \ref{thm:gen} is the following lemma.

\ble
\label{lem:notallowed}
\beq\label{eq:Thursday1}
\text{ If }\,\, 
w\in \mathcal W \,\,\text{ is not allowed then, for any $s\geq 1$, } 
 \,\,
\|\mathscr T_w\|_{H^1_k(\Omega) \rightarrow H^s_k(\Omega)}= O(k^{-\infty}).
\eeq
\ele

\bpf[Proof of Theorem \ref{thm:gen} using Lemma \ref{lem:notallowed}]
By the definition of the error propagation matrix $\mathbf{T}$ \eqref{eq:bold_T} and the composite map $\mathscr{T}_w$ \eqref{eq:Tw}, for any $m\in \mathbb{Z}^+$, 
all entries of $\mathbf{T}^{m}$ are sums of operators $\mathscr{T}_w$ over words $w$ of size $m+1$. 
Since words of size $\mathcal{N}+1$ are not allowed by the definition of $\mathcal{N}$ \eqref{eq:mathcalN}, the result that $\|\mathbf{T}^{\mathcal N}\|_{(H^1_k(\Omega))^N \rightarrow (H^s_k(\Omega))^N} = O(k^{-\infty})$ follows from Lemma \ref{lem:notallowed}.
\epf

We show in \S\ref{ss:key_prop2} that Lemma \ref{lem:notallowed} reduces to the following result, which 
describes precisely how, in the $k\to \infty$ limit, the error travels from subdomain to subdomain through the iterations. The lemma is stated in phase space $(x,\xi)$, where $x$ is position and $\xi$ is momentum; note that $\xi$ can be thought of as a Fourier variable measuring the direction of oscillation.

\begin{lemma*}[Informal statement of Lemma \ref{lem:key_prop2}]
For any $v\in H^1_0(\Omega)$, in the $k\to \infty$ limit,
$\mathscr T_w v$ is non-negligible (i.e., not $O(k^{-\infty})$) only at points in phase space that are end points of trajectories that follow the word $w$.
\end{lemma*}

Lemma \ref{lem:key_prop2} is in turn proved using the following propagation result.

\begin{lemma*}[Informal statement of Lemma \ref{lem:key_prop}]
Let $v \in H( \Omega_j)$ be the solution to
$$
P^j_{\newtheta} v = f \hspace{0.3cm} \text{in }{\Omega_j}, \qquad
v = g \hspace{0.3cm} \text{on }\partial\Omega_j.
$$
If $P$ is nontrapping, then,  
in the $k\to \infty$ limit, $v$ is non-negligible only at points in phase space that come from the data $f$ under the flow associated to $P$.
\end{lemma*}

Lemma \ref{lem:key_prop} is stated rigorously and proved  in \S\ref{s:sec_prop}, and we give the main ideas behind its proof in \S\ref{sec:sketch_prop} below. Lemma \ref{lem:key_prop2} is stated in \S\ref{ss:key_prop2}, and proved in \S\ref{ss:proof_key_prop2}. The idea is the following.

\bpf[Sketch proof of Lemma \ref{lem:key_prop2} using Lemma \ref{lem:key_prop}]
Given $v\in H^1_0(\Omega)$, by the definitions of $(\mathbf T)_{i,j}$ \eqref{eq:bold_T} and $\mathcal{T}_{j}$ \eqref{eq:Tjl},
$(\mathbf T)_{w_2, w_1}v$ equals $\chi_{w_2}$ multiplied by a Helmholtz solution 
on $\Omega_{w_2}$ with data $(P_\newtheta^{w_2}-P_\newtheta)\Tchi_{w_1}v$. 
Therefore, Lemma \ref{lem:key_prop} shows that, in the $k\to \infty$ limit, the mass of 
$(\mathbf T)_{w_2, w_1}v$ in phase space 
comes only from $\supp \chi_{w_1} \cap \supp (P_\newtheta^{w_2}-P_\newtheta)$, i.e., the first domain in \eqref{eq:follow_domains}. (This concept of ``mass in phase space" is made precise by the notion of \emph{wavefront set}; see Definition \ref{def:WF} below.)

Similarly, $(\mathbf T)_{w_3, w_2}(\mathbf T)_{w_2, w_1}v$ satisfies a  Helmholtz problem on $\Omega_{w_w}$ with data 

\noi $(P_\newtheta^{w_3}-P_\newtheta)\chi_{w_2}(\mathbf T)_{w_2, w_1}v$. 
Therefore, Lemma \ref{lem:key_prop} shows that, in the $k\to \infty$ limit, the mass of 
$(\mathbf T)_{w_3, w_2}(\mathbf T)_{w_2, w_1}v$ comes only from the mass in 
$\supp \chi_{w_2} \cap \supp (P_\newtheta^{w_3}-P_\newtheta)$ that in turn came from
$\supp \chi_{w_1} \cap \supp (P_\newtheta^{w_2}-P_\newtheta)$.

In the same way, $\mathscr T_w v$ only contains mass that propagates between subdomains following the word $w$, i.e., the subdomains listed in \eqref{eq:follow_domains}.
\epf

\subsection{Sketch proof of the propagation result (Lemma \ref{lem:key_prop})}\label{sec:sketch_prop}

The two  main concepts from semiclassical analysis 
 used in the proof of Lemma \ref{lem:key_prop} are the following.
\ben
\item Propagation of singularities, used in the form that the absence of mass (in phase space) of the solution of $Q u=f$ propagates forwards along the flow defined by $Q$ as long as the trajectory does not intersect the data $f$ (see Lemma \ref{lem:FPR_app} below) and the imaginary part of the principal symbol of $Q$ is non-positive. 
\item Semiclassical ellipticity:~an operator that is elliptic (i.e., has non-zero principal symbol, see \S\ref{sec:elliptic} below)
in some region of phase space can be inverted in that region, up to some error that is 
$O(k^{-\infty})$ 
(see Lemma \ref{lem:ell_wf} below).
\een
Since $P^j_{\newtheta}$ is elliptic in the directions that the scaling of the PML takes place (Lemma \ref{lem:ell_unidri}),
\ben
\item[3. ] A trajectory cannot exit $\Omega_j$ without passing through a point where $P^j_{\newtheta}$ is elliptic (Lemma \ref{lem:go_to_elliptic}).
\een
We now face the issue that 
the assumptions and conclusions of Lemma \ref{lem:key_prop} involve the flow for $P$ (since this is the flow with physical relevance), but to prove the lemma we need to use 
propagation of singularities in the flow defined by $P_\newtheta^j$ (since Lemma \ref{lem:key_prop} is all about solutions of $P_\newtheta^j u= f$), and these two flows are different.
The resolution of this apparent difficulty is that 
\ben
\item[4.] The flow defined by $P_{\newtheta}^j$ equals the flow defined by $P$ 
as long as it doesn't reach a point where $P^j_\newtheta$ is elliptic 
(Lemma \ref{lem:traj_energy_surf}).
\een
By the assumption that $P$ is nontrapping, any trajectory exits the domain; combining this with Points 3 and 4 we obtain
\ben
\item[5.] Any point in $\Omega_j$ is the endpoint of a trajectory of $P_\newtheta^j$ coming \emph{either} from an elliptic point without intersecting the data, \emph{or} comes from the data. 
\een
Point 5 combined with Point 2 implies that trajectories of $P_\newtheta^j$ coming from an elliptic point without intersecting the data carry no mass. Therefore Point 5 combined with Point 1 implies the result;
i.e., that the solution contains mass only on points in phase space that come from the data under the flow associated to $P$ 
(this is illustrated graphically in Figure \ref{drw:key_prop}, next to the proof of Lemma \ref{lem:key_prop}).

\bre[Why we assume that $\fPML$ is linear near the PML boundary]
\label{rem:PoS}
The principal symbol of the scaled operator $P_\newtheta$ has a negative imaginary part; see \eqref{eq:imag_symb} below.
For such an operator, the relevant propagation of singularities results in the semiclassical calculus for domains with a boundary have not yet been written down in the literature. Indeed, the relevant propagation of singularities results in the semiclassical calculus has been proved (i) on manifolds without boundary for operators with symbols whose imaginary parts are single-signed \cite[Theorem E.47]{DyZw:19} (see Lemma \ref{lem:FPE} below) and (ii)
on manifolds with boundary for $P$ \cite{Va:04}. 

The assumption that $\fPML$ is linear near the PML boundary (see~\eqref{e:fLinear}) allows us to use a reflection argument (similar to that used for the Cartesian PML analyses in 2-d in \cite[Proof of Theorem 5.5]{BrPa:13}, \cite[Lemma 3.4]{ChXi:13}) to extend the solution past $\partial\Omega$ and use the propagation results of (i) above, i.e., bypassing the issue of propagation up to the boundary (exactly where in the proofs of the main results we do this is highlighted in Remark \ref{rem:extend} below). Once the relevant propagation results are written down in the literature, this assumption can be removed.
\ere

\bre\mythmname{Why we do not have a variable coefficient in the highest-order term of the PDE}\label{rem:A}
There is also large interest in solving the Helmholtz equation with a variable coefficient in the highest-order term; i.e., the PDE
\beq\label{eq:APDE}
k^{-2}\nabla \cdot( A \nabla u) + c^{-2}u = -f
\eeq
where $A$ is a symmetric-positive-definite-matrix-valued function with $\supp(I-A)\subset \Omega_{\rm int}$.
When $f\in L^2(\Omega_{\rm int})$ and $u$ satisfies the Sommerfeld radiation condition, 
the solution of this problem is well-approximated by Cartesian PML. Indeed, Theorem \ref{thm:outgoing_approx} below holds verbatim:~since $A\equiv I$ when the PML is active, the imaginary part of the principal symbol is unaffected and propagation of singularities (as described informally in Point 1 above) still holds.

However, the analogue of the domain-decomposition methods defined in \S\ref{sec:def} for the PDE \eqref{eq:APDE} involve local problems with the operator 
\beq\label{eq:APDEj}
P_{\newtheta}^j u:= k^{-2}\nabla_{\newtheta,j} \cdot( A \nabla_{\newtheta,j} u) + c^{-2}u,
\eeq
where
\beqs
\nabla_{\newtheta,j} := 
\left(
\begin{array}{c}
{\displaystyle\frac{1}{1+ig_{1,j}'(x_1)} \partial_{x_1} }\\
\vdots
\\
{\displaystyle\frac{1}{1+ig_{d,j}'(x_d)} \partial_{x_d}}
\end{array}
\right)
\eeqs
(compare to \eqref{eq:PPj} and \eqref{eq:DeltaPMLj}). 
For general $A$, the imaginary part of the principal symbol of $P_\newtheta^j$ \eqref{eq:APDEj} no longer has a sign, and propagation of singularities does not hold. 
This is more than just a technical problem with the proofs:~the numerical experiments in \S\ref{sec:num_A} give an example of a simple $A$ 
for which the PDE \eqref{eq:APDE} is well posed, but the parallel Schwarz method applied to \eqref{eq:APDE} with local problems involving \eqref{eq:APDEj} diverges.
\ere

\subsection{Summary of the three ingredients used in the proof of Theorem \ref{thm:gen} (the result about the parallel method)}\label{sec:3_ingredients}

\begin{enumerate}
\item Algebra of the error propagation, showing which words appear in the products $\mathscr T_w$ (discussed in \S\ref{sec:idea_error} and \S\ref{sec:powers}).
\item Semiclassical analysis, showing that 
$\|\mathscr T_w\|_{H^1_k(\Omega)\to H^s_k(\Omega)} = O(k^{-\infty})$
 if $w$ is not allowed -- Lemma \ref{lem:notallowed}.
\item Properties of the trajectories of the flow associated with $P$, 
dictating which $w$ are allowed.
\end{enumerate}
All the discussion in this section has so far been about the parallel method. However, 
the three ingredients above show that to prove results about different methods, such as the sequential methods in \S\ref{sec:sequential}, we only need to understand Point 1, i.e., the algebra of the error propagation, and then check whether the words appearing in the products $\mathcal T_w$ are allowed or not. 
In the next subsection we illustrate this process for forward-backward sweeping on a strip when $\coeffc \equiv 1$ (i.e., the geometric-optic rays are straight lines).

\subsection{Error propagation for forward-backward sweeping on a strip (i.e., the idea behind the proof of Theorem \ref{thm:sweep_strip})}\label{sec:idea_sequential}

We now seek the sweeping analogue of the parallel error propagation result of Lemma \ref{lem:model_err_prop}. To do this, we set up notation for the error analogous to that used for the parallel method in \S\ref{sec:idea_error}. 

Let $\{\sequentialu^n\}_{n \geq 0}$ be a sequence of iterates for 
the forward-backward sweeping defined in \S\ref{sec:sequential}. 
Let 
\beq\label{eq:errors_mult}
 e_{\times} ^n := u -  \sequentialu^n \in H^1_0 (\Omega),\quad\text{ and }\quad e_{{\times}, j}^n := u_{} -  u^n_j, \quad 1\leq j\leq N
\eeq
(with the subscript $\times$ indicating that this is the error for the multiplicative/sequential method).
We define the localised physical error by 
\beq\label{eq:error_sequential}
 \epsilon^n_{{\times}, j} := \chi_j  e_{{\times}, j}^n \in H^1_0(\Omega),
\quad\text{ so that } \quad
 e_{\times}^n = \sum_{j=1}^N  \epsilon^n_{{\times}, j}
\eeq
(in exact analogue with \eqref{eq:local_phys_error}), 
Finally, we define the physical errors vector ${\boldsymbol{\epsilon}}_{\times}^n \in (H^1_0(\Omega))^N$ by
$
({\boldsymbol{\epsilon}}_{\times}^n)_j :=\epsilon^n_{{\times}, j}.
$

For any set of subdomains, $({\mathbf T})_{j,j}=0$ for all $1\leq j\leq N$ -- this follows from the definitions of $\mathbf T$ \eqref{eq:bold_T} and $\mathcal{T}_j$ \eqref{eq:Tjl} and the fact that $\Tchi_j = 0$ where $P_{\newtheta}^j - P_{\newtheta}$ is supported. 
Therefore 
\beq\label{eq:stripLU}
\mathbf T = \mathbf L + \mathbf U,
\eeq
where $\mathbf L$ is lower triangular and $\mathbf U$ is upper triangular, both with zero on the diagonal.

Using the structure \eqref{eq:stripLU}, we arrive at the analogue of Lemma \ref{lem:model_err_prop} (proved in \S\ref{sec:error_sweeping} below).

\begin{lemma}[Error propagation for forward-backward sweeping]\label{lem:model_seq_err_prop}
For forward-backward sweeping, for  $n\geq 0$, 
\beq\label{eq:sweep_fwbw_error}
{\boldsymbol{ \epsilon}_{\times}^{2n+1}} = \mathbf L {\boldsymbol{ \epsilon}_{\times}^{2n+1}} + \mathbf U {\boldsymbol{ \epsilon}_{\times}^{2n}}\quad\tand\quad {\boldsymbol{ \epsilon}_{\times}^{2n+2}} = \mathbf L {\boldsymbol{ \epsilon}_{\times}^{2n+1}} + \mathbf U {\boldsymbol{ \epsilon}_{\times}^{2n+2}}.
\eeq
\end{lemma}

In the parallel case, Lemma \ref{lem:model_err_prop} gives us immediately that ${\boldsymbol \epsilon}^n = {\mathbf T}^n {\boldsymbol \epsilon}^0$.
We now seek the analogue of this for  forward-backward sweeping (see \eqref{eq:sweeping_words} below), using the additional structure of $\mathbf L$ and $\mathbf U$ when 
 $\{\Omega_j\}_{j=1}^N$ is a strip (in the sense of \S\ref{sec:strip_checker}).

\ble \label{lem:banded}
If $\{\Omega_{j}\}_{j=1}^N$ is a strip then 
\beq\label{eq:banded}
{(\mathbf L})_{i, j} = 
\begin{cases}
({\mathbf T})_{j+1, j} &\text{ if } i = j+1, \\
0 &\text{ otherwise},
\end{cases}
\hspace{0.2cm}\text{ and }\hspace{0.2cm}
({\mathbf U})_{i, j} = 
\begin{cases}
(\mathbf T)_{i, i+1} &\text{ if }j = i+1, \\
0 &\text{ otherwise}.
\end{cases}
\eeq
\ele

\bpf 
By definition of the partition of unity 
$\{\chi_j\}_{j=1}^N$ for a strip decomposition, and the ``slightly bigger" set of cut-offs $\{\Tchi_j\}_{j=1}^N$, 
if $|j - i | \neq 1$, then $\Tchi_i = 0$ on $\Omega_j$. 
Thus $({\mathbf T})_{i,j}=0$ if $|j - i | \neq 1$, and 
the result \eqref{eq:banded} follows.
\epf

\bpf[Proof of Theorem \ref{thm:sweep_strip} using Lemma \ref{lem:notallowed}]
By the first equation in \eqref{eq:sweep_fwbw_error} with $n=0$,
$$
(\mathbf I - \mathbf L){\boldsymbol{ \epsilon}_{\times}^{1}} = \mathbf U {\boldsymbol{ \epsilon}_{\times}^{0}};
$$
hence, since $\mathbf L^N = 0$,
$$
{\boldsymbol{ \epsilon}_{\times}^{1}} =  \sum_{\ell=0}^{N-1} \mathbf L^\ell \mathbf U {\boldsymbol{ \epsilon}_{\times}^{0}}.
$$
Similarly, by the second equation in \eqref{eq:sweep_fwbw_error} and then the above,
\begin{align}
{\boldsymbol{ \epsilon}_{\times}^{ 2}} &=  \sum_{m=0}^{N-1} \mathbf U^m \mathbf L {\boldsymbol{\epsilon}_{\times}^{ 1 }} = \sum_{0\leq \ell, m \leq N-1} \mathbf U^m \mathbf L^{\ell+1} \mathbf U {\boldsymbol{\epsilon}_{\times}^{0}}.\label{eq:sweeping_words}
\end{align}
By Lemma \ref{lem:banded}, all the entries of $\mathbf L \mathbf U$ consist of 
$(\mathbf T)_{j+1,j} (\mathbf T)_{j,j+1}$ for some $j$.
Therefore, all the entries of any product of matrices appearing in the second sum in \eqref{eq:sweeping_words} are of the form $\mathcal T_w$, where $w = (\bar w, w_0)$ and $w_0=(j+1, j, j+1)$ (and $\bar w$ is empty if $m=0$ and $\ell=0$ in the sum).

We have now completed the analogue of Point 1 in \S\ref{sec:3_ingredients}; i.e., we have characterised the words appearing the the error propagation.

Now, since $\coeffc\equiv 1$, the trajectories are straight lines and the 
word  $(j+1,j,j+1)$ is not allowed (recall the discussion above Definition \ref{def:allowed}). If $w_0$ is not allowed, then the word $ (\bar w, w_0)$ is not allowed (regardless of $\bar w$). 
Therefore, by Lemma \ref{lem:notallowed}, for any $s\geq 1$,
\beqs
\Big\|\sum_{0\leq \ell, m \leq N-1} \mathbf U^m \mathbf L^{\ell+1} \mathbf U\Big\|_{(H^1_k(\Omega))^N\to (H^s_k(\Omega))^N} = O(k^{-\infty})
\eeqs
and the result of Theorem \ref{thm:sweep_strip} follows from \eqref{eq:sweeping_words} and \eqref{eq:error_sequential}.
\end{proof}

\subsection{Error propagation for the sequential method on a checkerboard (i.e., the idea behind the proof of Theorem \ref{thm:sweep})}\label{sec:idea_sequential2}

The idea of the proof of the analogous result for the general sequential method (i.e., Theorem \ref{thm:sweep}) is the same as for forward-backward sweeping discussed above; i.e., the goal is to show that the words appearing in the products $\mathcal{T}_w$ are not allowed. 
Theorem \ref{thm:sweep} is proved under the condition that the sequence of subdomain orderings 
is exhaustive (defined informally above Theorem \ref{thm:sweep} and defined precisely in Definition \ref{def:exhaust} below). 
This condition ensures that there are enough subdomain orderings so that the words that appear in the error propagation relation correspond
 to trajectories that come back on themselves, and so are not allowed when the geometric-optic rays are straight lines (i.e., when $c\equiv 1$). 
 
Examples \ref{ex:exhaust2} and \ref{ex:exhaust4} construct exhaustive sequences of orderings with $2^{\mathfrak d}$ orderings, where $\mathfrak d$ is the effective dimension of the checkerboard, with Lemma \ref{lem:ex_size} showing that this number of orderings is optimal when $\mathfrak d=2$.
We highlight that 
the source-transfer-type methods of \cite{LeJu:21, LeJu:22} also require $2^d$ ``sweeps" on a $d$-checkerboard for the Helmholtz equation with $c\equiv 1$.

For the general sequential method on a checkerboard when $\coeffc \not\equiv1$, the definition of exhaustive has to be modified to depend on the geometric-optic rays (which are no longer straight lines); see 
Remark \ref{rem:sequential_variable_c} below.

\subsection{Discussion on the importance (or not) of PML conditions on the subdomains}\label{sec:PMLCAP}

The discussion in \S\ref{sec:sketch_prop} above shows that the main property of PML used is that $P^j_{\newtheta}$ is elliptic  in certain directions in the PML region of $\Omega_j$. 
Appendix \ref{sec:general} describes how Theorems \ref{thm:strip}-\ref{thm:gen} hold for a much wider class of operators and subdomain boundary conditions, including when the PML is replaced by a \emph{complex absorbing potential}; i.e., the Helmholtz operator in $\Omega_j$ is replaced by 
\beqs
-k^{-2} \nabla \cdot (A\nabla ) - \coeffc^{-2} - i V_j 
\eeqs
where $V_j \in C^\infty_{\rm comp}(\Rea^d)$ is supported in what was the PML region of $\Omega_j$ and is strictly positive near $\partial\Omega_j$ (see, e.g., 
\cite{SeMi:92, RiMe:93, RiMe:95, Mu:04, St:05}); see Example \ref{ex:CAP}. 
In particular, Example \ref{ex:CAP} shows that the analogues of Theorems \ref{thm:strip}-\ref{thm:sweep} and \ref{thm:gen}
with complex absorption hold when $A\not \equiv I$, in contrast to PML (see Remark \ref{rem:A}); this will be investigated further elsewhere.
Note that complex absorbing potentials can themselves be used to approximate the radiation condition with $O(k^{-\infty})$ error -- see Theorem \ref{thm:CAP} below --  and thus can also be imposed on $\partial \Omega$. 

We highlight, however, that the propagation result Lemma \ref{lem:key_prop} (stated informally in \S\ref{sec:powers}) does not hold if the PML is replaced by a local absorbing boundary conditions, such as the impedance boundary condition.
Indeed, in the $k\to \infty$ limit, mass is reflected by these boundary conditions 
\cite{Mi:00, GLS1} and thus the reflections from $\partial \Omega_j$ create mass that comes from the data \emph{not} under the flow associated with the Helmholtz operator on $\Omega$. These additional reflections make analysing DD methods with impedance boundary conditions challenging \cite{BeCaMa:22, LS1}.

\section{Preliminary results for the PML operators} \label{s:preli}

It will be convenient to work with the semiclassical small parameter $\hbar := k^{-1}$.
We then have
\beqs
P = - \hbar^2\Delta -\coeffc^{-2}, 
\qquad
P_{\newtheta} = - \hbar^2\Delta_{\newtheta} -\coeffc^{-2},
\quad\tand\quad
 P_{\newtheta}^j = -\hbar^2\Delta_{\newtheta, j} -\coeffc^{-2}, \quad 1\leq j \leq N.
\eeqs

\subsection{Useful notation in $\Omega_j$} \label{ss:def_geo}

For any $1\leq j \leq N$, we define the following, denoting the analogous subsets of $\Omega$ by the same notation but omitting the $j$.

\begin{itemize}
\item Let $\mathcal X_\ell^j$ for $0 \leq \ell \leq d-1$ denote the set of $\ell$-dimensional generalised edges of $\Omega_j$. For example, if $d=3$, $\mathcal X_0^j$ are the vertices, $\mathcal X_1^j$ the edges, and $\mathcal X_2^j$ the faces; see Figure \ref{drw:3}.

\item Let $\mathcal Z^j:= \cup_{0 \leq \ell \leq d-1} \mathcal X_\ell^j$ be the set of generalised edges of $\Omega_j$.

\item Let $\mathcal Z_{\partial \Omega}^j:= \big\{ \mathfrak e \in \mathcal Z^j, \; \text{ s.t. } \exists {\mathfrak e}_0 \in \mathcal Z \text{ with } {\mathfrak e} \subset {\mathfrak e}_0 \big\}$, the generalised edges of $\Omega_j$ that are edges of $\Omega$.

\item For $0 \leq \ell \leq d-1$ and ${\mathfrak e} \in \mathcal X_\ell^j$, we write $\xi \parallel {\mathfrak e}$ (informally ``$\xi$ is parallel to ${\mathfrak e}$'') if $\xi \in \hat{\mathfrak e}$, where $\hat{\mathfrak e} := \operatorname{Vect}(e_1, \cdots, e_{\ell})$, where $(e_1, \cdots, e_{\ell})$ are $\ell$ linearly independent vectors tangent to ${\mathfrak e}$; see Figure \ref{drw:4}.

\item For $\mathfrak e \in \mathcal Z^j$, let $\width(\mathfrak e) := \width$ if $\mathfrak e \in \mathcal Z_{\partial \Omega}^j$ and  $\width(\mathfrak e) := \width_0$ otherwise.

\item  For $0 \leq \ell \leq d-1$ and ${\mathfrak e} \in \mathcal X_\ell^j$,
let 
\beq\label{eq:mathfrake}
\widetilde{{\mathfrak e}} := \bigg\{ x + z + \sum_{1\leq i \leq d-\ell} t_i n_i(x), \quad x \in {\mathfrak e} , \; z \in \hat{\mathfrak e},\; t_i > - \width(\mathfrak e) \bigg\},
\eeq
 where, for any 
$x \in {\mathfrak e}$, $(n_i(x))_{1 \leq i \leq d-\ell}$ is a set of $d-\ell$ orthogonal outward-pointing (with respect to $\Omega_j$) normal vectors to ${\mathfrak e}$.
Informally, $\widetilde{{\mathfrak e}}$ is the region into the part of the PML layer of $\Omega_j$ corresponding to edge $\mathfrak{e}$
that is then extended out of $\Omega_j$ to infinity, both normally (corresponding to $t_i n_i(x)$ in \eqref{eq:mathfrake}) and tangentially (corresponding to $z$ in \eqref{eq:mathfrake}); 
see Figure \ref{drw:4}.

\item Finally, let 
\beqs
\widetilde{\mathcal L}^j := \bigcup_{\mathfrak e \in \mathcal X_{d-1}^j} \widetilde{{\mathfrak e}}.
\eeqs
the extended PML layer of $\Omega_j$.

\end{itemize}

\begin{figure}[h!]
\begin{center}
\includegraphics[scale=0.45]{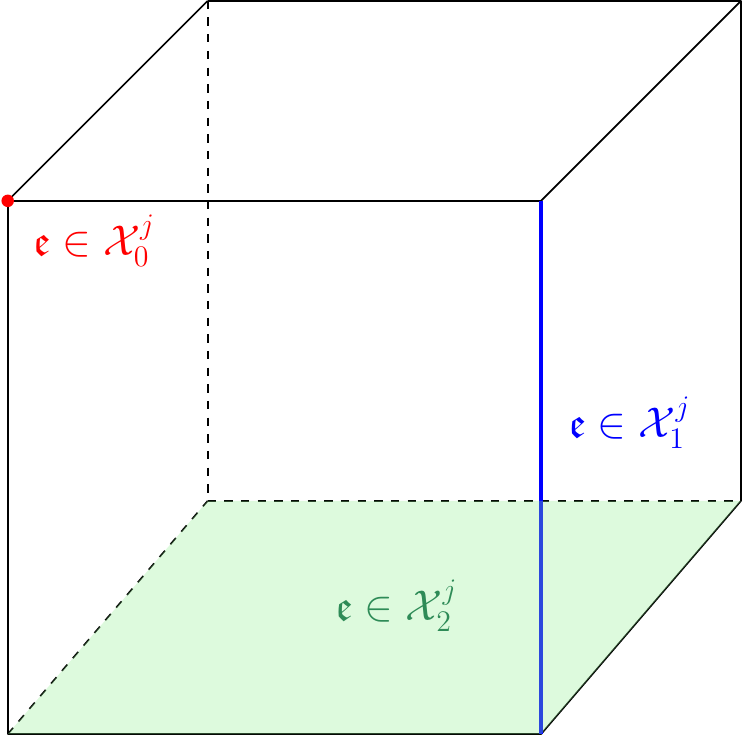}
\caption{Examples of elements of $\mathcal {X}^j_0$, $\mathcal {X}^j_1$ and $\mathcal {X}^j_2$ for a 3-d subdomain.}\label{drw:3}
\end{center}
\end{figure}

\begin{figure}[h!]
\begin{center}
\includegraphics[scale=0.5]{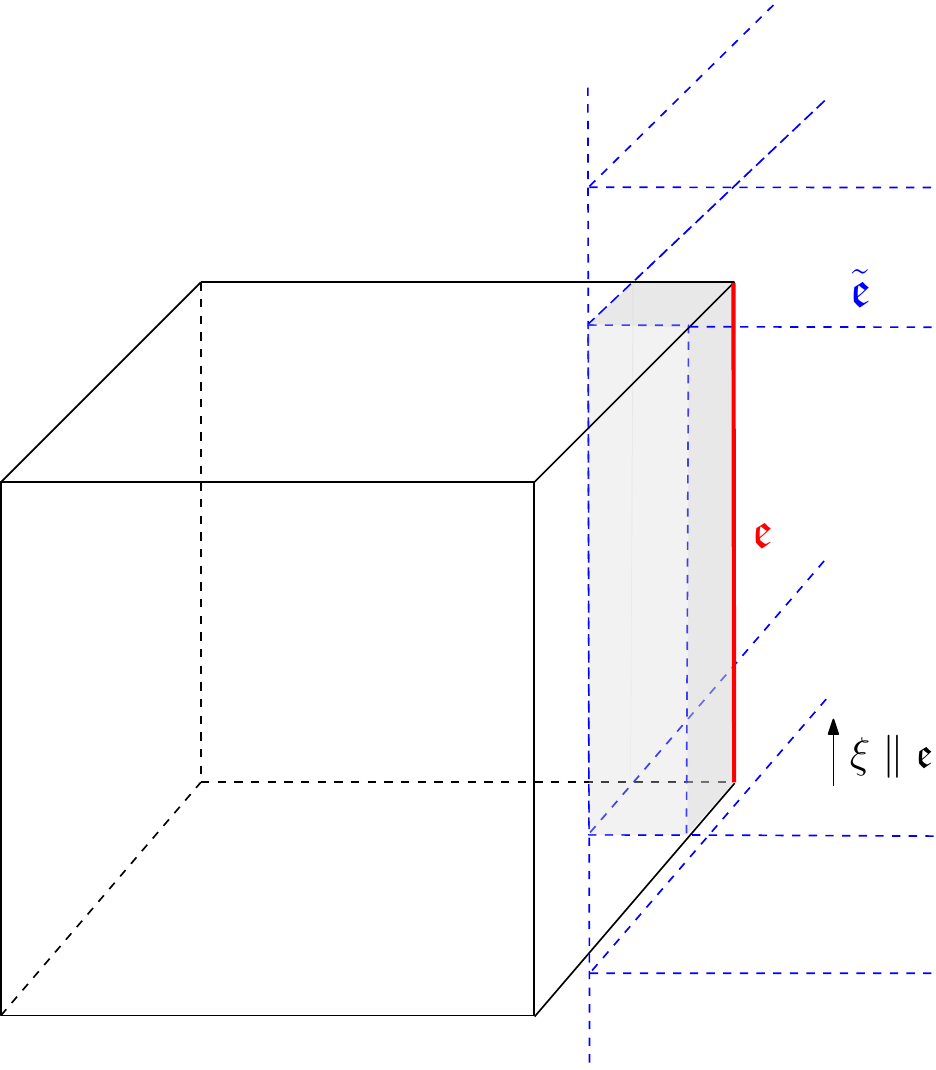}
\caption{Illustration of some elements defined in \S\ref{ss:def_geo} for a 3-d subdomain: an edge $\mathfrak{e} \in {\mathcal X}_1^j$, its associated infinitely extended region $\widetilde{\mathfrak{e}}$, the corresponding region in the PML of $\Omega_j$ (in grey), and a vector $\xi$ such that $\xi \parallel \mathfrak e$.} \label{drw:4}
\end{center}
\end{figure}

\subsection{Extensions of PML solutions in $\Omega_j$}\label{sec:extension}

\red{The main result of this subsection, Lemma \ref{lem:ext_h}, extends solutions of the PML problem in $\Omega_j$ to a slightly-larger domain, via}
odd-symmetric extension on parts of $\partial\Omega_j$ that intersect $\partial \Omega$.

\begin{definition}[Odd-symmetric extension] \label{def:os_ext}
\red{Given $\omega \subset \mathbb R^d$ open and a closed half space $\mathcal{H}$, let $\mathcal V := \omega \cap \mathcal{H}$.
Let $z=(z_1,z')\in \mathbb{R}\times \mathbb{R}^{d-1}$ be Euclidean coordinates such that $\mathcal{H}=\{z_1\geq 0\}$. 
For any $u \in C^\infty(\overline{\cV}),$
 we denote $S_{\mathcal V,\mathcal{H}} u :   { \mathcal V} \cup  \mathcal{V}^\dagger \mapsto \red{\mathbb C}$ its odd-symetric extension to $ \mathcal V \cup  \mathcal V^\dagger$,  $\mathcal V^\dagger := \{ (-z_1, z'), \; z \in \mathcal V\}$, defined by, for any $z \in \mathcal V \cup  \mathcal V^\dagger$
$$
S_{\mathcal V,\mathcal{H}} u (z) := 
\begin{cases}
u(z_1, z') &\text{ if } z\in\mathcal V, \\
-u( - z_1, z')&\text{ if } z \notin\mathcal V.
\end{cases}
$$}
\end{definition}

\red{
\ble\label{lem:extension_dual}
Let $\omega\subset \mathbb{R}^d$ open, $\mathcal{H}\subset \mathbb{R}^d$ a closed half space and $\mathcal{V}:=\omega\cap \mathcal{H}$. Then $S_{\mathcal V,\mathcal{H}}$ extends to a bounded map $H^{-1}(\mathcal{V}^\circ) \to H^{-1}(\mathcal V \cup  \mathcal V^\dagger)$.
\ele

\bpf
Recall that $H^{-1}(\mathcal{V}^\circ) = (H^1_0(\cV^\circ))^*$.
If $ f\in C^\infty (\overline{\cV})$ and $\phi \in C^\infty_c(\cV \cup \cV^\dagger)$, then 
\begin{align*}
\langle S_{\cV,\mathcal{H}} f, \phi \rangle&= \int_{\cV^\dagger } \big(-f(-z_1,z')\big)\phi(z_1,z') \, dz_1 dz'  +\int_\cV f(z_1,z')\phi(z_1,z') \, dz_1 dz' \\
&=\int_\cV f(z_1,z')\Big(\phi(z_1,z')-\phi(-z_1,z')\Big) \, dz_1 dz'.
\end{align*}
Observe that $\psi(z_1,z'):=\phi(z_1,z')-\phi(-z_1,z') \in C^\infty (\cV\cup \mathcal V^\dagger)$ and $\psi(0,z')=0$.
Thus $\psi|_{\cV^\circ} \in H^1_0(\cV^\circ)$ and then 
\beqs
\big|\langle S_{\cV,\mathcal H} f, \phi \rangle\big|\leq 
\N{f}_{H^{-1}(\cV^\circ)} \N{\psi}_{H^1(\cV^\circ)}
\leq
2\N{f}_{H^{-1}(\cV^\circ)} \N{\phi}_{H^1(\cV\cup \cV^\dagger)}.
\eeqs
Since 
$C^\infty(\overline{\cV})$ is dense in $H^{-1}(\cV^\circ)$
and 
$C^\infty_c(\cV \cup \cV^\dagger)$ is dense in $H^1_0(\cV\cup \cV^\dagger)$ (see, e.g., \cite[Page 77]{Mc:00}), the result follows.
\epf
}

\red{
\begin{definition}[Extension from a hyperrectangle $U$ to a bigger hyperrectangle $\widehat{U}$]\label{d:bigRectangle}
Given a hyperrectangle $U$, choose coordinates on $\mathbb{R}^d$ such that $U= \prod_{\ell=1}^d (0, \newa_\ell)$. Define $S_U: H^{-1}(U) \to H^{-1}(\widehat{U})$, where
$$
\widetilde{U}:= \prod_{\ell=1}^d (-2\alpha_\ell,2\alpha_\ell)
$$
as follows.
Define
 \begin{align*}
 \omega_{j,r}&:=\prod_{\ell=1}^{j-1}(-2\newa_\ell,2\newa_\ell)\times (0,2\newa_j)\times \prod_{\ell=j+1}^{d}(0,\newa_\ell),& \mathcal{H}_{j,r}&:=\{x_j\leq \newa_j\},\\
 \omega_{j,l}&:=\prod_{\ell=1}^{j}(-2\newa_\ell,2\newa_\ell)\times \prod_{\ell=j+1}^{d}(0,\newa_\ell),&\mathcal{H}_{j,l}&:=\{x_j\geq 0\}.
 \end{align*}
Then, set 
$$
S_U:= S_{\omega_{d,l},\mathcal{H}_{d,l}}\circ S_{\omega_{d,r},\mathcal{H}_{d,r}}\circ \dots \circ S_{\omega_{1,l},\mathcal{H}_{1,l}}\circ S_{\omega_{1,r},\mathcal{H}_{1,r}}.
$$
\end{definition} 
}

\begin{figure}[h]
\begin{center}
\includegraphics[scale=0.6]{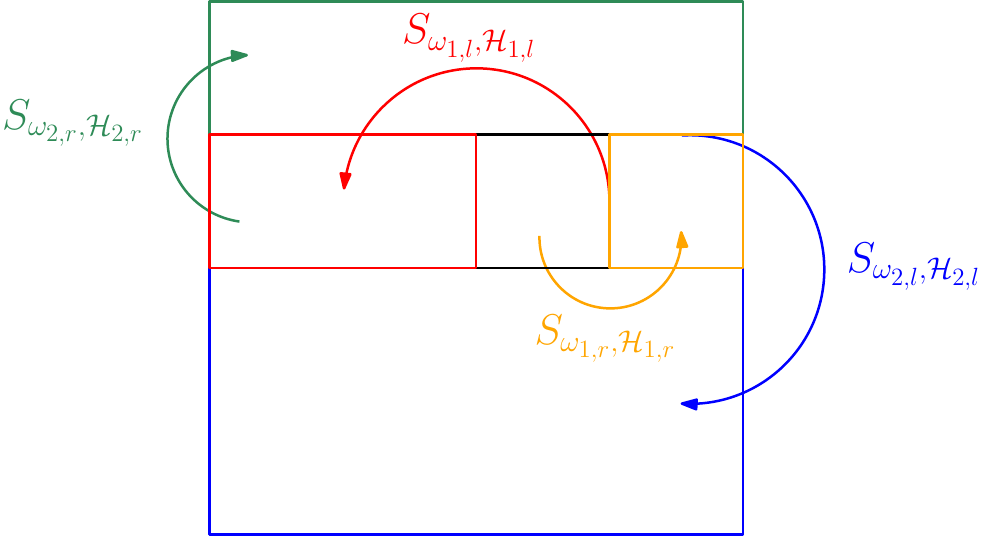}
\caption{\red{The extension of Definition \ref{d:bigRectangle} illustrated in 2-d from the rectangle $U$ (in black) to $\widehat U$ (the largest rectangle).}}\label{drw:5}
\end{center}
\end{figure}

Let $\epsilon_{\rm ext} > 0$ be a parameter to be fixed later. 
We define the extended domain
$$
\widetilde \Omega_j := \Omega_j \cup \bigcup_{\mathfrak e \in {\mathcal Z}_{\partial \Omega}^j} \bigg\{ x  + \sum_{1\leq i \leq d-\ell} t_i n_i(x), \quad x \in {\mathfrak e} , \; 0 \leq t_i < \epsilon_{\rm ext}\bigg\},
$$
where, for any 
$x \in {\mathfrak e}$, $(n_i(x))_{1 \leq i \leq d-\ell}$ is a set of $d-\ell$ orthogonal outward-pointing (with respect to $\Omega_j$) normal vectors to ${\mathfrak e}$;
i.e., $\widetilde \Omega_j$ is $\Omega_j$ extended normally for a distance $\epsilon_{\rm ext}$ from the edges of $\Omega_j$ that are (subset of) edges of $\Omega$.
Note that if $\partial \Omega_j \cap \partial \Omega=\emptyset$, then $\widetilde{\Omega}_j= \Omega_j$.

\begin{figure}[h]
\begin{center}
\includegraphics[scale=0.6]{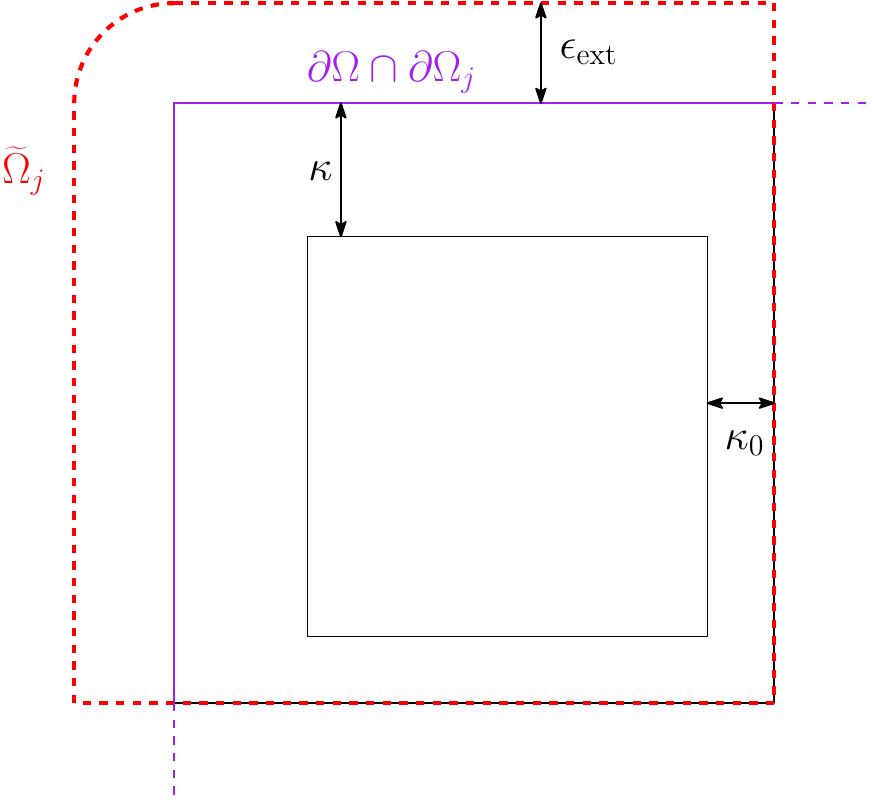}
\caption{The extended domain $\widetilde{\Omega}_j$ for a subdomain $\Omega_j$ belonging to a corner of $\Omega$. The subdomain is only extended through $\partial \Omega \cap \partial \Omega_j$.}
\end{center}
\end{figure}

\red{
\begin{remark}
\label{r:EuanAgitated}
We emphasise that $\widetilde{\Omega}_j$ is an extension of $\Omega_j$ only through $\partial\Omega\cap \partial\Omega_j$; i.e. boundary faces of $\Omega_j$ that are also subsets of the boundary of $\Omega$. By assumptions~\eqref{e:fLinear} and~\eqref{e:whereIsTheC}, the coefficients of $P_s^j$ are constant near the extension faces.
\end{remark}
}

\red{With Lemma \ref{lem:extension_dual} and Definition \ref{d:bigRectangle} in hand, we now define the extension of distributions from $\Omega_j$ to $\widetilde{\Omega}_j$ via the extension $S_{\Omega_j}$ of distributions from $\Omega_j$ to $\widehat{\Omega}_j$.} 

\begin{definition}[Extension \red{of distributions} from ${\Omega}_j$ to $\widetilde{\Omega}_j$] \label{def:ext}
\red{For $f\in H^{-1}(\Omega_j)$, we define $S_jf\in H^{-1}(\widetilde{\Omega}_j)$ by
\beqs
\langle S_jf, u\rangle
_{H^{-1}(\widetilde{\Omega}_j)\times H^1_0(\widetilde{\Omega}_j)}
 := \langle S_{\Omega_j}f, u1_{\widetilde{\Omega}_j}\rangle_{H^{-1}(\widehat{\Omega}_j)\times H^1_0(\widehat{\Omega}_j)}
\quad \tfor u\in H_0^1(\widetilde{\Omega}_j),
\eeqs
where, since $\Omega_j$ is a hyperrectangle (see~\eqref{e:everyoneKnows}), we define $\widehat{\Omega}_j$ and $S_{\Omega_j}:H^{-1}(\Omega_j)\to H^{-1}(\widehat{\Omega}_j)$ in the sense of Definition~\ref{d:bigRectangle}.}
\end{definition}

We now fix $\epsilon_{\rm ext} > 0$ small enough so that 
$\width - \epsilon_{\rm ext} > \width_{\rm lin}$ and $\chi_j$ does not vary normally in an $\epsilon_{\rm ext}$-neighbourhood of  $\partial\Omega$ (recall from \S\ref{sec:def} that we assume that in a neighbourhood of $\partial\Omega$, each $\chi_j$ does not vary normally). As consequence of the latter, 
\begin{equation} \label{eq:small_ext}
\tfa\,  j \in \{1,\cdots,N\}, \hspace{0.3cm} \big(S_j \chi_j v \big)
= \chi_j \big(S_j v \big) \quad\tfa v\in H_0^1(\Omega_j).
\end{equation}

\begin{lemma}[Extension for homogeneous Dirichlet data]\label{lem:ext_h}
Suppose $u \in H^1(\Omega_j)$, $f \in H^{-1}(\Omega_j)$ and $g\in H^{1/2}(\partial\Omega_j)$ are such that 
$$
\begin{cases}
P^j_{\newtheta} u = f \text{ in } \Omega_j, \\
u = g \text{ on } \partial \Omega_j.
\end{cases}
$$
Let $\widetilde u := S_{j} u$ and $\widetilde f := S_{j} f \in H^{-1}(\widetilde\Omega_j)$.
If $g=0$ on $\partial \Omega_j \cap \partial\Omega$, then $\widetilde u\in H^1(\widetilde\Omega_j)$ and 
\begin{equation} \label{eq:ext_ver_eq}
 P^j_{\newtheta} \widetilde u = \widetilde f \text{ in }\widetilde \Omega_{j}. 
\end{equation}
\end{lemma}

Although we have so far been considering $P^j_\newtheta$ as an operator on $\Omega_j$, it is defined on all of $\Rea^d$ by \eqref{eq:DeltaPMLj} and \eqref{eq:PPj}, and so its application in \eqref{eq:ext_ver_eq} on $\widetilde \Omega_j$ makes sense.

\bpf[Proof of Lemma \ref{lem:ext_h}]
\red{The fact that $u\in H^1(\widetilde{\Omega}_j)$ follows from $g=0$ on $\partial\Omega_j\cap \partial\Omega$.}
Since $\width - \epsilon_{\rm ext} > \width_{\rm lin}$, the extension only occurs when the PML scaling functions $g_{\ell,j}$ are either linear or zero \red{and where $c\equiv 1$ (see Remark~\ref{r:EuanAgitated})}. 
Therefore, by the definitions in \S\ref{sec:scaled}, the coefficients of $P^j_\newtheta$ are constant near $\partial\Omega_j \red{\cap \partial \Omega}$.
The result \eqref{eq:ext_ver_eq} then follows from the definition of the weak derivative and integrating by parts. Indeed, 
in 1-d the basic calculation is that, for $\phi \in C^\infty_c(\Rea)$, if $u(0)=0$ and $u''(x) =f(x)$ for $x>0$, then 
$\int_{-\infty}^\infty \widetilde{u}(x) \phi''(x) \, dx = \int_{-\infty}^\infty \widetilde{f}(x) \phi(x) \, dx$; i.e. $\widetilde{u}''(x) = \widetilde{f}(x)$ and so 
$P(S_j u) = S_j(Pu)$ with $Pu = u''$. For $d \geq 2$, the calculation is more complicated, but the basic idea is the same, crucially using that the coefficients of $P^j_\newtheta$ are constant near $\partial\Omega_j
\red{\cap \partial\Omega}$.
\epf

\subsection{Semiclassical ellipticity statements for $P^j_{\newtheta}$}

We begin with a computation quantifying the ellipticity of $P^j_{\newtheta}$ at the symbolic level.

\begin{lemma} \label{lem:comp_symbol} 
Let $p^j_{\newtheta}$ denote the semiclassical principal symbol associated with $P^j_{\newtheta}$.
Given a compact set $K\subset \Rea^d$, there exists $C>0$ such that for all $x\in K$ and $\xi\in \Rea^d$ the following is true.
\begin{enumerate}
\item\label{it:symb_ell_dir} \emph{(Bounding the symbol below by the PML scaling function)}
$$
|p^j_{\newtheta}(x, \xi)| \geq C \sum_{m=1}^d \xi_m^2 g_{x_m,j}'(x_m).
$$
\item\label{it:symb_ell_infnt} \emph{(Ellipticity at infinity in the $\xi$ variable)}
$$
\text{ if } \quad|\xi| \geq C\quad\text{ then }\quad \langle \xi\rangle^{-2}|p^j_{\newtheta}(x, y, \xi_1, \xi_2)| \geq C^{-1}.
$$
\end{enumerate}
Furthermore, the same properties hold for $P^j_{\newtheta}$ replaced with $P_{\newtheta}$.
\end{lemma}
\begin{proof}
Observe that
\beqs
p^j_{\newtheta}(x, \xi) =
\sum_{\ell=1}^d \frac{ \xi_\ell^2 \big(1- i g_{\ell,j}'(x_\ell)\big)^2}{ \big(1+ g'_{\ell,j}(x_\ell)^2\big)^2} - \coeffc^{-2}(x)\nonumber
=\sum_{\ell=1}^d \xi_\ell^2 \frac{\Big( 1- g_{\ell,j}'(x_\ell)^2 - 2 i g_{\ell,j}'(x_\ell)\Big)}{\big( 1+ g_{\ell,j}'(x_\ell)^2\big)^2} - \coeffc^{-2}(x),
\eeqs
so that
\begin{align}
\label{eq:imag_symb}
\operatorname{Im} p^j_{\newtheta}(x,\xi)
=-2\sum_{\ell=1}^d \xi_\ell^2 \frac{g_{\ell,j}'(x_\ell)}{\big( 1+ g_{\ell,j}'(x_\ell)^2\big)^2} \,\,\tand\,\,
\operatorname{Re} p^j_{\newtheta}(x,\xi) = \sum_{\ell=1}^d \xi_\ell^2  \frac{\big(1- g_{\ell,j}'(x_\ell)^2\big)}{ \big(1+ g'_{\ell,j}(x_\ell)^2\big)^2} -\coeffc^{-2}(x). 
\end{align}
On the one hand, the first equation in \eqref{eq:imag_symb} implies that, for $x \in K$
\beqs
\big| \operatorname{Im} p^j_{\newtheta} (x,\xi)\big| \geq C
\sum_{\ell=1}^d \xi_\ell^2 g_{\ell,j}'(x_\ell),
\eeqs
from which Point \ref{it:symb_ell_dir} follows. 

Since $\Im(1-s^2-2is)\leq 0$ for $s\in [0,\infty]$ with equality only when $s=0$,  
given $S>0$ there exists $c_1>0$ and $\epsilon>0$ such that
\beqs
\Re\Big( e^{i (\pi/2-\epsilon)} \big( 1- s^2 -2 i s\big) \Big) \geq c_1 \quad\tfa s\in [0,S].
\eeqs
Therefore, with $S:= \max_{x\in K} \max_j g'_{x_\ell, j}(x_\ell)$, 
\beqs
\Re \Big( e^{i (\pi/2-\epsilon)} \Big(
p^j_{\newtheta}(x, \xi) + \coeffc^{-2}(x)\Big) \Big)
\geq c_1 
\sum_{\ell=1}^d \xi_\ell^2 \frac{1}{\big( 1+ g_{\ell,j}'(x_\ell)^2\big)^2}
\eeqs
from which Point \ref{it:symb_ell_infnt} follows. 
\end{proof}

Recall from \S\ref{sec:pseudo} that $T^*\Rea^d$ can be informally understood as $\{ (x,\xi) : x\in \Rea^d, \xi\in \Rea^d\}$. Similarly, 
$T^* D$ can be informally understood as $\{ (x,\xi) : x\in D, \xi\in \Rea^d\}$.

\begin{lemma}[Directions of ellipticity]\label{lem:ell_unidri} 
$$
\big\{ p^j_{\newtheta} = 0 \big\} \cap T^* \widetilde{\mathcal L}_j \subset \Big\{ (x, \xi):\,\, \xi \parallel \mathfrak e  \,\, \tfa \mathfrak e \in \mathcal Z^j\text{ s.t. } x \in \widetilde{\mathfrak e}\Big\}
$$
(i.e., at a point $x$ in the extended PML layer of $\Omega_j$, the symbol $p^j_{\newtheta}$ vanishes only when $\xi$ is parallel to every extended generalised edge of $\Omega_j$ that $x$ belongs to a $\width$-neighbourhood of).
\end{lemma}
\begin{proof}
Let $(x,\xi) \in T^* \widetilde{\mathcal L}_j$ be such that 
$$
(x,\xi) \notin \Big\{ (x', \xi') : \, \,\xi' \parallel \mathfrak e \,\, \tfa\,\, \mathfrak e \in \mathcal Z^j\text{ s.t. } x' \in \widetilde{\mathfrak e}  \Big\}.
$$
Then, there exists an edge ${\mathfrak e} \in \mathcal Z^j$ so that $x \in \widetilde {\mathfrak e}$ and $\xi \not\parallel {\mathfrak e}$. Since $x \in \widetilde {\mathfrak e}$, 
if $m$ is such that the Cartesian-coordinates vector $e_m$ is normal to ${\mathfrak e}$ then 
$g'_{x_m, j}(x_m) > 0$ (i.e., the PML is active in this coordinate direction).
On the other hand, since $\xi \not\parallel {\mathfrak e}$, we have  $\xi_m \neq 0$ for at least one such $m$. Therefore, by Lemma \ref{lem:comp_symbol} (\ref{it:symb_ell_dir}), $|p^j_{\newtheta}(x, \xi)| \geq C \xi_m ^2 g'_{x_m, j}(x_m) >0$.
\end{proof}

\section{A propagation result for the PML solutions and its consequences} \label{s:sec_prop}

The main purpose of this section is to prove the following propagation result (stated informally in \S\ref{sec:powers}, with the main ideas behind the proof described in \S\ref{sec:sketch_prop}). 

\begin{lemma} \label{lem:key_prop} 

Let $1 \leq j \leq N$, $f \in H^{-1}(\Omega_j)$, $g \in H^{1/2}(\partial \Omega_j)$, $\eta >0$, $\widetilde \Omega_j$ the extension defined by Definition \ref{def:ext}, and $\psi\in C^\infty_c(\widetilde \Omega_j)$. We assume that 
\begin{enumerate}
\item \label{it:kp1}  any trajectory from $T^* (\operatorname{supp} \psi\cap \Omega_j)\cap \big\{ p_\newtheta^j =0\big\}$ goes to infinity under the backward flow.
\item \label{it:kp2} $g=0$ on $\partial\Omega_j \cap \partial\Omega$.
\end{enumerate}
Let $w \in H^1( \Omega_j)$ be the solution to
$$
\begin{cases}
P_{\newtheta}^j w = f \hspace{0.3cm} \text{in }{\Omega_j}, \\
w = g \hspace{0.3cm} \text{on }\partial\Omega_j,
\end{cases}
$$
and assume that $f$ and $w$ are tempered (in the sense of Definition \ref{def:tempered}). Then, with $\widetilde w := S_{j}w$, $\widetilde f := S_{j}f$, 
$$
\operatorname{WF}_\hbar (\psi\widetilde w) \subset \bigcup_{\chi \in C^\infty_c(\widetilde \Omega_j)}
\Big\{ \rho \in T^* \big(\operatorname{supp} \psi\big)\,: \,\exists t \leq 0 \,\,\text{s.t.}\,\,\gamma_{[t, 0]}(\rho) \subset T^* \widetilde \Omega_j \text{ and }  \Phi_t(\rho) \in \operatorname{WF}_\hbar (\chi \widetilde f) \Big\}.
$$
\end{lemma}

The interpretation of Lemma \ref{lem:key_prop} is that, in the ``measurement location" of $\supp \psi$, the wavefront set of the (extended) solution consists at most of points $\rho =(x,\xi)$ coming from the wavefront set of the (extended) source $f$ under the forward flow; see Figure \ref{drw:key_prop}. 
The assumption in Point \ref{it:kp1} is a non-trapping assumption. 

Lemma \ref{lem:key_prop} is proved in \S\ref{ss:proof_key} below, with \S\ref{sec:forward} and \S\ref{sec:escape} containing intermediate results.
\S\ref{ss:res} and \S\ref{ss:outgoing_approx} contain other consequences of the propagation results of this section, namely 
Lemma \ref{lem:res_est} (a priori bound on the PML solution operator) and Theorem \ref{thm:outgoing_approx} (error incurred by Cartesian PML approximation of outgoing Helmholtz solutions), respectively.

\subsection{Forward propagation of regularity}\label{sec:forward}

Recall that $\Phi_t$ is the Hamilton flow for $P$ defined by \eqref{eq:Hamilton}.
Let $\Phi^\theta_t$ denote the Hamilton flow for $P_{\newtheta}$, that is, defined by \eqref{eq:Hamilton} with $p$ replaced by $\operatorname{Re}p_{\newtheta}$. For any $1\leq j \leq N$, let $\Phi_t^{\theta, j}$ denote the Hamilton flow for  $P^j_{\newtheta}$, that is, defined by \eqref{eq:Hamilton} with $p$ replaced by $\operatorname{Re} p^j_{\newtheta}$.
Finally, given an interval $I$, let 
\beq\label{eq:gamma_I}
\gamma_I(x,\xi):= \Big\{ \Phi_t(x,\xi), \; t \in I\Big\}.
\eeq 

\begin{lemma} \label{lem:traj_energy_surf}
On $\{ p_{\newtheta}^j = 0\}$, $\Phi_t^{\theta,j} = \Phi_t$ for any $t\in \mathbb R$; 
and on $\{ p_{\newtheta} = 0\}$, $\Phi^{\theta}_t = \Phi_t $ for any $t\in \mathbb R$.
\end{lemma}
\begin{proof}
We prove the first assertion; the proof of the second one is similar.
By the expression for $\operatorname{Re} p^j_{\newtheta}$ in \eqref{eq:imag_symb}, 
\beq \label{eq:FriMoB2}
\partial_{\xi_\ell} \operatorname{Re} p^j_{\newtheta} = 2 \xi_m  \frac{\big(1- g_{\ell,j}'(x_\ell)^2\big)}{ \big(1+ g'_{x_\ell,j}(x_\ell)^2\big)^2},
\eeq
and
\beq \label{eq:FriMoB3}
\partial_{x_\ell} \operatorname{Re} p^j_{\newtheta} = - \xi_\ell^2  \frac{2g_{\ell,j}''(x_\ell) g_{\ell,j}'(x_\ell)}{ \big(1+ g'_{x_\ell,j}(x_\ell)^2\big)^2} - 4 \xi_\ell^2  \frac{g''_{x_\ell,j}(x_\ell) g'_{x_\ell,j}(x_\ell) \big(1- g_{\ell,j}'(x_\ell)^2\big)}{ \big(1+ g'_{x_\ell,j}(x_\ell)^2\big)^3} + \partial_{x_\ell} (\coeffc^{-2}).
\eeq
Let $(x,\xi)$ be such that $p_{\theta, j}(x, \xi) = 0$. Since also $\operatorname{Im}p_{\theta, j}(x, \xi) = 0$, by \eqref{eq:imag_symb}, 
$$
\tfa \ell, \quad \text{ either } \xi_\ell = 0 \text{ or } g'_{x_\ell,j}(x_\ell) = 0.
$$
In both these two cases, \eqref{eq:FriMoB2} and \eqref{eq:FriMoB3} above simplify to
$$
\partial_{\xi_\ell} \operatorname{Re} p^j_{\newtheta}(x, \xi) = 2 \xi_m, \quad  \partial_{x_\ell}\operatorname{Re} p^j_{\newtheta} (x, \xi)  = \partial_{x_\ell} (\coeffc^{-2}),
$$
and these are the equation defining the flow associated with $p$ (by \eqref{eq:Hamilton}).
\end{proof}

The following lemma uses the notion of wavefront set defined in Definition \ref{def:WF} below.

\begin{lemma}[Forward propagation of regularity from an elliptic point] \label{lem:FPR} 
Let ${v_\hbar}$ be a family of $\hbar$-tempered distributions in the sense of Definition \ref{def:tempered}. 
Assume there exists 
 $(x_0, \xi_0) \in T^* \mathbb R^d$, $t_0\leq 0$, and $(s, M) \in [0, +\infty] \times [0, +\infty]$ such that 
 \beq\label{eq:FPRass}
\Phi_{t_0}(x_0,\xi_0) \in \big\{ (x,\xi) \,:\, |p^j_{\newtheta}(x,\xi)|>0\big\} \quad
\tand
\quad
\gamma_{[t_0,0]}(x_0, \xi_0) \cap \operatorname{WF}^{s-1, M}_\hbar (\hbar^{-1}P^{j}_{\newtheta}  v) = \emptyset.
\eeq
Then $(x_0, \xi_0) \notin \operatorname{WF}^{s,M}_\hbar ( v)$.
The same property holds with $P^{j}_{\newtheta}$ replaced with $P_{\newtheta}$. 
\end{lemma}

\begin{proof}
We do the proof for $P^{j}_{\newtheta}$, the proof for $P_{\newtheta}$ is similar. 
The objects involved in the proof are illustrated in Figure \ref{drw:6}. 
We can assume that $p^j_{\newtheta}(x_0, \xi_0) = 0$, otherwise the result follows directly from Lemma \ref{lem:ell_wf}. 

The idea of the proof is the following:~we seek to use forward propagation of regularity given by Lemma \ref{lem:FPR_app}. 
By the second assumption in \eqref{eq:FPRass}, the trajectory under the flow $\Phi_t$ (illustrated in red in Figure \ref{drw:6}) doesn't intersect $\operatorname{WF}^{s-1, M}_\hbar (\hbar^{-1}P^{j}_{\newtheta}  v)$ until at least $t_0$. However, to use 
Lemma \ref{lem:FPR_app}, we need that the trajectory under the flow $\Phi^{\newtheta,j}_t$
doesn't intersect $\operatorname{WF}^{s-1, M}_\hbar (\hbar^{-1}P^{j}_{\newtheta}  v)$, and this flow is not necessarily equal to $\Phi_t$ once $p^j_\newtheta \neq 0$. A solution to this issue is the following. When travelling backwards from $(x_0,\xi_0)$ under $\Phi_t$ (which starts off equal to $\Phi^{\newtheta,j}_t$ by Lemma \ref{lem:traj_energy_surf} and the fact that $p^j_{\newtheta}(x_0, \xi_0) = 0$), 
by continuity, the two flows must be close to each other in a short time interval after $p^j_\newtheta$ stops being zero, and $p^j_\newtheta >0$ in this time interval. 
That is, we know that under the backward $\Phi^{\newtheta,j}_t$ flow, $(x_0,\xi_0)$ reaches an elliptic point of $p^j_\newtheta$ without hitting the data, and the result then follows from Lemma \ref{lem:FPR_app}.

In more detail:~by Lemma \ref{lem:traj_energy_surf},
$$
\sup \big\{t\leq 0, \; |p^j_{\newtheta}(\Phi_{t}(x_0,\xi_0))|>0 \big\} =
\sup \big\{t\leq 0, \; |p^j_{\newtheta}(\Phi^{\newtheta,j}_{t}(x_0,\xi_0))|>0 \big\}=: \bar t
$$
(since the two flows are the same when $p^j_{\newtheta}=0$). 
By assumption (i.e., the existence of $t_0$), $\bar t > - \infty$. By Lemma \ref{lem:traj_energy_surf},
\beq \label{eq:new_fpr1}
\gamma_{[\bar t, 0]}(x_0,\xi_0) = \gamma^{\newtheta,j}_{[\bar t, 0]}(x_0,\xi_0),
\eeq
(where $\gamma^{\newtheta,j}_{[\bar t, 0]}(x_0,\xi_0)$ is defined analogously to (\ref{eq:gamma_I}) with $\Phi_t$ replaced by $\Phi^{\newtheta,j}_t$).
By the definition of $\bar t$, there exists $T < \bar t$ such that
\beq \label{eq:new_fpr2}
|p^j_{\newtheta}(\Phi^{\newtheta,j}_{t}(x_0,\xi_0))| > 0 \quad \tfa t \in ( T, \bar t).
\eeq
By (\ref{eq:new_fpr1}) and the second assumption in \eqref{eq:FPRass}, since $t_0 \leq \bar t$ and $ \operatorname{WF}^{s-1, M}_\hbar (\hbar^{-1}P^{j}_{\newtheta}  v)$ is a closed set, 
$$
d\Big(\gamma^{\newtheta,j}_{[\bar t, 0]}(x_0,\xi_0) ), \operatorname{WF}^{s-1, M}_\hbar (\hbar^{-1}P^{j}_{\newtheta}  v) \Big) > 0.
$$
Hence, since $t \mapsto \Phi^{\newtheta,j}_t(x_0,\xi_0)$ is continuous, 
there exists $T<  t_1 < \bar t$ close enough to $\bar t$ so that
\beq \label{eq:new_fpr3}
\gamma^{\newtheta,j}_{[ t_1,0]}(x_0, \xi_0) \cap \operatorname{WF}^{s-1, M}_\hbar (\hbar^{-1}P^{j}_{\newtheta}  v) = \emptyset.
\eeq
In particular, (\ref{eq:new_fpr2}) and (\ref{eq:new_fpr3}) imply by Lemma \ref{lem:ell_wf} that
\beq \label{eq:new_fpr4}
\Phi^{\newtheta,j}_{ t_1} (x_0, \xi_0) \notin \operatorname{WF}^{s,M}_\hbar ( v).
\eeq
Forward propagation of regularity given by Lemma \ref{lem:FPR_app} (the sign condition on $\operatorname{Im} p_{\newtheta}^j$ being satisfied thanks to the first equation in (\ref{eq:imag_symb})) allows us to conclude, by (\ref{eq:new_fpr4}) and (\ref{eq:new_fpr3}), that $ (x_0, \xi_0) \notin \operatorname{WF}^{s,M}_\hbar ( v)$.
\end{proof}

\begin{figure}
\begin{center}
\includegraphics{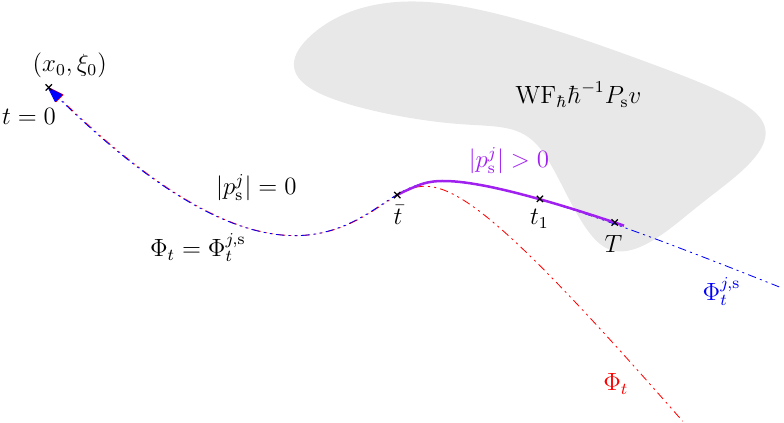}
\caption{The objects involved in the proof of Lemma \ref{lem:FPR}.}
\label{drw:6}
\end{center}
\end{figure}

\subsection{Escape to ellipticity}\label{sec:escape}

\begin{lemma}[Escape to ellipticity]\label{lem:go_to_elliptic}
For any $\rho \in T^*\widetilde{\Omega}_j$ such that the backward trajectory from $\rho$ in the flow $\Phi_t$ goes to infinity (i.e. $\rho$ is not trapped backward in time), there exists $\tau \leq 0$ such that
$$
\gamma_{[\tau, 0]}(\rho) \subset T^*\widetilde{\Omega}_j, \quad \text{ and } \quad  \Phi_{\tau}(\rho) \in \big\{ (x,\xi) \,:\, |p_{\newtheta}^j(x,\xi)|>0 \big\} .
$$
\end{lemma}

Although Lemma \ref{lem:go_to_elliptic} is stated (and then used below) with the flow $\Phi_t$ associated to $P$, we note that the result holds for a general continuous flow.

\bpf[Proof of Lemma \ref{lem:go_to_elliptic}]
By Lemma \ref{lem:ell_unidri}, it suffices to show the existence of $\tau \leq 0$ such that 
$$
\gamma_{[\tau, 0]}(\rho) \subset T^*\widetilde{\Omega}_j \quad \text{ and } \quad  \Phi_{\tau}(\rho) \in  T^* \widetilde{\mathcal L}_j \backslash 
\Big\{ (x, \xi)\,:\, \xi \parallel \mathfrak e 
\tfa \mathfrak e \in \mathcal Z^j\text{ s.t. }
x \in \widetilde{\mathfrak e}   \Big\}.
$$

For simplicity we give the proof in dimension 3; we then indicate how to generalise it to any dimension. 

By the non-trapping assumption, there exists $t_0 \geq 0$ such that the backward trajectory from $\rho$ enters $T^* \widetilde{\mathcal L}^j$ forever  after time $t_0$; i.e., $\gamma_{(-\infty, t_0]}(\rho) \subset T^* \widetilde{\mathcal L}^j$.
Reducing $t_0$ if necessary, we have $\gamma_{[t_0,0]} (\rho) \subset T^* \widetilde \Omega^j$. 

Observe that the trajectory from $\rho$ is in an extended PML-face for all times before time $t_0$, i.e. 
$\gamma_{(-\infty, t_0]} (\rho) \in T^* \widetilde { {\mathfrak e}_1}$ for some ${\mathfrak e}_1 \in {\mathcal X}^j_2$. If there exists $\tau\leq t_0$ such that $\pi_\xi \Phi_{\tau} (\rho) \not\parallel {\mathfrak e}_1 $ and $\gamma_{[\tau, t_0]}(\rho) \subset T^* \widetilde{\Omega}^j$, then we are done. Otherwise, 
the backward trajectory from $\Phi_{t_0} (\rho)$ stays parallel to the face ${\mathfrak e}_1$ it entered before leaving $\widetilde \Omega_j$, and therefore there exists $t_1\leq t_0$ such that it enters an extended PML-edge forever; i.e.,~$\gamma_{(-\infty, t_1]}(\rho) \subset T^* \widetilde {{\mathfrak e}_2}$ for some ${\mathfrak e}_2 \in \mathcal X^j_1$. Reducing $t_1$ if necessary, we have again $\gamma_{[t_1, t_0]}(\rho )\subset T^* \widetilde{\Omega}^j$.

If there exists $\tau\leq t_1$ such that $\pi_\xi \Phi_{\tau} (\rho) \not\parallel {\mathfrak e}_2$ and $\gamma_{[\tau, t_1]}(\rho) \subset T^* \widetilde{\Omega}^j$, then we are done. Otherwise, 
the backward trajectory from $\Phi_{t_1} (\rho)$ stays parallel to the edge ${\mathfrak e}_2$ it entered before leaving $\widetilde \Omega_j$, and therefore there is $\tau\leq t_1$ with $\gamma_{[\tau, t_1]}(\rho) \subset T^* \widetilde{\Omega}_j$ such that it enters a perpendicular PML-edge (because $\Omega_j$ is a  hyperrectangle). We therefore have $\pi_x \Phi_{\tau} (\rho) \in \widetilde {\mathfrak e}_3$ and $\pi_\xi \Phi_{\tau} (\rho) \not\parallel {\mathfrak e}_3$ for some ${\mathfrak e}_3 \in \mathcal X_1^j$ and we are done.

The proof generalises to any dimension $d\geq 2$ by observing that, if $\gamma_{[T,0)} (\rho) \subset T^* \widetilde{\Omega_j}$ and $\gamma_{(-\infty, T]} (\rho) \subset T^* \widetilde{\mathfrak e}$ for some $\mathfrak e \in {\mathcal X}^j_\ell$ with $2 \leq \ell \leq d-1$, then \emph{either} there exists $\tau\leq T$ such that $\pi_\xi \Phi_{\tau} (\rho) \not\parallel {\mathfrak e}$ and $\gamma_{(\tau, T]} \subset T^* \widetilde{\Omega}^j$, \emph{or} the corresponding trajectory enters a $\ell-1$ dimensional edge forever:~i.e., there exists $T' \leq T$ such that $\gamma_{[T', T]} (\rho) \subset T^* \widetilde{\Omega_j}$ and $\gamma_{(-\infty, T']} (\rho) \subset T^* \widetilde{\mathfrak e'}$ with $\mathfrak e' \in {\mathcal X}^j_{\ell-1}$. The initialisation of this process (entering $T^* \widetilde{\mathcal L}^j$)  and the end-case ($\ell=1$) are the same as in dimension $3$.
\epf

\subsection{Proof of Lemma \ref{lem:key_prop}} \label{ss:proof_key}

\begin{figure}[h]
\begin{center}
\includegraphics[scale=0.65]{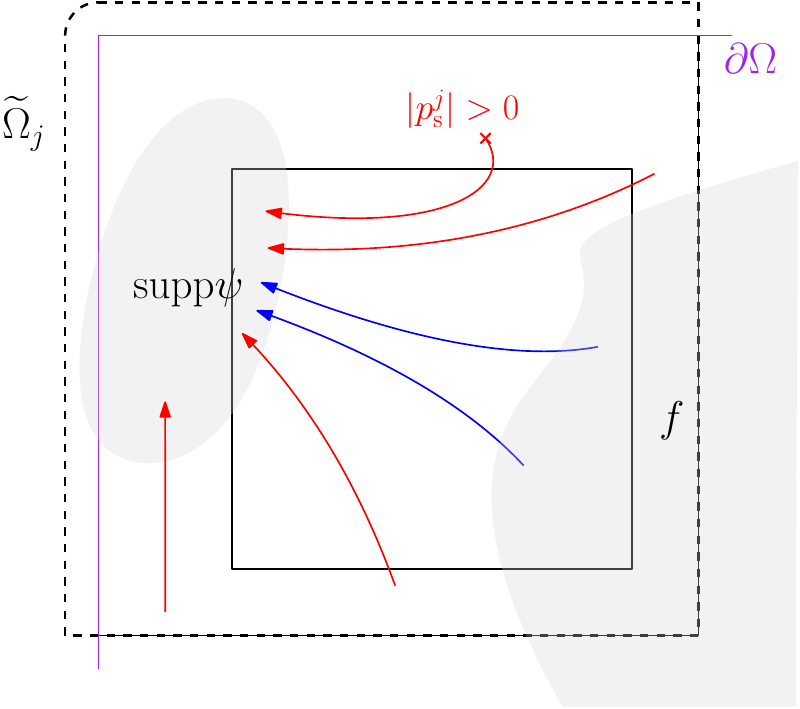}
\caption{Idea behind the proof of the propagation Lemma \ref{lem:key_prop} illustrated for a subdomain $\Omega_j$ belonging to a corner of $\Omega$, extended as $\widetilde{\Omega}_j$ (dashed). Any point of $T^*\operatorname{supp} \psi$ is the end-point of a trajectory coming from an elliptic point in the PML layer. If such a trajectory doesn't intersect $\operatorname{WF}_\hbar f$, it carries no mass by propagation of singularities (red trajectories). Therefore, the solution in the observation region, $\psi v$, has only wavefront corresponding to trajectories coming from $\operatorname{WF}_\hbar f$ (blue trajectories).}\label{drw:key_prop}
\end{center}
\end{figure}

By the assumption in Point \ref{it:kp2} and Lemma \ref{lem:ext_h},
$$
P_{\newtheta}^j \widetilde w = \widetilde f \hspace{0.3cm} \text{in }{\widetilde \Omega_j}.
$$
Let $\rho \in T^*\operatorname{supp} \widetilde \chi$ be such that
$$
\rho \notin \bigcup_{\chi \in C^\infty_c(\widetilde \Omega_j)} \Big\{\exists t <0, \; \gamma_{[t, 0]}(\rho) \subset T^* \widetilde \Omega_j \text{ and }  \Phi_t(\rho) \in \operatorname{WF}_\hbar ( \chi \widetilde f) \Big\};
$$
in other words, the backwards trajectory  for the Hamilton flow of $P$ starting at $\rho$ \emph{either} doesn't intersect $\operatorname{WF}_\hbar ( \chi \widetilde f)$ for any $\chi \in C^\infty_c(\widetilde \Omega_j)$ \emph{or} leaves $\widetilde \Omega_j$ before intersecting such a subset. 
By Lemma \ref{lem:go_to_elliptic}, there exists $\tau \leq 0$ such that
\beq \label{eq:29ev1}
\gamma_{[\tau, 0]}(\rho) \subset T^*\widetilde{\Omega}_j, \quad \text{ and } \quad  \Phi_{\tau}(\rho) \in \big\{ (x,\xi) \,:\,|p_{\newtheta}^j(x,\xi)| > 0\big\}
\eeq
(where we take $\tau=0$ if $\rho \in \{ p_\newtheta^j=0\}$).
From the first part of (\ref{eq:29ev1}), by assumption on $\rho$, we get that for any $\chi \in C^\infty_c(\widetilde \Omega_j)$
\beq \label{eq:29ev2}
\gamma_{[\tau, 0]}(\rho) \cap \operatorname{WF}_\hbar ( \chi \widetilde f) = \emptyset.
\eeq
We conclude from (\ref{eq:29ev1}) and (\ref{eq:29ev2}), by forward propagation of regularity from an elliptic point, given by Lemma \ref{lem:FPR},  that $\rho \notin \operatorname{WF}_\hbar(\psi\widetilde w)$.

\subsection{A priori bound on the local PML solution on the subdomains} \label{ss:res}

\red{The results in this section use the weighted Sobolev norms defined in \S\ref{sec:SC1}; just as in \eqref{eq:weighted_norm} these norms weight each derivative with $k^{-1}=\hsc$.}

\begin{lemma}[Bound on PML solution operator]\label{lem:res_est}
Assume that the flow associated with $P$ is nontrapping. 
There exists $C >0$ and $\hbar_0 >0$ such that the following is true. Given $f \in H^{-1}(\Omega_j)$ and $g \in H^{1/2}(\partial\Omega_j)$, if $u$ satisfies
\begin{equation}\label{eq:PMLbounded}
\begin{cases}
P^j_{\newtheta}u = f \text{ in } \Omega_j, \\
u = g \text{ on } \partial \Omega_j,
\end{cases}
\end{equation}
then if $0<\hbar<\hbar_0$ then
\beq\label{eq:resolvent}
\Vert u \Vert_{H^1_\hbar(\Omega_j)} \leq C  \Big(\hbar^{-1} \Vert f \Vert_{
H^{-1}_\hbar(\Omega_j)}
 + \hbar^{-3/2} \Vert g \Vert_{H_\hbar^{1/2}(\partial\Omega_j)} \Big).
\eeq
The same holds with $\Omega_j$ replaced with $\Omega$.
\end{lemma}

\red{Since the PML problem is Fredholm, Lemma \ref{lem:res_est} shows that the solution of the PML problem on $\Omega_j$ exists for $\hbar$ sufficiently small (i.e., for $k$ sufficiently large) if the flow for $P$ is nontrapping.}

We emphasise that the $\hbar^{-3/2}$ in front of the norm of $g$ on the right-hand side of \eqref{eq:resolvent} is not sharp. However,
the only way in which \eqref{eq:resolvent} is used in the proofs of Theorems \ref{thm:strip}-\ref{thm:gen} is to show that the solution to \eqref{eq:PMLbounded} is tempered (in the sense of Definition \ref{def:tempered}) if $f$ and $g$ are tempered, and so this non-sharpness is not important for the results of the paper.

\begin{proof}
We first claim that it it sufficient to prove that
there exists $C,\hbar_0>0$, such that 
 if $v \in H^1_0(\Omega)$ satisfies $P_\newtheta^j v= f\in H^{-1}(\Omega)$ then, if $0<\hbar<\hbar_0$, 
\beq\label{eq:resolvent2}
\Vert v \Vert_{H^1_\hbar(\Omega_j)} \leq C \hbar^{-1} \Vert f \Vert_{
H^{-1}_\hbar(\Omega_j)}.
\eeq
Indeed, by \cite[Theorem 5.6.4]{Ne:01}, for any Lipschitz domain $D$ 
there exists $E_D: H^{1/2}(\partial D)\to H^1(D)$ such that $\gamma E_D \phi=\phi$ for all $\phi \in H^{1/2}(\partial D)$, where $\gamma: 
H^1(D) \to H^{1/2}(\partial D)$ is the trace operator. Furthermore, 
\beq\label{eq:Nedelec}
\|E_D\|_{H^{1/2}_\hbar(\partial D) \to H^1_\hsc(D)}\leq C \hbar^{-1/2},
\eeq
noting that the weighted $H^{1/2}$ norm on $\partial D$ used in \cite{Ne:01} is equivalent to $h^{-1/2}\| \cdot\|_{H^{1/2}_\hsc(\partial D)}$ and the weighted $H^1$ norm on $D$ used in \cite{Ne:01} is equivalent to $h^{-1}\| \cdot\|_{H^{1}_\hsc(D)}$.
Therefore, with $u$ the solution to \eqref{eq:PMLbounded}, let $v:= u - E_D g \in H^1_0(\Omega_j)$. Then, by \eqref{eq:resolvent2},
\begin{align*}
\N{u}_{H^1_\hbar(\Omega_j)} 
&\leq \N{v}_{H^1_\hbar(\Omega_j)} + \N{E_D g}_{H^1_\hbar(\Omega_j)}\\
&\leq C \hbar^{-1} \Big( \N{f}_{H^{-1}_\hbar(\Omega_j)} + \N{P_\newtheta^j E_D g}_{H^{-1}_\hbar(\Omega_j)}\Big) +\N{E_D g}_{H^1_\hbar(\Omega_j)}\\
&\leq C \hbar^{-1} \Big( \N{f}_{H^{-1}_\hbar(\Omega_j)} + C\N{E_D g}_{H^{1}_\hbar(\Omega_j)} \Big)+\N{E_D g}_{H^1_\hbar(\Omega_j)}
\end{align*}
and the bound \eqref{eq:resolvent} then follows from \eqref{eq:Nedelec}.

We now prove the bound \eqref{eq:resolvent2}.
Seeking a contradiction, we assume that \eqref{eq:resolvent2} does not hold. Then, there exists $\hbar_n \rightarrow 0$, $v_n \in H^1_0(\Omega_j)$, $f_n \in L^2(\Omega_j)$, 
such that
$P^j_{\newtheta}(\hbar_n)v_n = f_n \text{ in } \Omega_j,$
and
$$
\Vert v_n \Vert_{H^1_{\hbar_n}(\Omega_j)} \geq n  
\hbar_n^{-1} \Vert f_n \Vert_{
(H_{\hbar_n}^{1}(\Omega_j))^*}.
$$
Renormalising, we can assume that
\begin{equation}\label{eq:res:ren}
\Vert v_n \Vert_{H^1_{\hbar_n}(\Omega_j)} = 1.
\end{equation}
Then
\begin{equation}\label{eq:res_dec}
 \hbar_n^{-1} \Vert f_n \Vert_{ (H_{\hbar_n}^{1}(\Omega_j))^*}
 \rightarrow 0.
\end{equation}
We now extend $v_n$ and $f_n$ to $\widetilde{\Omega}_j$ using the extension defined by Lemma \ref{def:ext};
i.e., we let $\widetilde v_n := S_j v_n$ and $\widetilde f_n := S_j f_n$. By Lemma \ref{lem:ext_h}, $P_\newtheta^j \widetilde v_n = \widetilde f_n$. 
By \eqref{eq:res:ren} $\widetilde v_n$ is tempered, and so propagation of singularities applies. 
For any $\chi \in C^\infty_c(\widetilde{\Omega}_j)$, 
by \eqref{eq:ext_ver_eq} 
 and the fact that  $S_j: H^{-1}(\Omega_j) \to H^{-1}(\widetilde{\Omega}_j)$ (by Lemma \ref{lem:ext_h}), 
\beqs
\big\|\chi P^j_{\newtheta}(\hbar_n) \widetilde v_n\big\|_{
H^{-1}_\hbar(\widetilde{\Omega}_j)}
=\big\|\chi \widetilde f_n\big\|_{
H^{-1}_\hbar(\widetilde{\Omega}_j)}
\leq C \big\| f_n\big\|_{H^{-1}_\hbar(\Omega_j)}.
\eeqs
Therefore, by (\ref{eq:res_dec}) and the definition of $\operatorname{WF}^{-1,1}$ (Definition \ref{def:WF}), 
\begin{equation}\label{eq:res:dt_osc}
\operatorname{WF}^{-1,1}(\chi P^j_{\newtheta}(\hbar_n) \widetilde v_n) = \emptyset, \hspace{0.5cm} \tfa\,  \chi \in C^\infty_c(\widetilde{\Omega}_j).
\end{equation}
By Lemma \ref{lem:go_to_elliptic} (escape to ellipticity),  Lemma \ref{lem:FPR} (forward propagation of regularity from an elliptic point), and  (\ref{eq:res:dt_osc}),
\beqs
\operatorname{WF}^{0,1}(\chi \widetilde v_n)= \emptyset \hspace{0.5cm} \tfa\,  \chi \in C^\infty_c(\widetilde{\Omega}_j).
\eeqs
Therefore, by Lemma \ref{lem:osci}, Point \ref{it:osci1} together with 
Point \ref{it:symb_ell_infnt} of Lemma \ref{lem:comp_symbol} (ellipticity at infinity) and (\ref{eq:res_dec}), 
\begin{equation*}
\N{\chi \widetilde v_n}_{L^2} \rightarrow 0 \hspace{0.5cm} \tfa\,  \chi \in C^\infty_c(\widetilde{\Omega}_j).
\end{equation*}
Lemma \ref{lem:osci}, Point \ref{it:osci2} together with 
Point \ref{it:symb_ell_infnt} of Lemma \ref{lem:comp_symbol} 
and  (\ref{eq:res_dec}) allow us to strengthen this to 
\begin{equation*}
\N{\chi \widetilde v_n}_{H^1_{\hbar_n}} \rightarrow 0 \hspace{0.5cm} \tfa\,  \chi \in C^\infty_c(\widetilde{\Omega}_j)
\end{equation*}
which contradicts (\ref{eq:res:ren}).
\end{proof}

\subsection{Error incurred by Cartesian PML approximation of outgoing Helmholtz solutions}
\label{ss:outgoing_approx}

\begin{theorem}
\mythmname{The error in Cartesian PML approximation of outgoing Helmholtz solutions $O(k^{-\infty})$}
 \label{thm:outgoing_approx}
Suppose \emph{either} (i) $P$ is nontrapping, \emph{or} (ii) with $P_\newtheta$ considered as an operator either $H^1_0(\Omega) \to H^{-1}(\Omega)$ or $H^1(\Rea^d)\to H^{-1}(\Rea^d)$, its solution operator 
 is bounded polynomially in $\hbar^{-1}$. Then for any $k_0,M,s>0$, there exists $C>0$ such that the following holds for all $k\geq k_0$. 
Let $f \in H^{-1}(\red{\Omega})$ with \red{$\supp f\subset \overline{\Omega_{\rm int}}$}  and let $u$ and $v$ be the solutions to
$$
\begin{cases}
P_{\rm s} u = f \text{ in }\Omega, \\
u = 0  \text{ on }\partial\Omega,
\end{cases}
\hspace{1.5cm}
\begin{cases}
P v = f \text{ in }\mathbb R^d, \\
v \text{ satisfies the Sommerfeld radiation condition}.
\end{cases}
$$
Then, 
$$
\Vert u - v \Vert_{H^s_\hbar(\Omega_{\rm int})}\leq C k^{-M}\|f\|_{H_\hbar^{-1}(\red{\Omega})} \quad\tfa k\geq k_0.
$$
\end{theorem}

\red{In the statement of Theorem \ref{thm:outgoing_approx},
observe that the equation $Pv=f$ in $\Rea^d$ makes sense since 
the fact that $\supp f\subset \overline{\Omega_{\rm int}}$ means that $f$ can be extended to $H^{-1}(\mathbb{R}^d)$ via
$$
\langle f, w\rangle_{H^{-1}(\mathbb{R}^d) \times H^1(\mathbb{R}^d)} :=\langle f, w \chi \rangle_{H^{-1}(\Omega) \times H^1_0(\Omega)} \text{ for } w \in H^1(\mathbb R^d),
$$
where $\chi \in C^\infty_c(\Omega)$ is such that $\chi \equiv 1$ on a neighbourhood of $\Omega_{\rm int}$. Since $\supp f\subset \overline{\Omega_{\rm int}}$, this definition is independent of $\chi$.
}

Theorem \ref{thm:outgoing_approx} is not used in the proofs of Theorems \ref{thm:strip}-\ref{thm:gen}, but we include it in the paper because of its independent interest. Indeed, as noted in \S\ref{sec:context}, the only other $k$-explicit Cartesian-PML convergence result in the literature is that of \cite[Lemma 3.4]{ChXi:13} for the case $c\equiv 1$ (i.e., no scatterer), with this proof based on the explicit expression for the Helmholtz fundamental solution in this case. $k$-explicit convergence results for radial PMLs applied to general Helmholtz problems are in \cite{GLS2}.

\begin{proof}
By dividing $u$ and $v$ by $\|f\|_{H_\hbar^{-1}(\red{\Omega})}$, we can assume, without loss of generality, that $\|f\|_{H_{\hbar}^{-1}(\red{\Omega})}=1$. 
It is then sufficient to prove that 
$$
\Vert u - v \Vert_{H^s_\hbar(\Omega_{\rm int})}\leq C k^{-M}\quad\tfa k\geq k_0.
$$
Let $u_\infty$ be the solution to
$$
P_{\rm s} u_\infty = f \text{ in }\mathbb R^d. 
$$
By \cite[Theorem 4.37]{DyZw:19}, 
\beq\label{eq:agreement}
u_{\infty}|_{\Omega_{\rm int}} = v;
\eeq
indeed, the Cartesian PML fits into the framework of \cite[\S 4.5.1]{DyZw:19} by setting
\beqs
F_\theta(x) = \sum_{\ell=1}^d \int_{\red{0}}^{x_\ell} g_{\ell}(\widetilde{x}_\ell)\, d \widetilde{x}_\ell,
\eeqs
and observing that $(\nabla F_{\red{\theta}}(x))_m = g_{x_m}(x_m)$ and $D^2 F_{\red{\theta}}\geq 0$ in the sense of quadratic forms (i.e., $F_\theta$ is convex) since $g_{x_m}'\geq 0$.
The relation \eqref{eq:agreement} implies that it is sufficient to prove that
\beq\label{eq:STP1}
\tfa s>0, \quad \Vert u - u_\infty \Vert_{H^s_\hbar(\Omega_{\rm int})} = O(\hbar^\infty).
\eeq
We claim that 
\beq\label{eq:claim}
\text{ for any } s>0, \quad \Vert u_\infty \Vert_{H^{s}_\hbar(\partial \Omega)} = O(\hbar^\infty)
\eeq
and postpone the proof for now. 

We now let $u_{\epsilon}\in H^1(\Omega)$ be the solution to
$$
\begin{cases}
P_{\rm s} u_{\epsilon}= 0 \text{ in }\Omega, \\
 u_{\epsilon} = u_{\infty}  \text{ on }\partial\Omega.
\end{cases}
$$
By \eqref{eq:claim} and \emph{either} Lemma \ref{lem:res_est} under Assumption (i), \emph{or} polynomial boundedness of the solution operator on \red{$H^1(\Omega)$} under Assumption (ii), 
\beq\label{eq:peanut1}
\Vert u_{\epsilon} \Vert_{H^1_\hsc(\Omega)} = O(\hbar^\infty).
\eeq
By Lemma \ref{lem:osci}, Part \ref{it:osci2}, the bound (\ref{eq:peanut1}) can be improved to
\beq\label{eq:peanut1bis}
\tfa s >0\text{ and }\chi \in C^\infty_c(\Omega),\quad  \Vert \chi u_{\epsilon} \Vert_{H^s_\hsc(\Omega)} = O(\hbar^\infty).
\eeq
The whole point of the definition of $u_\epsilon$ is that, 
by uniqueness of $u$,
$$
u = u_{\infty} + u_{\epsilon}.
$$
Thus \eqref{eq:STP1}, and hence also the result, follows from \eqref{eq:peanut1bis}.

We now prove \eqref{eq:claim}. 
Under Assumption (ii), $u_\infty$ is immediately tempered in the sense of Definition \ref{def:tempered}. Under Assumption (i), 
repeating the proof of Lemma \ref{lem:res_est} (without the extensions), we see that $\chi u_\infty$ is polynomially bounded in $\hbar^{-1}$ in terms of the data $f$ for any $\chi \in C^\infty_c(\Rea^d)$, and hence $u_\infty$ is tempered.  

Let $\psi \in C^\infty_c(\mathbb R^d)$ be such that $\psi \equiv 1$ near $\partial \Omega$ and
$\operatorname{supp} \psi \subset \mathbb R^d \backslash \Omega_{\rm int}$.
We now seek to apply Lemma \ref{lem:key_prop}. 
By Lemma 
\ref{lem:ell_unidri}, any $\rho \in T^*(\supp \psi) \cap \{ p_\newtheta=0\}$ is perpendicular to the scaling direction.
Therefore, since $c\equiv 1$ in the PML region, any such $\rho$ goes to infinity backwards in time under the flow associated to $P$. Therefore Point \ref{it:kp1} in Lemma \ref{lem:key_prop} is satisfied, and the result of this lemma is that
$$
\operatorname{WF}_\hbar ({\psi}u_\infty) \subset 
\Big\{ \rho \in T^*\operatorname{supp}\psi, \hspace{0.3cm}\exists t<0, \; \Phi_t(\rho) \in \operatorname{WF}_\hbar (f) \Big\}.
$$
By similar arguments, Lemma \ref{lem:ell_unidri} implies that
$$
\Big\{ \rho \in T^*\operatorname{supp}\psi, \hspace{0.3cm}\exists t<0, \; \Phi_t(\rho) \in \operatorname{WF}_\hbar (  f) \Big\} \subset \Big\{ |p_{\rm s}| > 0 \Big\}.
$$
By  Lemma \ref{lem:ell_wf}, 
$$
\operatorname{WF}_\hbar ({\psi}u_\infty) = \emptyset.
$$
By Lemma \ref{lem:osci}, Part \ref{it:osci1}, and Lemma \ref{lem:comp_symbol}, Part \ref{it:symb_ell_infnt}, for any $s>0$, 
\beqs
\N{\psi u_\infty}_{H^s_\hbar} = O(\hbar^\infty)
\eeqs
(where Lemma \ref{lem:osci} is applied with $\supp\widetilde{\chi} \subset \Rea^d\setminus \Omega_{\rm int}$, so that $\widetilde{\chi} P_\newtheta u_\infty= 
\widetilde \chi f =0$).
The claim \eqref{eq:claim} then follows from the trace theorem and the fact that $\psi \equiv 1$ near $\partial \Omega$, and the proof is complete.
\end{proof}

\section{The propagation of physical error}

\subsection{The propagation of physical error for the parallel method}\label{sec:5error}

\begin{proof}[Proof of Lemma \ref{lem:model_err_prop}]
By the definition of $\mathcal T_{i}$ \eqref{eq:Tjl}, $\mathcal{T}_{i} v =0$ on $\partial \Omega_i \cap\partial\Omega$ for any $v \in H^1_0(\Omega)$. Thus, since $\supp \chi_i
\cap \Omega
 \subset \Omega_i$ (by \eqref{eq:PoU}),
$\chi_i\mathcal{T}_{i} v =0$ on $\partial\Omega_i\setminus \partial\Omega$ and so $\mathbf T$ defined by \eqref{eq:bold_T} maps $(H^1_0(\Omega))^N \mapsto  (H^1_0(\Omega))^N$.
Now, in $\Omega_j$, by the definition of $e^{n+1}_j$ \eqref{eq:errors} and the iteration \eqref{eq:localprob},
\begin{align*}
P^j_{\newtheta} \parallelejnpo =
P_\newtheta^j \big( u|_{\Omega_j} - u_j^{n+1}\big)
&= P_\newtheta^j( u|_{\Omega_j}) - P^j_\newtheta(\parallelu^n|_{\Omega_j}) + (P_\newtheta \parallelu^n)|_{\Omega_j} - f|_{\Omega_j}\\
&= P_\newtheta^j \big( (u-\parallelu^n)|_{\Omega_j}\big) - \big( P_\newtheta(u-\parallelu^n)\big)|_{\Omega_j} \\
&= P_\newtheta^j \big( \parallele^n|_{\Omega_j}\big) - \big( P_\newtheta \parallele^n\big)|_{\Omega_j}.
\end{align*}
Therefore, by \eqref{eq:error_decomp}, 
\beqs
P^j_{\newtheta} \parallelejnpo = \sum_{\ell=1}^N
\Big(P_\newtheta^j \big( (\chi_\ell \paralleleelln)|_{\Omega_j}\big) - \big( P_\newtheta (\chi_\ell \paralleleelln)\big)|_{\Omega_j}\Big).
\eeqs
Furthermore, by \eqref{eq:localprob}, $u_j^{n+1}=\parallelu^n$ on $\partial \Omega_j$, and thus, on $\partial\Omega_j$,
 $$
\parallelejnpo = u - u_j^{n+1} = u - \parallelu^n =  \parallele^n=\sum_{\ell=1}^N \chi_\ell  \paralleleelln.
 $$
Therefore, by the definition of $\mathcal T_{j}$ \eqref{eq:Tjl}, and the fact (noted after \eqref{eq:local_phys_error}) that $\chi_\ell \paralleleelln\in H^1_0(\Omega)$,
 $$
\parallelejnpo 
  = \sum_{\ell=1}^N \mathcal T_{j}  \chi_\ell \paralleleelln.
 $$
Therefore, since $\Tchi_\ell \chi_\ell=\chi_\ell$, 
$$
\chi_j  \parallelejnpo 
    = \sum_{\ell=1}^N \Big( \chi_j \mathcal T_{j}  \Tchi_\ell \Big)\chi_\ell \paralleleelln
 $$
and the result follows by definitions of $\parallelepsilonjn$ \eqref{eq:local_phys_error} and $\mathbf T$ \eqref{eq:bold_T}.
\end{proof}

\bre[The rationale for putting $\Tchi_j$ in the definition of $\mathbf T$]\label{rem:Tchi}
The end of the proof of Lemma \ref{lem:model_err_prop} shows that if 
$\Tchi_j$ is omitted from the definition of $\mathbf T$, then the result \eqref{eq:Monday1} still holds. The advantage of including 
$\Tchi_j$ is that this builds into the operator the localisation property of the physical error. 
Indeed, we do not expect the power contractivity proved for $\mathbf T$ (which is a result about $\mathbf T$ applied to arbitrary functions in $(H^1_0(\Omega))^N$) to hold if $\Tchi_j$ is omitted from the definition of $\mathbf T$.
\ere

\subsection{The propagation of physical error for the sequential method}\label{sec:error_sweeping}

\begin{proof}[Proof of Lemma \ref{lem:model_seq_err_prop}]
We prove  the first equation in \eqref{eq:sweep_fwbw_error}; the proof of the second equation is very similar.
To keep the notation concise, in the proof we omit the restrictions onto $\Omega_j$ so that, in particular, the PDE defining $u_j^{2n+1}$ in \eqref{eq:forward_sweeping} becomes 
\beqs
P_{\newtheta}^j
u_j^{2n+1} = 
(P_{\newtheta}^j-P_{\newtheta})
(u_{j,n}^{\rightarrow}) + f \hspace{0.5cm}\text{in } \Omega_j. 
\eeqs
By the definition of $e_{\times,j}^{2n+1}$ \eqref{eq:errors_mult} and the definition  of $u^{2n+1}_j$ \eqref{eq:forward_sweeping}, in $\Omega_j$,
\begin{align*}
P_{\newtheta}^j
 e_{\times,j}^{2n+1} &= 
P_{\newtheta}^j u -  P_{\newtheta}^j u^{2n+1}_j =
(P_{\newtheta}^j - P_{\newtheta}) u + f - (P_{\newtheta}^j-P_{\newtheta})
(u_{j,n}^{\rightarrow}) - f \\
&= (P_{\newtheta}^j -  P_\newtheta) \Big(u - \sum_{\ell<j} \chi_\ell u_\ell^{2n+1} - \sum_{j\leq \ell} \chi_\ell u_\ell^{2n}\Big) \\
&= (P_{\newtheta}^j -  P_\newtheta) \Big(\sum_{\ell=1}^N \chi_\ell u - \sum_{\ell<j} \chi_\ell u_\ell^{2n+1} - \sum_{j\leq \ell} \chi_\ell u_\ell^{2n}\Big) \\
&= (P_{\newtheta}^j -  P_\newtheta) \Big(\sum_{\ell<j} \chi_\ell ( u - u_\ell^{2n+1}) +\sum_{j\leq \ell} \chi_\ell ( u - u_\ell^{2n} ) \Big) \\
&= (P_{\newtheta}^j -  P_\newtheta) \Big(\sum_{\ell<j} \chi_\ell e_{{\times}, \ell}^{2n+1} +\sum_{j\leq \ell} \chi_\ell e_{{\times}, \ell}^{2n} \Big).
\end{align*}
Similarly, 
on $\partial\Omega_j$,
\begin{align*}
 e_{\times,j}^{2n+1} & = u -  u^{2n+1}_j = u - \sum_{\ell<j} \chi_\ell u_\ell^{2n+1} - \sum_{j\leq \ell} \chi_\ell u_\ell^{2n} = \sum_{\ell<j} \chi_\ell e_{{\times}, \ell}^{2n+1} +\sum_{j\leq \ell} \chi_\ell e_{{\times}, \ell}^{2n},
\end{align*}
where,  in the last equality, we used the same manipulation as in the previous displayed equation. Therefore, by definition of ${\mathcal T}_{j}$ \eqref{eq:Tjl} and the definition of $\Tchi_j$ in \S\ref{sec:idea_error},
\begin{align*}
 e_{\times,j}^{2n+1} &= \sum_{\ell<j}  {\mathcal T}_{j} \chi_\ell e_{{\times}, \ell}^{2n+1} +\sum_{j\leq \ell} {\mathcal T}_{j} \chi_\ell e_{{\times}, \ell}^{2n} = \sum_{\ell<j}  {\mathcal T}_{j} {\widetilde \chi}_\ell \chi_\ell e_{{\times}, \ell}^{2n+1} +\sum_{j\leq \ell} {\mathcal T}_{j} {\widetilde \chi}_\ell \chi_\ell e_{{\times}, \ell}^{2n}.
\end{align*}
Therefore 
\begin{align*}
({\boldsymbol{\epsilon}_{\times}^{2n+1}})_j =  \chi_j  e_{\times,j}^{2n+1}
&= \sum_{\ell<j}  (\chi_j{\mathcal T}_{j} \Tchi_\ell) \chi_\ell e_{{\times}, \ell}^{2n+1} +\sum_{j\leq \ell} (\chi_j {\mathcal T}_{j} \Tchi_\ell) \chi_\ell e_{{\times}, \ell}^{2n} \\
&= \sum_{\ell<j}  (\chi_j{\mathcal T}_{j} \Tchi_\ell)  \epsilon_{{\times}, \ell}^{2n+1} +\sum_{j\leq \ell} (\chi_j {\mathcal T}_{j} \Tchi_\ell) \epsilon_{{\times}, \ell}^{2n}.
\end{align*}
The result (i.e., the first equation in \eqref{eq:sweep_fwbw_error}) follows by the definition of $\mathbf T$ \eqref{eq:bold_T} and the fact that $\mathbf L$ and $\mathbf U$ are the lower/upper triangular parts of $\mathbf T$.
\end{proof}

The following lemma is proved in an analogous way to Lemma \ref{lem:model_seq_err_prop}.

\begin{lemma} \label{lem:model_seq_err_prop_gen}
For the general sequential method with ordering $\{\sigma_n\}$, for  $n\geq 0$, 
$$
{\boldsymbol{\epsilon}_{\times}^{n+1}} = \mathbf A_{n} {\boldsymbol{\epsilon}_{\times}^{n+1}} + \mathbf B_{n} {\boldsymbol{\epsilon}_{\times}^{n}},
$$
where
$$
(\mathbf A_{n})_{i,\ell}=
\begin{cases}
\mathbf T_{i,\sigma_n(\ell)} &\text{ if }\ell < \sigma_n^{-1}(i),\\
 0 &\text{ otherwise }
\end{cases}, \quad
(\mathbf B_{n})_{i,\ell}=
\begin{cases}
\mathbf T_{i,\sigma_n(\ell)} &\text{ if }\ell > \sigma_n^{-1}(i),\\
 0 &\text{ otherwise }
\end{cases}.
$$
\end{lemma}

Comparing Lemmas \ref{lem:model_seq_err_prop} and \ref{lem:model_seq_err_prop_gen}, we see that, for the forward-backward method, both ${\bf A}_n$ and ${\bf B}_n$ alternate between equalling $\bf L$ and $\bf U$.

By their definitions, we see that 
$({\mathbf A}_n)_{i,\ell}$ is non-zero if $\sigma_n(\ell)$ appears before $i$ in the ordering of subdomains $\{ \sigma_n(1), \ldots, \sigma_n(N)\}$ and $({\mathbf B}_n)_{i ,\ell}$ is non-zero if $\sigma_n(\ell)$ appears after $i$ in this ordering.
Compare this to the fact that $({\mathbf L}_n)_{i ,\ell}$ is non-zero if $\ell <i$, i.e., $\ell$ appears before $i$ in the ordering $\{1, 2, \ldots, N\}$ or after $i$ in the ordering $\{N, N-1, \ldots, 1\}$, and 
$({\mathbf U}_n)_{i ,\ell}$ is non-zero if $\ell >i$, i.e., $\ell$ appears after $i$ in the ordering $\{1, 2, \ldots, N\}$ or before $i$ in the ordering $\{N, N-1, \ldots, 1\}$.

\begin{proof}[Proof of Lemma \ref{lem:model_seq_err_prop_gen}]
Arguing exactly as in the proof of Lemma \ref{lem:model_seq_err_prop}, we obtain that
$$
({\boldsymbol{\epsilon}_{\times}^{n+1}})_{\sigma(j)} =
\sum_{\ell< j} 
 (\chi_{\sigma_n(j)}{\mathcal T}_{\sigma_n(j)} {\widetilde \chi}_{\sigma_n(\ell)})  \epsilon_{{\times}, {\sigma_n(\ell)}}^{n+1} +
 \sum_{j\leq \ell } 
 (\chi_{\sigma_n(j)} {\mathcal T}_{\sigma(j)} {\widetilde \chi}_{\sigma_n(\ell)}) \epsilon_{{\times}, {\sigma_n(\ell)}}^n.
$$
Therefore, with $i=\sigma_n(j)$, 
$$
({\boldsymbol{\epsilon}_{\times}^{n+1}})_i =
\sum_{\ell< \sigma_n^{-1}(i)} 
 (\chi_i{\mathcal T}_{i} {\widetilde \chi}_{\sigma_n(\ell)})  \epsilon_{{\times}, {\sigma_n(\ell)}}^{n+1} +
 \sum_{\sigma_n^{-1}(i)\leq \ell } 
 (\chi_i{\mathcal T}_{i} {\widetilde \chi}_{\sigma_n(\ell)}) \epsilon_{{\times}, {\sigma_n(\ell)}}^n
$$
and the result follows by the definition of $\mathbf T$ \eqref{eq:bold_T}.
\end{proof}

\section{Proof of Theorem \ref{thm:gen}} \label{s:mainproof}

\subsection{Precise definition of following a word}
 \label{ss:proof_gen}

Recall from \S\ref{sec:powers} the definitions of the set of words $\mathcal W$ and the composite map $\mathscr{T}_w$ \eqref{eq:Tw}.
Recall the definition of the trajectory staring at $(x,\xi)$ for a time interval $I$, $\gamma_I(x,\xi)$ \eqref{eq:gamma_I}.

We now define a set of cutoffs that are bigger than the $\chi_j$s. Let $\{\TTchi_j\}_{j=1}^N$ be such that, for each $j$, $\TTchi_j\in C^\infty(\Rea^d)$, $\supp \TTchi_j\cap \Omega \subset \Omega_j$, $\TTchi_j \equiv 1$ on $\supp \chi_j$, and $\TTchi_j$ is zero 
on $\supp(P^j_\newtheta-P_\newtheta)$.

\begin{definition}\label{def:follow}
A trajectory $\gamma$ for $P$ \emph{follows} a word $w \in \mathcal W$ of length $n\geq 2$ if 
\beqs
\gamma  = \prod_{2 \leq j \leq n} \gamma_j,
\eeqs
where the product stands for concatenation, 
$\gamma_j = \gamma_{(0, T_j]}(x_{j-1},\xi_{j-1})$ for some $T_j >0$ and $(x_j,\xi_j) \in T^* \mathbb R^d $ with
\bit
\item
for $1 \leq j \leq n-1$,
$x_j \in \operatorname{supp} \TTchi_{w_j} \cap \operatorname{supp} \big(P^{w_{j+1}}_{\rm s} - P_{\rm s} \big)$,
\item $\gamma_j \subset T^*\widetilde \Omega_{w_{j}}$ for $2 \leq j \leq n$ (recall Definition \ref{def:ext} of the extended domains),
\item 
$(x_n, \xi_n) := \Phi_{ T_{n}}(x_{n-1},\xi_{n-1}) \in T^* (\operatorname{supp} \TTchi_{w_n})$.
\eit
\end{definition}

Figure \ref{drw:word} illustrates a simple example of a trajectory following a word.

\subsection{Reduction to a propagation result for $\mathcal T_w$ (Lemma \ref{lem:key_prop2})} \label{ss:key_prop2}

Recall that in \S\ref{sec:ideas} we showed how Theorem \ref{thm:gen} follows from Lemma \ref{lem:notallowed}; i.e., the result that
\beqs
\text{ If }\,\, 
w\in \mathcal W \,\,\text{ is not allowed then } 
 \,\,
\|\mathscr T_w\|_{H^1_\hsc(\Omega) \rightarrow H^1_\hsc(\Omega)}= O(\hbar^{\infty}).
\eeqs
The heart of the proof of Lemma \ref{lem:notallowed} is the following result, described informally in \S\ref{sec:powers}, and proved in the next subsection using the propagation result of Lemma \ref{lem:key_prop}.

\begin{lemma}\label{lem:key_prop2} 
Suppose that $\coeffc$ is nontrapping.
Let $w \in \mathcal W$ of length $n := |w| \geq 1$. 
For any $\chi_{\rm meas} \in C^\infty_c(\widetilde \Omega_{w_n})$ and any tempered $\hbar$-family $v$ of elements of $H_0^1(\Omega)$, $\chi_{\rm meas} \mathscr T_w v$ is tempered and
\beq\label{eq:induction}
\operatorname{WF}_\hbar(\chi_{\rm meas} \widetilde{\mathscr T_w v}) \subset \Big \{ \rho \in T^* \operatorname{supp} \chi_{\rm meas}\,:\,\exists t_0 \leq 0 \text{ s.t. } \gamma_{(t_0, 0]}(\rho) \text{ follows } w \Big\},
\eeq
where $\widetilde{\mathscr T_w v} := S_{w_{n}} \mathscr T_w v$.
\end{lemma}

\bpf[Proof of Lemma \ref{lem:notallowed} using Lemma \ref{lem:key_prop2}]
The fact that  $\chi_{\rm meas} \mathscr T_w v$ is tempered is a consequence of the resolvent estimate given by Lemma \ref{lem:res_est}.
We fix a word $w \in \mathcal W$ not allowed.
Seeking a contradiction, we assume that \eqref{eq:Thursday1} fails; 
that is, there exists $\hbar_n \rightarrow 0$, $v_n \in H^1_0(\Omega)$ with $\Vert v_n \Vert_{H_{\hbar_n}^1} = 1$ and $N_0 \in \mathbb N$ such that
\beqs
\tfa n \geq 1, \hspace{0.3cm}\Vert \mathscr T_w(\hbar_n) v_n \Vert_{H^s_{\hbar_n}(\Omega)} \geq \hbar_n^{N_0}.
\eeqs
We show that 
\beqs
 \Vert \mathscr T_w(\hbar_n) v_n \Vert_{H^{s}_{\hbar_n}(\Omega)} = O(\hbar_n^\infty),
 \eeqs
  leading to a contradiction. To lighten the notation, from now on we drop the subscripts $n$ from $v_n$ and $\hbar_n$. 
By the definition of a word being allowed,
Lemma \ref{lem:key_prop2} implies that, for any $\chi \in C^\infty_c(\widetilde \Omega_{w_n})$, 
\beqs
\operatorname{WF}_{\hbar}(\chi \widetilde{\mathscr T_w v}) =\emptyset.
\eeqs
Therefore, by Lemma \ref{lem:osci}, Part \ref{it:osci1} together with 
 Lemma \ref{lem:comp_symbol}, Part \ref{it:symb_ell_infnt}, 
\beq\label{eq:almost1}
\tfa \chi \in C^\infty_c(\widetilde \Omega_{w_n})
\hspace{0.3cm} \Vert \chi \widetilde{\mathscr T_w v} \Vert_{H^s_\hbar(\widetilde{\Omega}_{w_n})} = O(\hbar_n^\infty).
\eeq
Recall that $\mathscr T_w v$ has a $\chi_{w_n}$ at the front by the definitions of $\mathscr T_w$ \eqref{eq:Tw} and $\mathbf T$ \eqref{eq:bold_T}.
We now choose $\chi$ in \eqref{eq:almost1} to be one on $\supp \chi_{w_n} \cap \Omega_{w_n}$ and supported in $\widetilde{\Omega}_{w_n}$, 
and obtain that 
\beqs
\Vert \mathscr T_w v \Vert_{H^s_{\hbar}(\Omega)} =
\Vert \mathscr T_w v \Vert_{H^s_{\hbar}(\Omega_{w_n})} = O(\hbar_n^\infty).
\eeqs
\end{proof}

\bre[Why we don't extend interior subdomains]\label{rem:extend}
Consider the end of the proof of Lemma \ref{lem:notallowed} using Lemma \ref{lem:key_prop2}. If $\Omega_{w_n}$ is an interior subdomain, i.e., $\partial \Omega_{w_n} \cap \partial\Omega =\emptyset$, then $\chi_{w_n}$ is compactly supported in $\Omega_{w_n}$ (indeed, its support does not touch the PML layer of $\Omega_{w_n}$), and we can choose $\chi$ to be compactly supported in $\Omega_{w_n}$; i.e., we do not need information about $\mathscr T_w v$ up to $\partial \Omega_{w_n}$. However, if $\Omega_{w_n}$ is not an interior subdomain, then $\chi_{w_n}$ can be supported up to the boundary of $\Omega_{w_n}$ (by the partition of unity property) and  
we therefore need information about $\mathscr T_w v$ up to $\partial \Omega_{w_n}$. This is why we extend $\Omega_{w_n}$ to $\widetilde{\Omega}_{w_n}$, choose 
$\chi$ to have support in $\widetilde{\Omega}_{w_n}\setminus \Omega_{w_n}$, and do the propagation in the interior of $\widetilde{\Omega}_{w_n}$ (to bypass the issue of propagating up to $\partial \Omega_{w_n}$, as discussed in Remark \ref{rem:PoS}).
\ere

\subsection{Proof of Lemma \ref{lem:key_prop2} (ending the proof of Theorem \ref{thm:gen})} \label{ss:proof_key_prop2}
It therefore remains to prove Lemma \ref{lem:key_prop2}.
The idea of the proof is to iteratively apply Lemma \ref{lem:key_prop}. 
We proceed by induction on the length $n$ of $w$. 
We first assume that the result of the lemma holds for any word of length $n$ and we show it for any word of length $n+1$. 
At the end of the proof we show that the result holds when $n=2$.

If $|w| = n+1$, we decompose $w$ as
$$
w = (\bar w, w_{n+1}),
$$
and write (by the definitions of $\mathscr T_w$ \eqref{eq:Tw} and $\mathcal{T}_{w_{n+1}}$ \eqref{eq:Tjl})
\beqs
\mathscr T_w v = \chi_{w_{n+1}} \mathcal T_{w_{n+1} } \Tchi_{w_n} \mathscr T_{\bar w} v.
\eeqs
We now apply Lemma \ref{lem:key_prop}. Let 
\beqs
W := \mathcal T_{w_{n+1}}  \Tchi_{w_n} \mathscr T_{\bar w} v
\eeqs
so that $\mathscr T_w v =\chi_{w_{n+1}} W$. 
By \eqref{eq:small_ext},
\beq\label{eq:imp2}
\chi_{\rm meas} \chi_{w_{n+1}} \widetilde W
=\chi_{\rm meas} \widetilde{(\chi_{w_{n+1}}  W)}
 =  \chi_{\rm meas} \widetilde{\mathscr T_w v}.
\eeq
 By the definition of $\mathcal{T}_{w_{n+1}}$ \eqref{eq:Tjl},
$$
\begin{cases}
P_{\newtheta}^{w_{n+1}} W = (P_{\newtheta}^{w_{n+1}} - P_{\newtheta})(\Tchi_{w_n} \mathscr T_{\bar w} v) \hspace{0.3cm} \text{in }{\Omega_{w_{n+1}}}, \\
W = \Tchi_{w_n} \mathscr T_{\bar w} v \hspace{0.3cm} \text{on }\partial\Omega_{w_{n+1}},
\end{cases}
$$
where we omit the restriction operators on $\Omega_{w_{n+1}}$ to keep the notation concise (as in the proof of Lemma \ref{lem:model_seq_err_prop} above).

We therefore apply Lemma \ref{lem:key_prop} with $j := w_{n+1}$, $f := (P_{\newtheta}^{w_{n+1}} - P_{\newtheta})(\Tchi_{w_n} \mathscr T_{\bar w} v)$, $g := \Tchi_{w_n} \mathscr T_{\bar w} v$, and $\psi :=\chi_{\rm meas}\chi_{w_{n+1}}$. 

We need to check that the assumptions in Points \ref{it:kp1}-\ref{it:kp2} are satisfied.
The assumption in Point \ref{it:kp1} is satisfied 
because $\coeffc$ is nontrapping.
The assumption in Point \ref{it:kp2} is satisfied since $\mathscr T_{\bar w} v \in H^1_0(\Omega)$ by Lemma \ref{lem:model_err_prop}.

The result of Lemma \ref{lem:key_prop} is that
\begin{multline*}
\operatorname{WF}_\hbar (\chi_{\rm meas}\chi_{w_{n+1}} \widetilde W) \subset \\ \bigcup_{\chi \in C^\infty_c(\widetilde \Omega_{w_{n+1}})}
\Big\{ \rho \in T^* \operatorname{supp} \chi_{\rm meas}\,:\, \hspace{0.3cm} \exists t <0\, \text{ s.t. }\, \gamma_{[t, 0]}(\rho) \subset T^* \widetilde \Omega_{w_{n+1}} \text{ and }  \Phi_t(\rho) \in \operatorname{WF}_\hbar (\chi \widetilde f) \Big\},
\end{multline*}
where $\widetilde f = S_{w_{n+1}} f$, $\widetilde W = S_{w_{n+1}}W$. 
Recall that in the extended region the PML is linear and the coefficients of both $P_\newtheta^j$ and $P_\newtheta$ are constant.
The arguments used in the proof of Lemma \ref{lem:ext_h} therefore show that  
the extension operator $S_j$ of Definition \ref{def:ext} is such that 
\beq\label{eq:rhsweird}
S_j\Big( (P_{\newtheta}^{j+1} - P_{\newtheta})(\chi_j w_{j})\Big)
=(P_{\newtheta}^{j+1} - P_{\newtheta})S_j(\chi_j  w_{j})
=(P_{\newtheta}^{j+1} - P_{\newtheta})(\chi_j \widetilde w_{j}).
\eeq
where we have used \eqref{eq:small_ext} in the second step.

Therefore
\begin{multline}
\operatorname{WF}_\hbar (\chi_{\rm meas}\chi_{w_{n+1}} \widetilde W) \subset  \bigcup_{\chi \in C^\infty_c(\widetilde \Omega_{w_{n+1}})}
\Big\{ \rho \in T^* \operatorname{supp} \chi_{\rm meas}\,:\, \exists t <0 \, \text{ s.t. }\, \gamma_{[t, 0]}(\rho) \subset T^* \widetilde \Omega_{w_{n+1}} \\\text{ and }  \Phi_t(\rho) \in \operatorname{WF}_\hbar (\chi (P_{\newtheta}^{w_{n+1}} - P_{\newtheta})(\Tchi_{w_n} \widetilde{\mathbf T_{\bar w} v})) \Big\}.
\label{eq:window1}
\end{multline}
We ultimately use that, by \eqref{eq:mult_WF} and \eqref{eq:mult_WF1},
\beq\label{eq:sunny3}
\operatorname{WF}_\hbar \chi(P_{\newtheta}^{w_{n+1}} - P_{\newtheta})(\Tchi_{w_n} \widetilde{\mathbf T_{\bar w} v}) \subset 
 \operatorname{WF}_\hbar (\chi^{>}\Tchi_{w_n} \widetilde{\mathbf T_{\bar w} v})
\cap  T^* (\operatorname{supp} \Tchi_{w_n}) \cap T^* (\operatorname{supp} (P_{\newtheta}^{w_{n+1}} - P_{\newtheta})),
\eeq
where $\chi^{>}\in C^\infty_c(\widetilde \Omega_{w_{n+1}})$ is such that $\chi^{>}\equiv 1$ on $\supp \chi$. 
However, we first change the cut-off $\chi \in C^\infty_c(\widetilde \Omega_{w_{n+1}})$ into a cut off in $C^\infty_c(\widetilde \Omega_{w_n})$ so that we can use 
the induction hypothesis, which gives information about $ \operatorname{WF}_\hbar (\chi \Tchi_{w_n} \mathbf T_{\bar w} v)$ with $\chi \in C^\infty_c(\widetilde \Omega_{w_n})$.

Observe that 
\beqs
\tfa \chi \in C^\infty_c(\widetilde{\Omega}_{w_{n+1}})\text{ there exists } \psi \in C^\infty_c(\widetilde{\Omega}_{w_n}) \text{ s.t. }
\chi (P_{\newtheta}^{w_{n+1}}- P_{\newtheta})\Tchi_{w_n} = \psi (P_{\newtheta}^{w_{n+1}}- P_{\newtheta})\Tchi_{w_n}
\eeqs
(note that this crucially relies on the presence of $\Tchi_{w_n}$ on the right of both sides of the equality).

Using this in \eqref{eq:window1}, we have 
\begin{align}\nonumber
&\operatorname{WF}_\hbar (\chi_{\rm meas}\chi_{w_{n+1}} \widetilde W) \subset  \bigcup_{\psi \in C^\infty_c(\widetilde \Omega_{w_{n}})}
\Big\{ \rho \in T^* \operatorname{supp} \chi_{\rm meas}\,:\, \exists t <0\,\text{ s.t. }\, \gamma_{[t, 0]}(\rho) \subset T^* \widetilde \Omega_{w_{n+1}} \\
&\hspace{6.6cm}
\text{ and }  \Phi_t(\rho) \in \operatorname{WF}_\hbar (\psi (P_{\newtheta}^{w_{n+1}} - P_{\newtheta})(\Tchi_{w_n} \widetilde{\mathbf T_{\bar w} v})) \Big\}.\label{eq:Friday17}
\end{align}
Therefore, by  \eqref{eq:sunny3} with $\chi$ replaced by $\psi$, 
\begin{align*}
&\operatorname{WF}_\hbar (\chi_{\rm meas}\chi_{w_{n+1}} \widetilde W)
\\&\hspace{1.5cm} \subset  \bigcup_{\psi \in C^\infty_c(\widetilde \Omega_{w_{n}})}
\Big\{ \rho \in T^* \operatorname{supp} \chi_{\rm meas} \,:\, \exists t <0\,\text{ s.t. }\, \gamma_{[t, 0]}(\rho) \subset T^* \widetilde \Omega_{w_{n+1}} \\
&\hspace{3.5cm}\text{ and }  \Phi_t(\rho) \in \operatorname{WF}_\hbar (\psi^{>} \Tchi_{w_{n}} \widetilde{\mathbf T_{\bar w} v})\cap  T^* (\operatorname{supp} \TTchi_{w_n} )\cap T^*\big(\operatorname{supp}(P_{\newtheta}^{w_{n+1}} - P_{\newtheta})\big) \Big\}.
\\&\hspace{1.5cm} =  \bigcup_{\psi \in C^\infty_c(\widetilde \Omega_{w_{n}})}
\Big\{ \rho \in T^* \operatorname{supp} \chi_{\rm meas} \,:\, \exists t <0\,\text{ s.t. }\, \gamma_{[t, 0]}(\rho) \subset T^* \widetilde \Omega_{w_{n+1}} \\
&\hspace{3.5cm}\text{ and }  \Phi_t(\rho) \in \operatorname{WF}_\hbar (\psi \Tchi_{w_{n}} \widetilde{\mathbf T_{\bar w} v})\cap  T^* (\operatorname{supp} \TTchi_{w_n} )\cap T^*\big(\operatorname{supp}(P_{\newtheta}^{w_{n+1}} - P_{\newtheta})\big) \Big\}.
\end{align*}
Therefore, by the induction hypothesis that \eqref{eq:induction} holds with $w$ replaced by $\overline{w}$ and then the Definition \ref{def:follow} of following a word,
\begin{align*}
\operatorname{WF}_\hbar (\chi_{\rm meas} \chi_{w_{n+1}} \widetilde W) &\subset 
\Big\{ \rho \in T^* \operatorname{supp}  \chi_{\rm meas} \,:\,\exists t <0\,\text{ s.t. }\, \gamma_{[t, 0]}(\rho) \subset T^* \widetilde \Omega_{w_{n+1}} \text{ and } \\
&\hspace{4cm} \Phi_t(\rho) \in   T^* (\operatorname{supp} \TTchi_{w_n}) \cap T^* \big(\operatorname{supp}(P_{\newtheta}^{w_{n+1}} - P_{\newtheta})\big), \\
&\hspace{4cm}\text{ and }  \exists t_0 \leq 0 \text{ s.t. } \gamma_{(t_0, 0]}(\Phi_t(\rho)) \text{ follows } \bar w  \Big\}\\
&\subset\Big\{ \rho \in T^* \operatorname{supp}  \chi_{\rm meas} \,:\,  \exists t_1 \leq 0 \,\text{ s.t. }\, \gamma_{(t_1, 0]}(\rho) \text{ follows }  w  \Big\},
\end{align*}
and the result for words of length $n+1$, assuming the result for words of length $n$, follows by \eqref{eq:imp2}. 
To complete the induction, we prove the result for words of length $2$. If $w=(w_1,w_2)$, then the inclusion \eqref{eq:Friday17} simplifies to 
\begin{align*}\nonumber
&\operatorname{WF}_\hbar (\chi_{\rm meas}\chi_{w_{2}} \widetilde W) \subset  \bigcup_{\psi \in C^\infty_c(\widetilde \Omega_{w_{1}})}
\Big\{ \rho \in T^* \operatorname{supp} \chi_{\rm meas}\,:\, \exists t <0\,\text{ s.t. }\, \gamma_{[t, 0]}(\rho) \subset T^* \widetilde \Omega_{w_{2}} \\
&\hspace{6.6cm}
\text{ and }  \Phi_t(\rho) \in \operatorname{WF}_\hbar (\psi (P_{\newtheta}^{w_{2}} - P_{\newtheta})\Tchi_{w_1}) \Big\}.
\end{align*}
Arguing as in \eqref{eq:sunny3} and using the definition of following a word (Definition \ref{def:follow}), we obtain
\begin{align*}
&\operatorname{WF}_\hbar (\chi_{\rm meas}\chi_{w_{2}} \widetilde W) 
\subset\Big\{ \rho \in T^* \operatorname{supp}  \chi_{\rm meas} \,:\,  \exists t< 0 \,\text{ s.t. }\, \gamma_{(t, 0]}(\rho) \text{ follows }  w  \Big\},
\end{align*}
and the proof is complete. 

\section{Proofs of Theorems \ref{thm:strip}, \ref{thm:sweep_strip}, \ref{thm:check}, \ref{thm:sweep}}
\label{sec:checkerboard}

\subsection{Checkerboard coordinates}

The proofs of Theorems \ref{thm:strip}, \ref{thm:sweep_strip}, \ref{thm:check}, \ref{thm:sweep} use a set of checkerboard coordinates $\{c_\ell\}_{1\leq \ell \leq \mathfrak d}$. 

For concreteness we use the natural coordinates coming from the checkerboard construction in \S\ref{sec:strip_checker}.
For a $d$-checkerboard, the coordinates 
(with origin $(0,\ldots, 0)$) of the subdomain $\Omega_j$ are $(c_1(j),\ldots,c_{d}(j))$ where 
$c_\ell(j) = m_\ell$ if the $U_j$ from which $\Omega_j$ is formed equals $\prod_{\ell=1}^d (y_{m_\ell-1}^\ell, y_{m_{\ell}}^\ell)$. 

For a $\mathfrak d$-checkerboard, we omit the $(d-\mathfrak d)$ directions that have $N_\ell = 1$ to obtain coordinates $\{c_\ell\}_{1\leq \ell \leq \mathfrak d}$.  

Examples \ref{ex:exhaust2} and \ref{ex:exhaust4} below use coordinates with respect to other checkerboard vertices; these are defined in an analogous way by first rotating the checkerboard to place that particular vertex at the origin and relabelling all the points $y^{\ell}_m$, $\ell =1,\dots,d$, $m=0,\ldots, N_\ell$.

\subsection{Allowed words in a checkerboard when $c\equiv 1$}

\begin{lemma} \label{lem:traj_check}
Assume that $\{ \Omega_j\}_{j=1}^N$ is a checkerboard and $\coeffc = 1$. Then, if $w \in \mathcal W$ is allowed, for any $1 \leq \ell \leq \mathfrak d$, the map
$j \mapsto c_\ell(w_j)$ is monotonic.
\end{lemma}
\begin{proof}
The underlying idea of the proof is that a straight line through a non-overlapping checkerboard
(i.e., the decomposition of $\Omega_{\rm int}$ described in Point (i) of \S\ref{sec:strip_checker}) necessarily has monotonic coordinate maps. The subtlety here is that 
the subdomains $\{ \Omega_j\}_{j=1}^N$ overlap; however, the fact that an allowed word must pass through the (disjoint) subdomains in \eqref{eq:follow_domains} essentially puts us back in the non-overlapping case.

Assume that $n := |w| \geq 2$, otherwise the statement is void.  
By the definition of allowed, there exists a trajectory following $w$. Since the coefficients are constant, this trajectory is a straight line, and by the definitions of allowed (Definition \ref{def:allowed}) and follow (Definition \ref{def:follow}), the trajectory intersects 
$$
\bigg[ \bigcup_{1 \leq j \leq n-1} T^* \widetilde{\Omega}_{w_{j+1}} \cap T^* \Big(\operatorname{supp} \TTchi_{w_j} \cap \operatorname{supp} \big(P^{w_{j+1}}_{\rm s} - P_{\rm s} \big)\Big)\bigg] \cup \big(T^* \operatorname{supp} \TTchi_{w_n} \cap  T^* \widetilde{\Omega}_{w_n} \big) =: \bigcup_{1 \leq j \leq n} X_j
$$
Since $\TTchi_\ell$ is supported in 
$T^* \widetilde{\Omega}_{\ell}$ and away from $\operatorname{supp} \big(P^{{\ell}}_{\rm s} - P_{\rm s} \big)$, this implies that 
$X_j \cap X_{j+1} = \emptyset$ for $1 \leq j \leq n-1$ and $X_j \subset T^*\widetilde{\Omega}_{w_j}$ for $1 \leq j \leq n$.
It follows that there exists $(x_1, \dots, x_n)$ on a line with $x_j \in \widetilde{\Omega}_{w_j} \backslash \widetilde{\Omega}_{w_{j+1}}$, $j = 1, \dots, n-1$ and $x_n \in \widetilde{\Omega}_{w_n}$. 
Because $x_j \in \widetilde\Omega_j\setminus \widetilde\Omega_{j+1}$ and 
$\{ \Omega_j\}_{j=1}^N$ is a checkerboard, it follows that 
for any $\ell$ such that $c_\ell(j) \neq c_\ell(j+1)$,
$c_\ell(j) \leq c_\ell(j+1)$ if and only if $(x_j)_\ell \leq (x_{j+1})_\ell$ (where subscript $\ell$ denotes the $\ell$th coordinate of $x$ in $\Rea^{\mathfrak d}$);
note that for this to be true it is key that the $x_j$ is not in the overlap of $\widetilde\Omega_j$ and $\widetilde\Omega_{j+1}$ (which we have ensured above). The result follows.
\end{proof}

\subsection{Proofs of Theorems \ref{thm:strip} and \ref{thm:check} (i.e., the results about the parallel method with $c\equiv 1$)}
\label{sec:allowed_constant_coeff}

By Lemma \ref{lem:traj_check}, if $\{ \Omega_j\}$ is a $\mathfrak d$-checkerboard,
then $\mathcal N = 1+(N_1-1) + \cdots + (N_{\mathfrak d} -1) =  N_1 + \cdots + N_{ d} - ({ d}-1)$. Indeed, this is the length of the longest component-wise monotonic sequence with values in $\{1, \dots, N_1 \} \times \dots \times \{1, \dots, N_d \}$. Thus
Theorem \ref{thm:check} is a consequence of Theorem \ref{thm:gen}. 
Theorem \ref{thm:strip} is then a consequence of Theorem \ref{thm:check} with ${\mathfrak d}=1$. 

\subsection{Exhaustive sequence of orderings} \label{ss:exhaust}

Given a sequence of orderings $\{\sigma_n\}_{1\leq n \leq S}$, for every $n$ we define the order relation $\preceq_n$ on $\{1, \cdots, N \}$ by
\beqs
i \preceq_n j \iff \sigma_n^{-1}(i) \leq \sigma_n^{-1}(j);
\eeqs
i.e., $i \preceq_n j$ if at step $n$ of the method, $\Omega_i$ comes before $\Omega_j$.

Recall that $\{c_\ell\}_{1\leq \ell \leq \mathfrak d}$ is a fixed set of checkerboard coordinates, i.e., $c_\ell : \{ 1, \ldots, N\} \mapsto \{ 1, \dots, N_\ell\}$ is such that $c_\ell(j)$ is the $\ell$th coordinate of $\Omega_j$ on the checkerboard.

\begin{definition}[Exhaustive ordering] \label{def:exhaust}
We say that an ordering $\{\sigma_n\}_{1\leq n \leq S}$ is \emph{exhaustive} if 
(i) for any $1\leq m \leq S$, there is $1\leq \overline m \leq S$ such that 
$\sigma_m(i) = \sigma_{\overline m}(N-i+1)$, and (ii)
for any sequence $\{j_i\}_{1\leq i \leq n}$ of elements of $\{1, \cdots, N \}$
such that
\beq\label{eq:mcm}
\tfa 1\leq\ell\leq \mathfrak d, \quad i \mapsto c_\ell(j_i) \text{ is monotonic},
\eeq
there exists an $1\leq s \leq n$ so that
\beq\label{eq:visitinorder}
j_1 \preceq_s \cdots \preceq_s j_{n}.
\eeq
\end{definition} 

Recall the informal definition of exhaustive in \S\ref{sec:statement_check} that 
(i) given any ordering in the sequence, the sequence also contains the reverse ordering (i.e., where the subdomains are visited in reverse order), and (ii) given any directed straight line (here a sequence of subdomains with monotonic coordinate maps \eqref{eq:mcm}) the sequence of sweeps has to contain at least one sweep such that the subdomains in the straight line are visited in order, but not necessarily sequentially, in that sweep; i.e., \eqref{eq:visitinorder}.

Observe that, if $m, \overline m$ are the indices of reverse orderings, as in Definition \ref{def:exhaust}, then $\preceq_m \; = \; \succeq_{\overline m}$.

In what follows, we assume that the Cartesian directions on the checkerboard are always ordered in the same way, i.e., if $(c^n_\ell)_{1\leq n \leq {\mathfrak d}}$ and $(c^m_\ell)_{1\leq n \leq {\mathfrak d}}$ are two set of checkerboard coordinates, then $c^n_\ell$ and $c^m_\ell$ represent coordinates in the same $\ell$th direction.

\begin{example}[Lexicographic orderings]  \label{ex:exhaust2}
Let $(\mathfrak v_1, \dots, \mathfrak v_{2^{\mathfrak d}})$
be the vertices of the checkerboard. For ${1\leq n \leq 2^{\mathfrak d}}$, let $(c^n_\ell)_{1\leq \ell \leq \mathfrak d}$ 
 be the set 
  of checkerboard coordinates with $\mathfrak v_n$ as origin. Define $\sigma_n$, for ${1\leq n \leq 2^{\mathfrak d}}$,  to be the lexicographic order with respect to $(c^n_\ell)_{1\leq \ell \leq \mathfrak d}$, ie the lexicographic order with vertex $\mathfrak v_n$ as origin.
Then $\{ \sigma_n \}_{1\leq n \leq 2^{\mathfrak d}}$ is exhaustive. 
One such ordering is given by 
defining 
  $$
  \sigma_n(j) := 1 + \sum_{\ell = 1}^{\mathfrak d} \big(\prod_{m \leq \ell-1}N_m \big) (c^n_\ell(j) - 1);
  $$
this is shown in the case $\mathfrak d = d=2$ and $N_1=N_2=3$ in Figure \ref{fig:lex}.
\end{example}

\begin{example}[An algorithm for generating exhaustive sequences of orderings]  \label{ex:exhaust4}
$\;$\newline
Let $(\mathfrak v_1, \dots, \mathfrak v_{2^{\mathfrak d}})$
be the vertices of the checkerboard.
For any ${1\leq n \leq 2^{\mathfrak d}}$, 
let $(x^n_\ell)_{1\leq \ell \leq \mathfrak d}$ be 
the set 
  of checkerboard coordinates with $\mathfrak v_n$ as origin. For  $n=1,\ldots, 2^{\mathfrak d}$, 
  if $\sigma_n$ is not already defined, define $\sigma_n$ in such a way that all the maps $j \mapsto x^n_\ell(\sigma_n(j))$ are non-decreasing, 
  then,  for $\bar n \in \{ 1, \dots, 2^{\mathfrak d} \}$ such that ${\mathfrak v}_{\bar n}$ has for coordinates $(c^n_1(\sigma_n(N)), \dots, c^n_\ell(\sigma_n(N))$, define $\sigma_{ \bar n}(j) :=  \sigma_n(N - j +1)$. Then $\{ \sigma_n \}_{1\leq n \leq 2^{\mathfrak d}}$ is exhaustive.
\end{example}
Figure \ref{fig:checker2} illustrates an exhaustive sequence of orderings for a 3$\times$3 checkerboard following the algorithm in Example \ref{ex:exhaust4}.

\begin{figure}
\begin{center}
\begin{tabular}{c|c|c}
6 & 8 & 9  \\
\hline
4 & 5 & 7  \\
\hline
1 & 2 & 3
\end{tabular}\hspace{0.5cm}
\begin{tabular}{c|c|c}
4 & 2 & 1  \\
\hline
6 & 5 & 3  \\
\hline
9 & 8 & 7
\end{tabular}\hspace{0.5cm}
\begin{tabular}{c|c|c}
9 & 6 & 5 \\
\hline
8 & 4 & 2 \\
\hline
7 & 3 & 1
\end{tabular}\hspace{0.5cm}
\begin{tabular}{c|c|c}
1 & 4 & 5 \\
\hline
2 & 6 & 8 \\
\hline
3 & 7 & 9
\end{tabular}
\end{center}
\caption{Example of an exhaustive ordering of a $3 \times 3$ checkerboard obtained by following the construction in Example \ref{ex:exhaust4};
the key point is that, moving through the subdomains in order, their coordinates on the checkerboard with respect to the first subdomain are non-decreasing.
}\label{fig:checker2}
\end{figure}

\ble \label{lem:ex_size}
On a $2$-checkerboard, if  $\{ \sigma_n \}_{1\leq n \leq S}$ is exhaustive then $S\geq 2^2$. 
\ele

\bpf
Suppose we have an exhaustive sequence of orderings on a $2\times 2$ $2$-checkerboard. Label the subdomains 
\begin{center}
\begin{tabular}{c|c}
$C$ & $D$\\
\hline
$A$ & $B$
\end{tabular}\hspace{0.5cm}
\end{center}
By the definition of exhaustive, there exist an ordering $\sigma_n$ that visits (in order) $C,A,B$ 
and an ordering $\sigma_m$ that visits (in order) $C, D,B$, with these subdomain orderings corresponding to diagonal lines from the top left to bottom right passing, respectively, below and above the centre.

If $\sigma_n=\sigma_m$, then $\sigma_n$ equals either $(C,A,D,B)$ or $(C,D,A,B)$.
In either case, neither ordering, nor its reverse, visits $A,B,D$ or $A,C,D$ in order 
(with these subdomain orderings corresponding to diagonal lines from the bottom left to top right passing, respectively, below and above the centre). Therefore 
the sequence of orderings consisting of $\sigma_n$ and its reverse is not exhaustive, and so $S\geq 4$. 

If $\sigma_n\neq \sigma_m$ 
then the reverse ordering to $\sigma_n$ (which visits $B,A,C$ in order, and hence $B$ before $C$) is neither $\sigma_m$ (which visits $C,D,B$ in order, and hence $B$ after $C$) or the reverse ordering to $\sigma_m$ (since 
$\sigma_n\neq \sigma_m$). Therefore $S\geq 4$.

For an arbitrary 2-checkerboard, we repeat the argument above in the top left $2\times 2$ corner. 
\epf

\subsection{Proof of Theorem \ref{thm:sweep}} \label{ss:proof_gen_sweep}

Arguing as in the proof of Theorem \ref{thm:sweep_strip} outlined in \S\ref{sec:idea_sequential}, and using  Lemma \ref{lem:model_seq_err_prop_gen} instead of Lemma \ref{lem:model_seq_err_prop}, we obtain
$$
{\boldsymbol{\epsilon}_{\times}^{n+1}} =  \sum_{l=0}^{N-1} \mathbf A_n^l \mathbf B_n {\boldsymbol{\epsilon}_{\times}^{n}},
$$
where we used that $\mathbf A_n ^N = 0$ since $\mathbf A_n$ is the rearrangment of an lower-triangular matrix. It follows that
$$
{\boldsymbol{\epsilon}_{\times}^{S}} =  \sum_{0 \leq \ell_1, \dots, \ell_S\leq N-1} \mathbf {\mathbf A}_S^{\ell_S} {\mathbf B}_S \dots {\mathbf A}_1^{\ell_1}{\mathbf B}_1 {\boldsymbol{\epsilon}_{\times}^{0}},
$$
and Theorem \ref{thm:sweep} therefore follows immediately from the following lemma.

\begin{lemma} \label{lem:key_gen_sweep}
Assume that $\{\sigma_n\}_{1\leq n \leq S}$ is exhaustive (in the sense of Definition \ref{def:exhaust}). Then, for any
$\ell_1, \dots, \ell_S \in \mathbb{N}$,
$$
\Vert   {\mathbf A}_S^{\ell_S} {\mathbf B}_S \dots {\mathbf A}_1^{\ell_1}{\mathbf B}_1 \Vert_{(H^1(\Omega))^N \mapsto (H^s(\Omega))^N } = O(\hbar^\infty).
$$
\end{lemma}

To prove Lemma \ref{lem:key_gen_sweep}, we identify the words arising in the entries of the products ${\mathbf A}_S^{\ell_S} {\mathbf B}_S \dots {\mathbf A}_1^{\ell_1}{\mathbf B}_1$, show that they are not allowed, and then apply Lemma \ref{lem:notallowed}.

\begin{lemma} \label{lem:crux_gen}
Assume that $\{ \sigma_n \}_{1 \leq n \leq S}$ is exhaustive. Then, no word  $w = ( w_1, \dots, w_n)$ satisfying
\beq\label{eq:ass1}
w_1 \succeq_{j_1} w_2 \succeq_{j_2} w_3 \succeq_{j_3} \dots \succeq_{j_{n-1}} w_{n},
\eeq
where $(j_i)_{1 \leq i \leq n-1}$ is a \emph{surjective} sequence 
 with values in $\{1, \dots, S\}$ is allowed.
\end{lemma}

\begin{proof}[Proof of Lemma \ref{lem:key_gen_sweep} assuming Lemma \ref{lem:crux_gen}]
Let
$$
\mathbf M :=  {\mathbf A}_S^{\ell_S} {\mathbf B}_S \dots {\mathbf A}_1^{\ell_1}{\mathbf B}_1.
$$
Observe that
\beq\label{eq:multB}
(\mathbf B_{m} \mathbf X)_{i,j}  = \sum_{\ell=1}^N (\mathbf B_{m})_{i,\ell} (\mathbf X)_{\ell,j} = \sum_{\ell \geq \sigma_{m}^{-1}(i)} (\mathbf T)_{i, \sigma_{m}(\ell)} (\mathbf X)_{\ell,j}, 
\eeq
and
\beq\label{eq:multA}
(\mathbf A_{m} \mathbf X)_{i,j}  = \sum_{\ell=1}^N (\mathbf A_{m})_{i,\ell} (\mathbf X)_{\ell,j} = \sum_{\ell \leq \sigma_{m}^{-1}(i)} (\mathbf T)_{i, \sigma_{m}(\ell)} (\mathbf X)_{\ell,j}.
\eeq

Now, any product of $\mathbf A_{\ell}$ and $\mathbf B_{\ell'}$, $1\leq\ell, \ell'\leq  N$, can be understood as the sum of $\mathcal{T}_{w'}$ for suitable words $w'$.
Recall from the definition of $\mathcal{T}_{w'}$ \eqref{eq:Tw} that the last letter in $w$ appears in the left-most operator in the composition defining $\mathcal{T}_{w'}$.
The equation \eqref{eq:multB} therefore implies that left-multiplying such a product by $\mathbf B_m$ 
results in a linear combination of operators $\mathcal{T}_{w}$ 
 such that the last letter of $w$ (corresponding to $i$ in \eqref{eq:multB}) is $\preceq_m$ 
the previous letter  (corresponding to $\sigma_{m}(\ell)$ in \eqref{eq:multB}). 
 Similarly, \eqref{eq:multA} implies that  left-multiplying a product of $\mathbf A_{\ell}$ and $\mathbf B_{\ell'}$, $1\leq\ell, \ell'\leq  N$, by $\mathbf A_m$ 
 results in a linear combination of operators $\mathcal{T}_{w}$  such that the last letter of $w$ (corresponding to $i$ in \eqref{eq:multA}) is $\succeq_m$ the previous letter (corresponding to $\sigma_{m}(\ell)$ in \eqref{eq:multA}). 

By induction, it follows that the entries of $\mathbf M$ are linear combinations of operators of the form $\mathcal T_w$, where $w$ is a word of size $|w| := n \geq S$ that can be written as the concatenation
$$
w = w^0 w^1  \dots  w^S,
$$
where 
$$
|w^0| = 1,
$$
and each word $w^m$ for $ 1 \leq m \leq S$ is of size $|w^m| = n_m := 1 + \ell_m $ and verifies
$$
\begin{cases}
w^m =  (\sigma_m(i_1^m), \dots, \sigma_m(i_{n_m}^m)),\\
i_1^m \leq \sigma^{-1}_{m}(w_{n_{m-1}}^{m-1}), \\
i_{k}^m \geq \sigma^{-1}_{m}(w_{k-1}^m)  \tfa 2 \leq k \leq n_{m}.
\end{cases}
$$
Observe that, by definition of $\succeq_m$,
for any $1 \leq m \leq S$
$$
w_1^m \preceq_m w_{n_{m-1}}^{m-1} \hspace{1cm} \text{ and }\hspace{1cm} w_{k}^m \succeq_m w_{k-1}^m  \tfa 2 \leq k \leq n_{m}.
$$
On the other hand, by the definition of exhaustive, for any $1 \leq m \leq S$, there is $1 \leq \overline m \leq S$ such that $\preceq_m \; = \; \succeq_{ \overline m}$. We thus have, for any  $1 \leq m \leq S$
$$
w_{n_{m-1}}^{m-1} \succeq_m w_1^m  \succeq_{\overline m} w_2^m \succeq_{\overline m} \dots \succeq_{\overline m} w_{n_m}^m.
$$
It follows that, denoting $w = ( w_1, \dots, w_n)$, there is a \emph{surjective} sequence 
$(j_i)_{1 \leq i \leq n-1}$ with values in $\{1, \dots, S\}$ so that
$$
w_1 \succeq_{j_1} w_2 \succeq_{j_2} w_3 \succeq_{j_3} \dots \succeq_{j_{n-1}} w_{n}.
$$
Hence $w$ is not allowed by Lemma \ref{lem:crux_gen}, and Lemma \ref{lem:notallowed} gives the result.
\end{proof}

\begin{proof}[Proof of Lemma \ref{lem:crux_gen}]
Seeking a contradiction, assume that such a word $w$ is allowed. By Lemma \ref{lem:traj_check}, all the coordinates maps $i \mapsto c_\ell(w_i)$ are monotonic. As a consequence, 
since $\{ \sigma_n \}$ is exhaustive, 
there exists an $1\leq s \leq S$ so that 
\beq
\label{eq:final1}
w_1 \preceq_{s} w_2 \preceq_{s} w_3 \preceq_{s} \dots \preceq_{s} w_{n}.
\eeq
However,  there exists $1 \leq i \leq n-1$ such that $j_i = s$. Therefore, by the assumption \eqref{eq:ass1} and \eqref{eq:final1}
$$
w_{i}  \succeq_{j_i} w_{i+1} \quad \text{ and }\quad w_{i}  \preceq_{j_i} w_{i+1},
$$
hence $w_{i} = w_{i+1}$, which contradicts the definition of a word.
\end{proof}

\begin{remark}[The analogue of Theorem \ref{thm:sweep} when $c\not\equiv1$]\label{rem:sequential_variable_c}
In the general case when $c\not\equiv1$, we define 
an ordering $\{\sigma_n\}_{1\leq n \leq S}$ to be \emph{exhaustive} if 
(i) for any $1\leq m \leq S$, there is $1\leq \overline m \leq S$ such that 
$\sigma_m(i) = \sigma_{\overline m}(N-i+1)$, and (ii)
for any allowed word $w$ of size $n \geq S+1$, there exists an $1\leq s \leq S$ so that
$$
w_1 \preceq_s \cdots \preceq_s w_{n}.
$$
The proofs of Lemma \ref{lem:key_gen_sweep} and Lemma \ref{lem:crux_gen} still hold, and this therefore
leads immediately to the generalisation of Theorem \ref{thm:sweep} when $c\not \equiv1$ with the result that, if $\{ \sigma_n \}$ is exhaustive (as defined above) and of size $S$ then
$$
\Vert u^S - u \Vert_{H^s_k(\Omega)} \leq k^{-M}\Vert u^0 - u \Vert_{H^1_k(\Omega)}.
$$
For such orderings to always exist, one would need in principle to define the sequential method in a more general way by not imposing the $\sigma_n$ to be bijections, allowing to visit a subdomain multiple times during a single sweep -- as would a non-straight trajectory visiting the same subdomain multiple times.
Constructing such orderings is, however, difficult, since it depends on the specific dynamic of the Hamilton flow imposed by the non-constant $\coeffc$, and we have therefore not pursued this further in this paper.
\end{remark}

\section{Numerical experiments}\label{sec:numerical}

\subsection{The overlapping Schwarz methods on the discrete level}\label{sec:num1}

Given a finite-dimensional subspace of $H^1_0(\Omega)$, which we denote by $V_h(\Omega)$, the discrete analogues of the parallel and sequential methods described in \S\ref{sec:parallel} and \S\ref{sec:sequential} (respectively) compute approximations to the solution of \eqref{eq:PDE} in $V_h(\Omega)$ using  Galerkin approximations 
to the solutions of the local problems \eqref{eq:localprob} and \eqref{eq:general_sweeping} in appropriate subsets of $V_h(\Omega)$. 

We now describe this process in detail for the discretisation of the parallel overlapping method in \S\ref{sec:parallel}, showing that it gives a natural PML-variant of the well-known RAS (restricted additive Schwarz) method (introduced in \cite{CaSa:99}).

\paragraph{The Galerkin approximation of  the PDE \eqref{eq:PDE}.}

Let $\{\mathcal{T}_h\}$ be a shape-regular conforming sequence of meshes  for $\Omega$. We assume that the edges of $\mathcal{T}_h$ resolve the boundaries of $\Omega$, $\Omega_{\rm int}$, $\Omega_{\rm int, j}$, and $\Omega_j$, for all $j=1,\ldots,N$; since all these domains are unions of hyperrectangles, this is straightforward to achieve in practice.

Let $V_h(\Omega)\subset H_0^1(\Omega)$ be the standard finite-element space of continuous piecewise-polynomial functions of degree $\leq p$ on $\mathcal{T}_h$. 
Given $f \in H^{-1}(\Omega)$, the Galerkin approximation to the solution of \eqref{eq:PDE} is defined as the solution of the problem:~find $u_h\in V_h(\Omega)$ such that 
\begin{equation}\label{eq:global_vari}
a(u_h,v_h) = \langle f, v_h\rangle_{H^{-1}(\Omega) \times H^1(\Omega)} \quad \tfa\,  v_h\in V_h(\Omega), 
\end{equation}
where the sesquilinear form $a(\cdot,\cdot)$ is obtained by multiplying the PDE 
\eqref{eq:PDE} by a test function and integrating by parts; i.e., when $d=2$, 
\begin{equation}\label{eq:sesqui}
a(u,v):=\int_{\Omega_j} \bigg(k^{-2}\big((D\nabla u) \cdot \nabla \overline{v}  - (\beta\cdot \nabla u)\overline{v} \big)-  c^{-2}u \overline{v} \bigg) dx\quad \tfor u,v\in  H_0^1(\Omega), 
\end{equation}
where 
\beqs
D :=\begin{pmatrix}\frac{1}{\gamma_{1}^2(x_1)}&0\\
0& \frac{1}{\gamma_{2}^2(x_2) }\end{pmatrix}
\quad
\tand\quad
\beta := \left(\frac{\gamma_{1}'(x_1)}{\gamma_{1}^3(x_1)},
 \frac{\gamma_{2}'(x_2)}{\gamma_{2}^3(x_2)} \right)^T
 \eeqs
with $\gamma_{1}(x_1) := 1+ig_{1}'(x_1)$ and $\gamma_{2}(x_2) := 1+ig_{x_2}'(x_2)$.

We abbreviate $V_h(\Omega)$ to $V_h$, define  $\cA_h:V_h\to V_h'$   
 by 
\beq\label{eq:morning0}
(\cA_h u_h)(v_h) =a(u_h,v_h)\quad \tfa u_h,v_h \in V_h 
\eeq
and denote the functional on the right-hand side of \eqref{eq:global_vari} by $f_h \in V_h'$,  
The variational problem \eqref{eq:global_vari} then becomes $\cA_h u_h= f_h$, and is equivalent to 
a linear system in $\mathbb{C}^m$, where $m$ is the dimension of $V_h$.

\paragraph{The Galerkin approximations of the local problems.}

Let $a_j(\cdot,\cdot)$ be the sesquilinear form obtained by multiplying $P_\newtheta^j$ be a test function and integrating by parts; when $d=2$ the definition of $a_j(\cdot,\cdot)$ is the same as the definition of $a(\cdot,\cdot)$ \eqref{eq:sesqui} except that  $g_\ell'$, $\ell=1,2,$ is replaced by $g_{\ell,j}'$.

Let $V_{h,j}: = \{ v_h\vert_{\Omega_j}: v_h \in V_h\} \cap H_0^1(\Omega_j)$ and let $\cA_{h,j}: V_{h,j} \to V_{h,j}'$ be the operator associated with $a_j(\cdot,\cdot)$; i.e.,
\beq\label{eq:morning1}
(\cA_{h,j} u_{h,j})(v_{h,j}) =a_j(u_{h,j},v_{h,j})\quad \tfa u_{h,j},v_{h,j} \in V_{h,j}.
\eeq
We assume that that  all the operators $\cA_h: V_h \rightarrow V_h'$ and $\cA_{h,j}: V_{h,j} \to V_{h,j}',$ $j=1,\ldots, N$ are  invertible. 
When the Helmholtz problem on $\Omega$ is nontrapping,
 we expect this to be true when $(hk)^{2p} k$ is sufficiently small; this threshold was famously identified for the 1-d Helmholtz $h$-FEM in \cite{IhBa:97}, and the latest results for $d\geq 2$ proving existence of the Galerkin solution (along with error bounds) under this threshold are given in \cite{GS3}. 
The results of \cite{GS3}, however, are not immediately applicable when $\Omega$ is a hyperrectangle because of the low regularity of $\partial \Omega$. 

\paragraph{Restriction and prolongation operators.}

We define the prolongation operator $\cR_{h,j}^T: V_{h,j}\to V_h$ by, for $v_h \in V_{h,j}$,
$$
\cR_{h,j}^T v_{h,j}(x) = \left\{
\begin{array}{ll}
v_{h,j}(x), &\text{ when $x$ is a node of $\Omega_j$},\\
0, &\text{ when $x$ is a node of $\Omega$ that is not a node of $\Omega_j$}.
\end{array}
\right.
$$
Let $\cR_{h,j}:V_h' \to  V_{h,j}' $ be the adjoint of $\cR_{h,j}^\top$, i.e.,
\beq\label{eq:morning2}
\langle \cR_{h,j} f, v_{h,j}\rangle_{V_{h,j}' \times V_{h,j}} = \langle f, \cR_{h,j}^T v_{h,j}\rangle_{V_{h}' \times V_{h}}
\quad \tfa f\in V_h'  \,\, \text{and} \,\, v_{h,j}\in V_{h,j}.
\eeq
We also define a weighted prolongation operator $\widetilde{\cR}_{h,j}^T: V_{h,j}\to V_h$,  involving the partition-of-unity function $\chi_j$,  by
$$
\widetilde{\cR}_{h,j}^T v_{h,j} =  {\cR}_{h,j}^T \big( \Pi_h (\chi_jv_{h,j})\big)\quad \tfa  v_{h,j} \in V_{h,j},
$$
where $\Pi_h$ is nodal interpolation onto $V_h$.  
Then, for all $w_h \in V_h$ and all finite element nodes $x \in \Omega$, 
  $$   \sum_{j=1}^N \widetilde{\cR}^T_{h,j} (w_h\vert_{\Omega_j}) (x)  = \sum_{j: x \in \Omega_j}   \big( \Pi_h
  (\chi_j w_h \vert_{\Omega_j})\big)(x) = \sum_{j=1}^N    \chi_j(x) w_{h}(x)  = w_{h}(x);   $$
i.e.,
\begin{align} \label{eq:prores}
 \sum_{j=1}^N \widetilde{\cR}^T_{h,j} (w_h\vert_{\Omega_j})  =  w_h  \quad\tfa  w_h \in V_h. 
 \end{align} 

\paragraph{The discrete parallel overlapping Schwarz method.}

Let $\mathfrak c_{h,j}^n$ be the finite-element approximation to the local corrector $\mathfrak c_j^n$ \eqref{eq:local_corrector}; i.e.,
given $u^n_h\in V_h$, $\mathfrak c_{h,j}^n \in V_{h,j}$ is the solution of 
\beqs
a_j \big(\mathfrak c_{h,j}^n, v_{h,j} \big) = \big\langle f, \cR^T_{h,j} v_{h,j}\big\rangle - a\big(u^n_h, \cR^T_{h,j} v_{h,j} \big) \quad\tfa v_{h,j}\in V_{h,j}.
\eeqs
Thus, by \eqref{eq:morning0}, \eqref{eq:morning1}, and \eqref{eq:morning2},
\beqs
\cA_{h,j} \mathfrak c_{h,j}^n = \cR_{h,j} (f - \cA_h u_h^n) \in V_{h,j}'
\quad\text{ and so } \quad
\mathfrak c_{h,j}^n =\cA_{h,j}^{-1} \cR_{h,j} (f - \cA_h u_h^n) \in V_{h,j}.
\eeqs
In analogy with the second equation in \eqref{eq:local_corrector} and recalling \eqref{eq:prores}, we then define
$$
u^{n+1}_h := u^n_h + \sum_{j=1}^N \widetilde{\cR}_j^T\cA_{h,j}^{-1}\cR_j (f - \cA_h u^n_h)\in V_h;
$$
i.e., 
the method 
 is the preconditioned Richardson iteration
\begin{align} \label{eq:PC}   u^{n+1}_{h} := u^n_{h} + \cB_h^{-1} (f - \cA_h u^n_h)\in V_h
 \quad \text{ with } \quad \cB_h^{-1} :=  \sum_j \widetilde{\cR}_{h,j}^\top \cA_{h,j}^{-1}\cR_{h,j}.
\end{align}
The preconditioner $\cB_h^{-1}: V_h'\to V_h$ is of the form of the RAS preconditioner  \cite{CaSa:99}, with the crucial fact that  
it involves the weighted prolongation operator $\widetilde{\cR}_{h,j}^T$ instead of ${\cR}_{h,j}^T$ (compare, e.g., \cite[Definitions 1.12 and 1.13]{DoJoNa:15}). 
The classical `Optimised' RAS (known as ORAS \cite{StGaTh:07}, \cite[\S2.3.2]{DoJoNa:15}) for Helmholtz uses the impedance boundary condition on the subdomain boundaries (see, e.g., \cite{GGGLS1, GoGrSp:23}), but here we have PML at the subdomain boundaries. We note that numerical results on this version of RAS when the PML scaling function is proportional to $k^{-1}$ 
were given in \cite{BoBoDoTo:22}.

\paragraph{GMRES.}

The main  results of the paper concern the overlapping Schwarz methods described in \S\ref{sec:parallel} and \S\ref{sec:sequential} as fixed-point iterations (on the continuous level).
We also present numerical results for the  generalised minimal residual (GMRES) algorithm \cite{SaSc:86} acting as an acceleration of the parallel fixed point iteration in \eqref{eq:PC}. 
Since the iteration counts of the sequential method used as a fixed-point iteration are relatively low, we do not consider GMRES acceleration for this iteration here.

We therefore briefly recall the definition of GMRES applied to the abstract  linear system 
$\matrixC \bx = \bfd$ where $\matrixC$ is a nonsingular complex matrix. For GMRES applied to the parallel method above, $\matrixC$ is the matrix corresponding to the operator $\cB_h^{-1} \cA_h: V_h\to V_h$ and $\bfd$ is the vector corresponding to $\cB_h^{-1} f_h \in V_h$.
Given initial guess $\bfx^0$, 
let $\bfr^0 := \bfd- \matrixC \bfx^0$ and 
$$  \mathcal{K}^n(\matrixC, \bfr^0) := \mathrm{span}\big\{\matrixC^j \bfr^0 : j = 0, \ldots, n-1\big\}.$$
For $n\geq 1$, define   $\bfx^n$  to be  the unique element of $\cK^n$ satisfying  the   
 minimal residual  property: 
$$ \ \Vert \bfr^n \Vert_2 := \Vert \bfd - \matrixC \bfx^n \Vert_2 \ = \ \min_{\bfx \in \mathcal{K}^n(C, \br^0)} \Vert {\bfd} - {\matrixC} {\bfx} \Vert_2. $$

\subsection{Common set-up for  the experiments}

\paragraph{The Helmholtz problem.}

We work in $\mathbb{R}^2$ with 
$\Omega_{\rm int} :=(0,1)^2$
and PML width $\kappa = 1/40$, so that $\Omega:= (-1/40, 1+ 1/40)^2$. 
Note that, for $k=100,150,200, 250, 300, 350$, the number of wavelengths in the PML is $(1/40)/(2 \pi/k)  = 0.40 , 0.60, 0.80, 0.99, 1.19, 1.39$, respectively.
We take the PML scaling function as $\fPML(x) = 5000 x^3/3$, and the right-hand side of the PDE \eqref{eq:PDE} to be $f(x) = J_0(k \vert x - x_0\vert )$
when $x_0=(0.5,0.5)$. Note that taking the right-hand side to be a Helmholtz solution, similar to for a scattering problem, ensures that all of the mass of the solution in phase space is oscillating at frequency approximately $k$, ensuring that the solution propagates.

\paragraph{The domain decomposition.}

We consider both strip and checkerboard decompositions of $\Omega$.
The overlapping decomposition $\{\Omega_{{\rm int},j}\}_{j=1}^N$ \eqref{eq:Omegaintj} is taken to have overlap $\delta=1/40$, with then each overlapping subdomain extended by PML layers of width $\kappa=1/40$.

The smooth  partition of unity functions are constructed with respect to a (still overlapping) cover of $\Omega_{{\rm int}} $ consisting of subsets of $\Omega_{{\rm int},j}$,  so that each PoU function $\chi_j$ vanishes in a neighbourhood of $\Omega_{{\rm int},j}$. In more detail, 
for strip decompositions with $\Omega_{{\rm int}, j} = (a^j,b^j)\times(0,1)$, we let
$$
\widetilde{\chi}_j(x) =
\begin{cases}
 e^{-\frac{(b^j-a^j)^2}{4(x-a^j-0.3\delta)(b^j - 0.3\delta -x)}},& x \in(a^j+0.3\delta,b^j -0.3\delta),\\
 0, & \text{elsewhere},
\end{cases}
\quad j=2,\ldots, N-1,
$$
and 
$$
\widetilde{\chi}_1(x) =
\begin{cases}
 e^{-\frac{b^1-a^1}{2(b^1 - 0.3\delta -x)}},& x <b^1 -0.3\delta,\\
 0, & \text{elsewhere}.
\end{cases}
\quad 
\widetilde{\chi}_N(x) =
\begin{cases}
 e^{-\frac{b^1-a^1}{2(x-a^N - 0.3\delta )}},& x >a^N +0.3\delta,\\
 0, & \text{elsewhere}.
\end{cases}
$$
The PoU functions $\chi_j$ are then defined by 
$
\chi_j := \widetilde{\chi}_j/(\sum_{j=1}^N\widetilde{\chi}_j).
$
For checkerboard decompositions, the PoU functions are created by normalising the Cartesian product of $\widetilde{\chi}_j(x)\widetilde{\chi}_\ell(y)$.

We use the abbreviations  RAS-PML and RMS-PML to denote the parallel (additive) and sequential (multiplicative) 
methods, respectively.

\paragraph{The finite-element method.}

The finite-dimensional subspace $V_h$ is taken to be standard Lagrange finite elements of degree $p=2$ on uniform meshes of $\Omega$, 
with $h$ chosen as the largest number $\leq k^{-1.25}$ such that $1/h$ is an integer (recall from the discussion immediately below \eqref{eq:morning1} 
that we expect the relative-error of the finite-element solutions to be bounded uniformly in $k$ under this choice of $h$). 
We write the resulting linear system as 
$\mathsf{A u =f}.$

\subsection{Experiment I:~$\coeffc\equiv 1$}
\label{subsec:ExptI} 

Figure \ref{fig:convergence_rate} plots the rate that the residual decreases as a function of $k$ for 
\bit
\item a strip decomposition, with RAS-PML (Plot (a) of the figure) and RMS-PML with forward-backward sweeping (Plot (b)),
\item a checkerboard decomposition, with RAS-PML (Plot (c)) and RMS-PML with the exhaustive sequence of four orderings constructed via Example \ref{ex:exhaust2} (Plot (d)). 
\eit
In all cases the rate of residual reduction increases as $k$ increases (as  expected from Theorems \ref{thm:strip}-\ref{thm:sweep} and \ref{thm:gen}).

\begin{figure}[h!]
    \centering
           \begin{subfigure}[t]{0.5\textwidth}
        \centering
        \includegraphics[width=.95\textwidth]{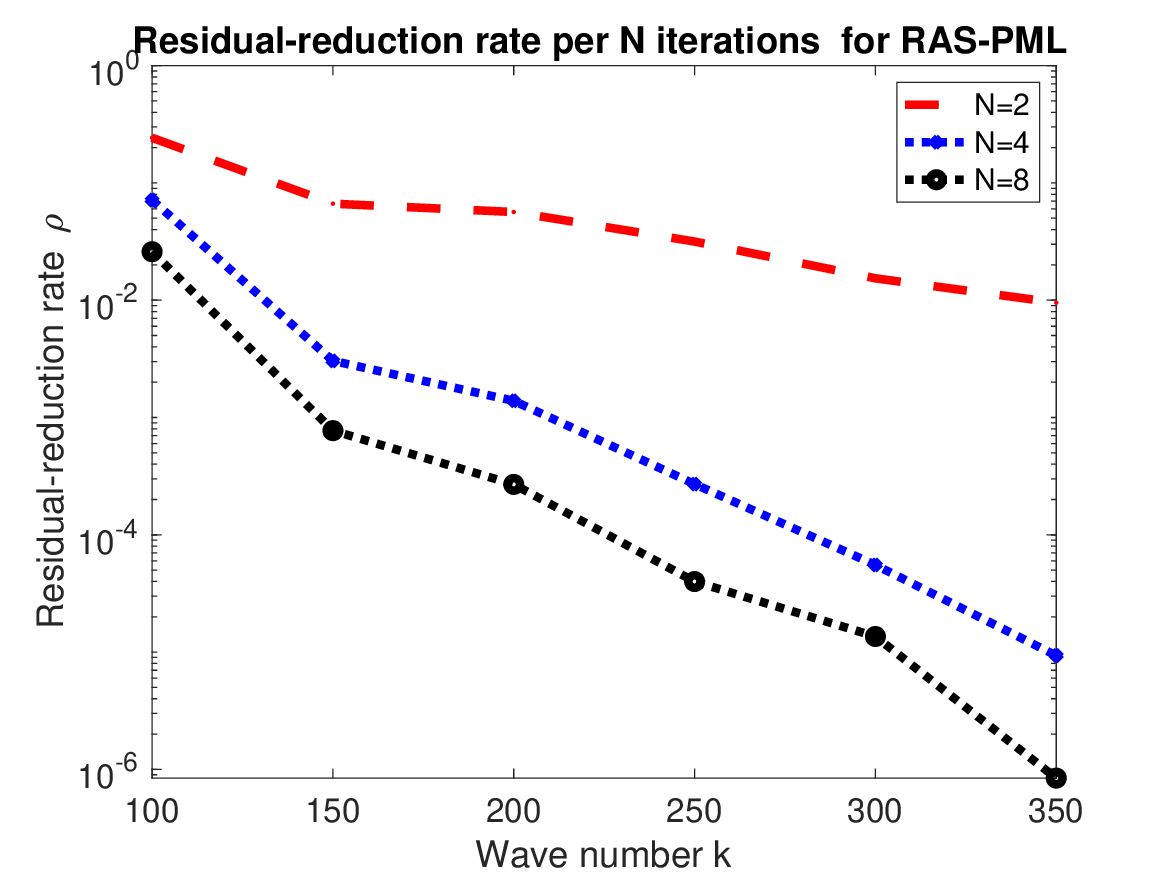}
        \caption{strip, RAS-PML: $\rho=\frac{\|\mathsf{A(u-u^N)}\|_{\ell^2}}{\|\mathsf{A(u-u^0)}\|_{\ell^2}}$}
         \end{subfigure}%
     \hfill
    \begin{subfigure}[t]{0.5\textwidth}
        \centering
        \includegraphics[width=.95\textwidth]{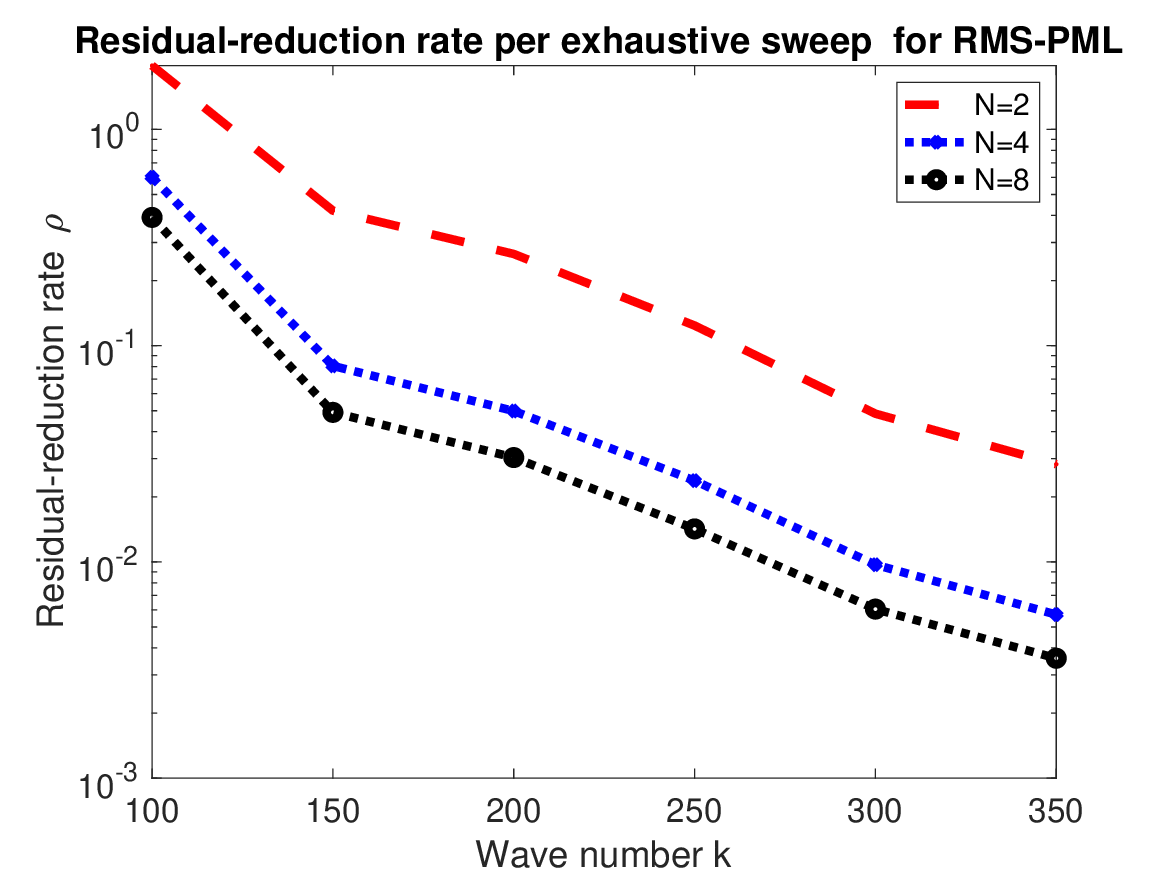}
        \caption{strip, RMS-PML: $\rho = \frac{\|\mathsf{A(u-u^2)}\|_{\ell^2}}{\|\mathsf{A(u-u^0)}\|_{\ell^2}}$}
    \end{subfigure}%
         \hfill
    \begin{subfigure}[t]{0.5\textwidth}
        \centering
        \includegraphics[width=.95\textwidth]{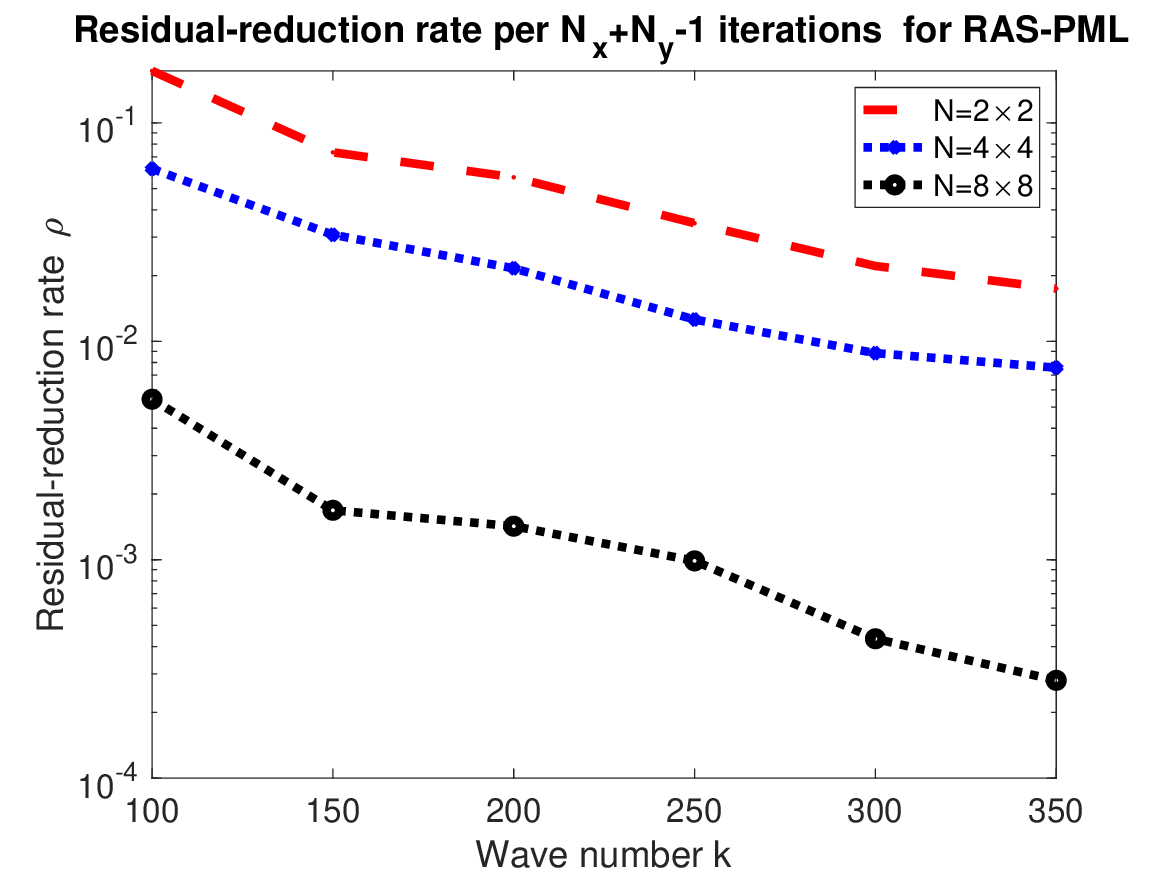}
        \caption{\footnotesize checkerboard, RAS-PML:  $\rho=\frac
        {\|\mathsf{A(u-u^{N_x+N_y-1})}\|_{\ell^2}}{\|\mathsf{A(u-u^0)}
        \|_{\ell^2}}$}
    \end{subfigure}%
         \hfill
    \begin{subfigure}[t]{0.5\textwidth}
        \centering
        \includegraphics[width=.95\textwidth]{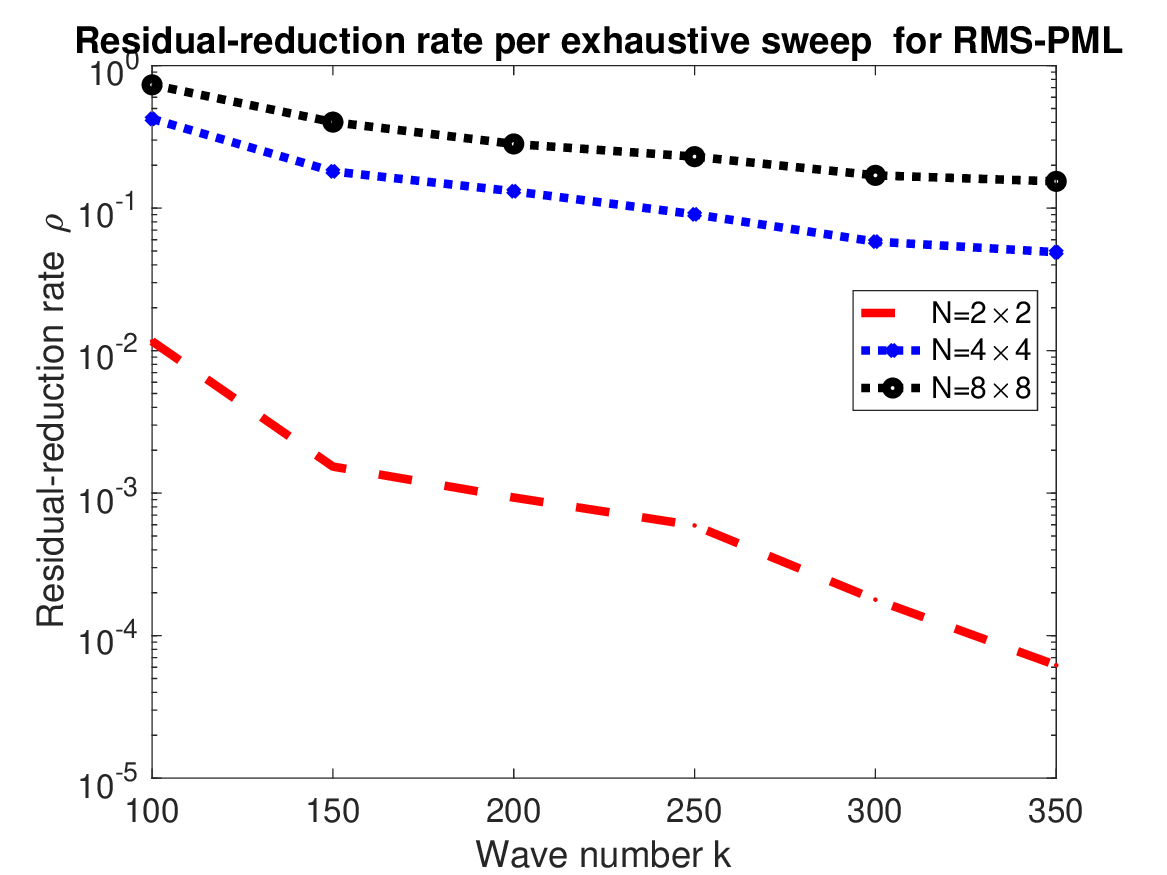}
        \caption{checkerboard, RMS-PML: $\rho=\frac
        {\|\mathsf{A(u-u^4)}\|_{\ell^2}}{\|\mathsf{A(u-u^0)}\|_{\ell^2}}$}
    \end{subfigure}%
   \caption{ The rate of reduction of the residual against $k$ }\label{fig:convergence_rate}
\end{figure}

Figures \ref{fig:his_ras_centre}, \ref{fig:his_rms},  \ref{fig:his_ras_checker_centre}, and \ref{fig:his_rms_checker} present the convergence histories for RAS-PML and RMS-PML for different $k$s and different domain decompositions. 
We highlight that, for the strip decompositions with $N=4$ and $N=8$, 
Figure \ref{fig:his_ras_centre} shows that the residual for RAS-PML on a strip undergoes a significant reduction after every $N$ iterations, as expected from Theorem \ref{thm:strip}.

\red{
\bre
The following is an heuristic explanation of why (i) the residual-reduction rate is smaller with RAS-PML than with RMS-PML and (ii) why this rate improves as $N$ increases more with RAS-PML than RMS-PML. For simplicity, we discuss a strip domain.
\begin{enumerate}
\item As recalled in \S\ref{sec:context}, if the subdomain problems are solved with the exact DtN map as a boundary condition then, after $N$ iterations of the parallel method or one exhaustive sweep of the sequential method, one obtains the exact solution; i.e. the error will be zero.
\item Since we use PML as an approximation to the exact DtN map, each subdomain solve will produce an error, $\epsilon$, from this approximation. 
\item RMS-PML involves only two solves in each domain and so can, at best, approximate the exact DtN map like $\epsilon ^2$. Whether this actually happens depends on the location of the initial data relative to the sweep order.
\item RAS-PML involves $N$ solves on each domain, and so, in principle, can approximate the exact DtN map with error $\epsilon^N$, but, again, only for certain data. 
\item In both cases, tracking the precise accumulation of error involves knowing the location (both in physical space and phase space) of the source. 
%
\end{enumerate}
These high-frequency considerations indicate that when $N=2$ the errors in RAS-PML and RMS-PML should be similar. Indeed, 
Plots (a) and (b) in Figure \ref{fig:convergence_rate} show that 
at $k=350$, $\rho \approx 10^{-1.8}$ for RMS-PML and $\rho \approx 10^{-2}$ for RAS-PML.
\ere
}

\begin{figure}[h!]
    \centering
           \begin{subfigure}[t]{0.33\textwidth}
        \centering
        \includegraphics[width=.95\textwidth]{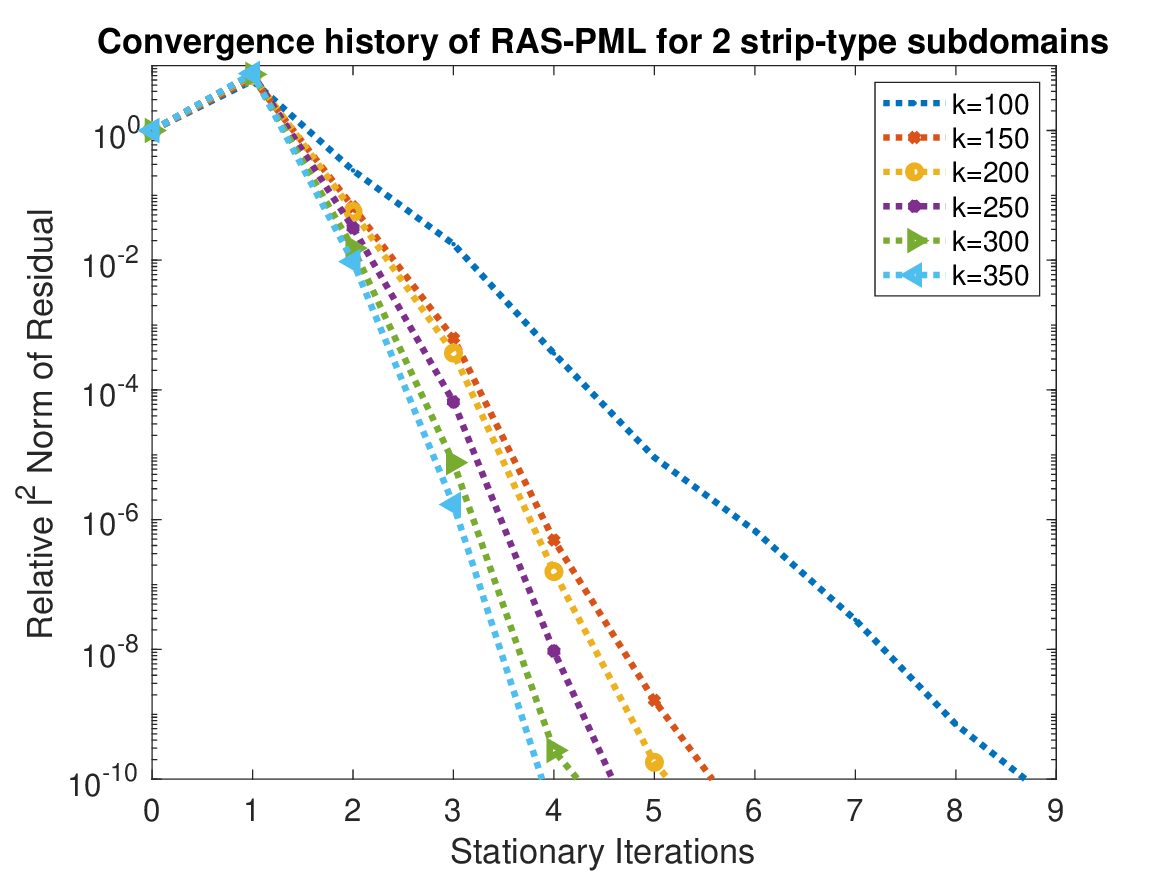}
        \caption{2 strip-type subdomains}
         \end{subfigure}%
     \hfill
    \begin{subfigure}[t]{0.33\textwidth}
        \centering
        \includegraphics[width=.95\textwidth]{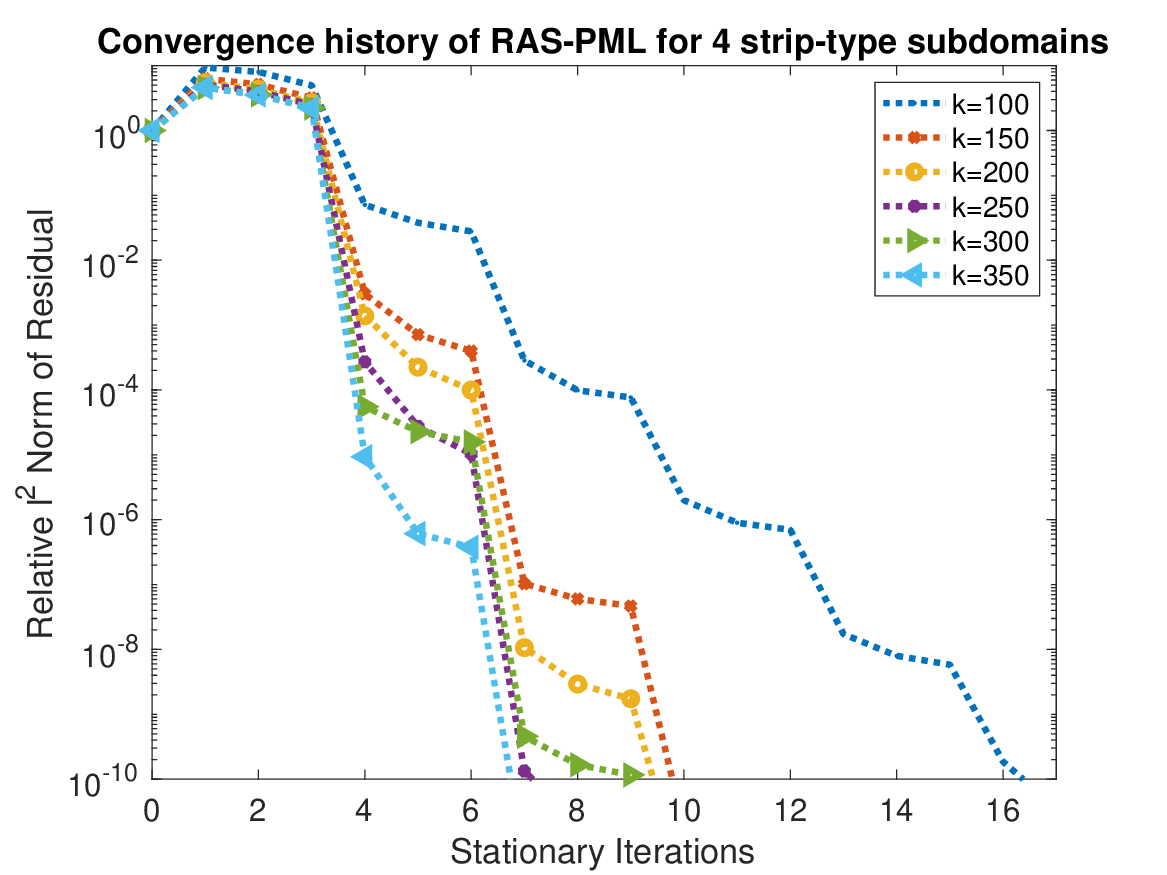}
        \caption{4 strip-type subdomains}
    \end{subfigure}%
         \hfill
        \begin{subfigure}[t]{0.33\textwidth}
        \centering
        \includegraphics[width=.95\textwidth]{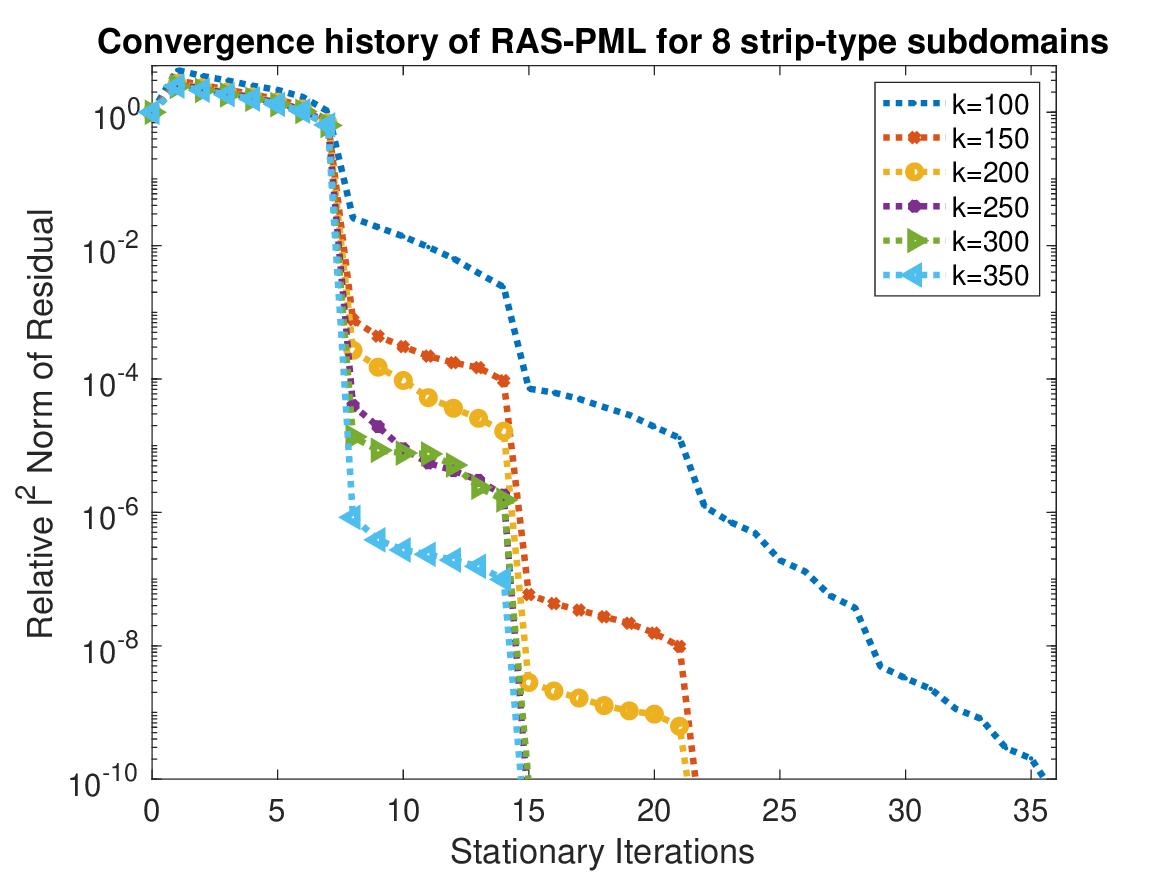}
        \caption{8 strip-type subdomains}
    \end{subfigure}%
   \caption{ The convergence histories of RAS-PML for strip decompositions}\label{fig:his_ras_centre}
\end{figure}

\begin{figure}[h!]
    \centering
           \begin{subfigure}[t]{0.33\textwidth}
        \centering
        \includegraphics[width=.95\textwidth]{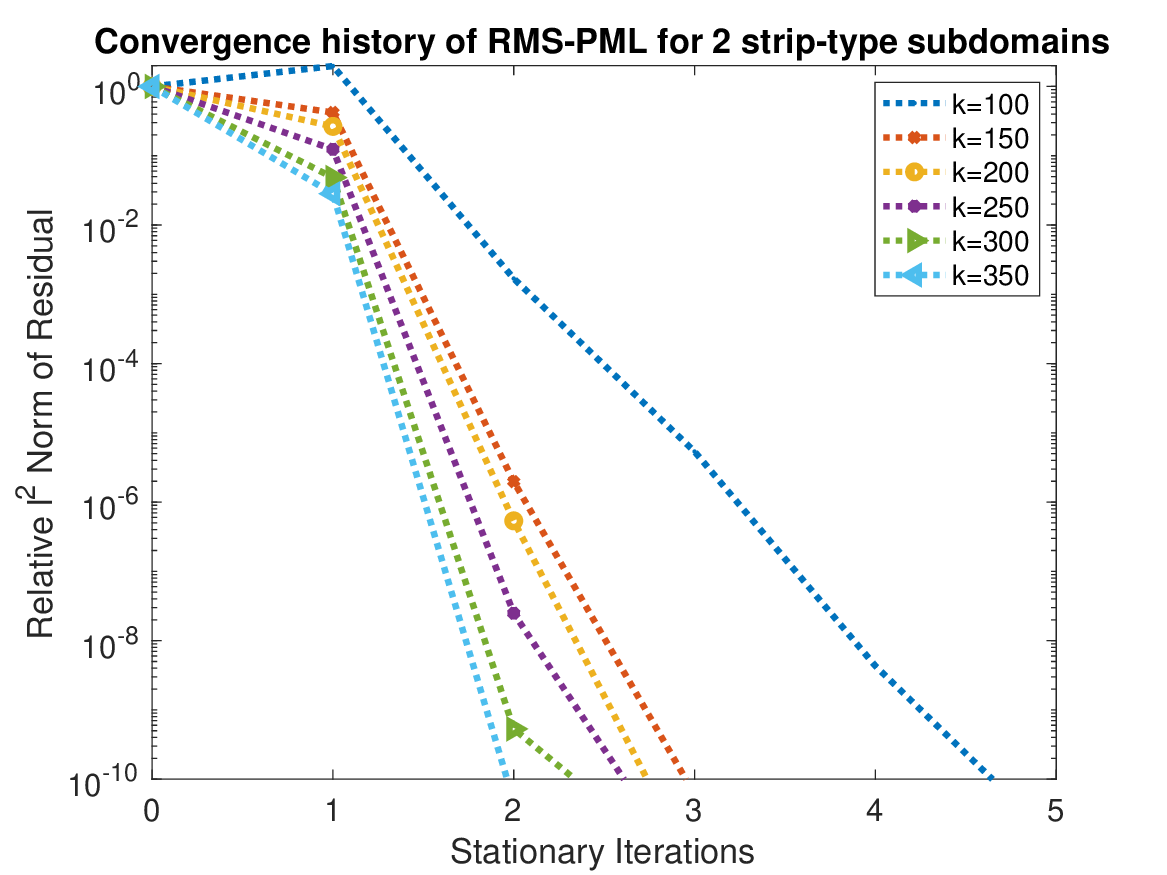}
        \caption{2 strip-type subdomains}
         \end{subfigure}%
     \hfill
    \begin{subfigure}[t]{0.33\textwidth}
        \centering
        \includegraphics[width=.95\textwidth]{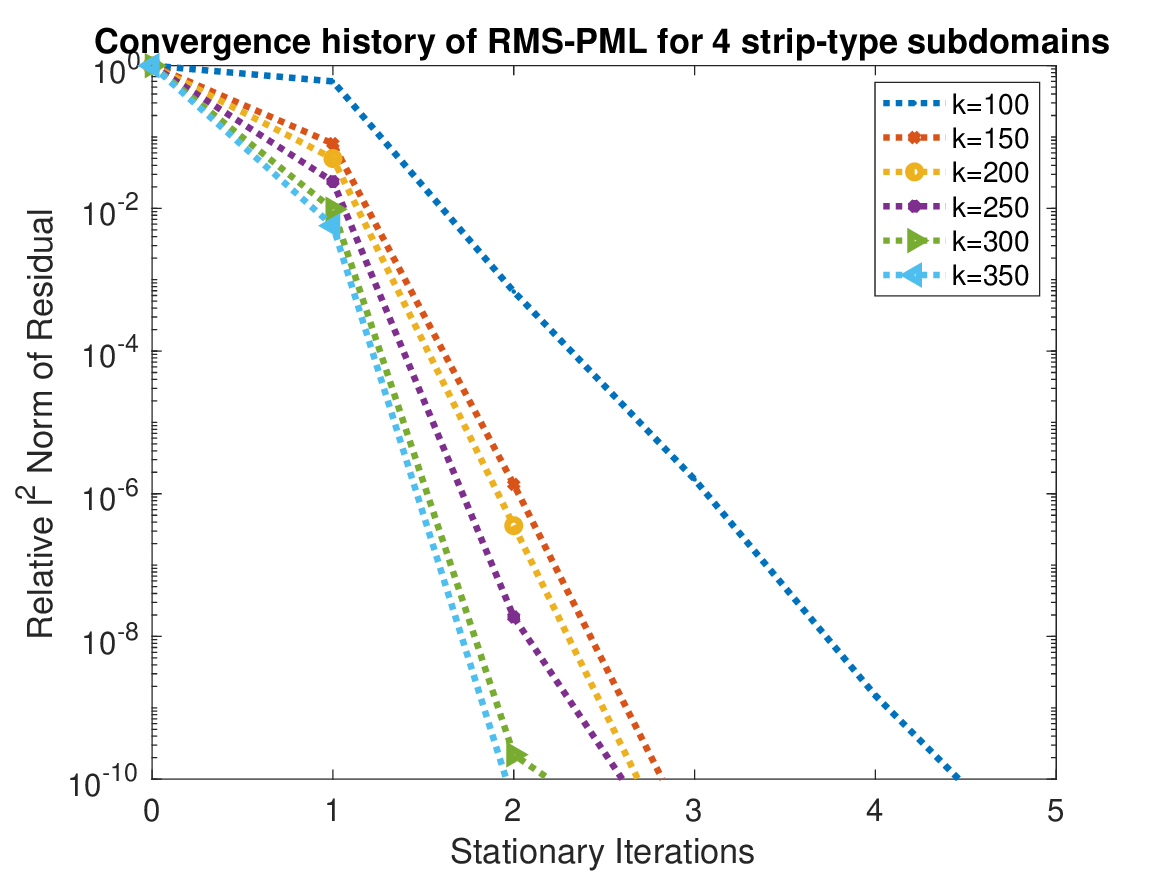}
        \caption{4 strip-type subdomains}
    \end{subfigure}%
         \hfill
        \begin{subfigure}[t]{0.33\textwidth}
        \centering
        \includegraphics[width=.95\textwidth]{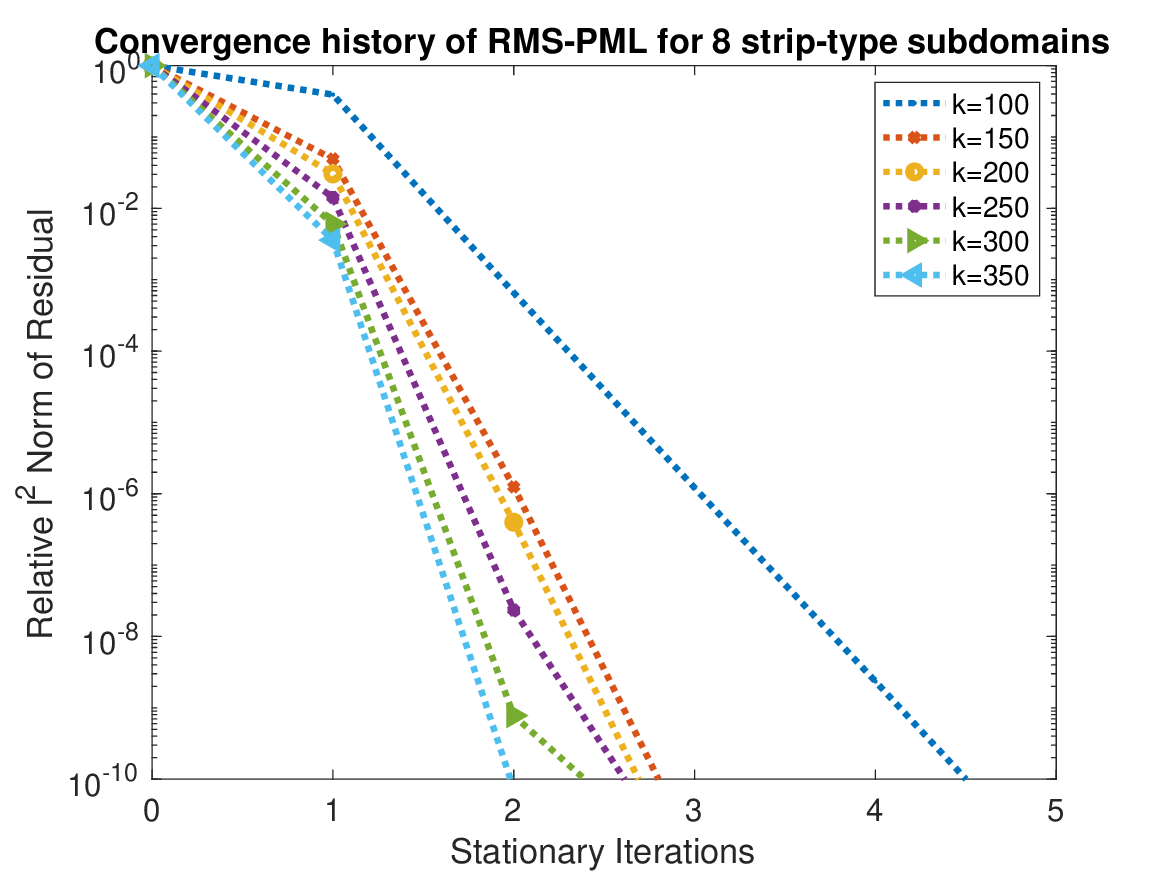}
        \caption{8 strip-type subdomains}
    \end{subfigure}%
   \caption{ The convergence histories of RMS-PML for strip decompositions}\label{fig:his_rms}
\end{figure}

\begin{figure}[h!]
    \centering
           \begin{subfigure}[t]{0.33\textwidth}
        \centering
        \includegraphics[width=.95\textwidth]{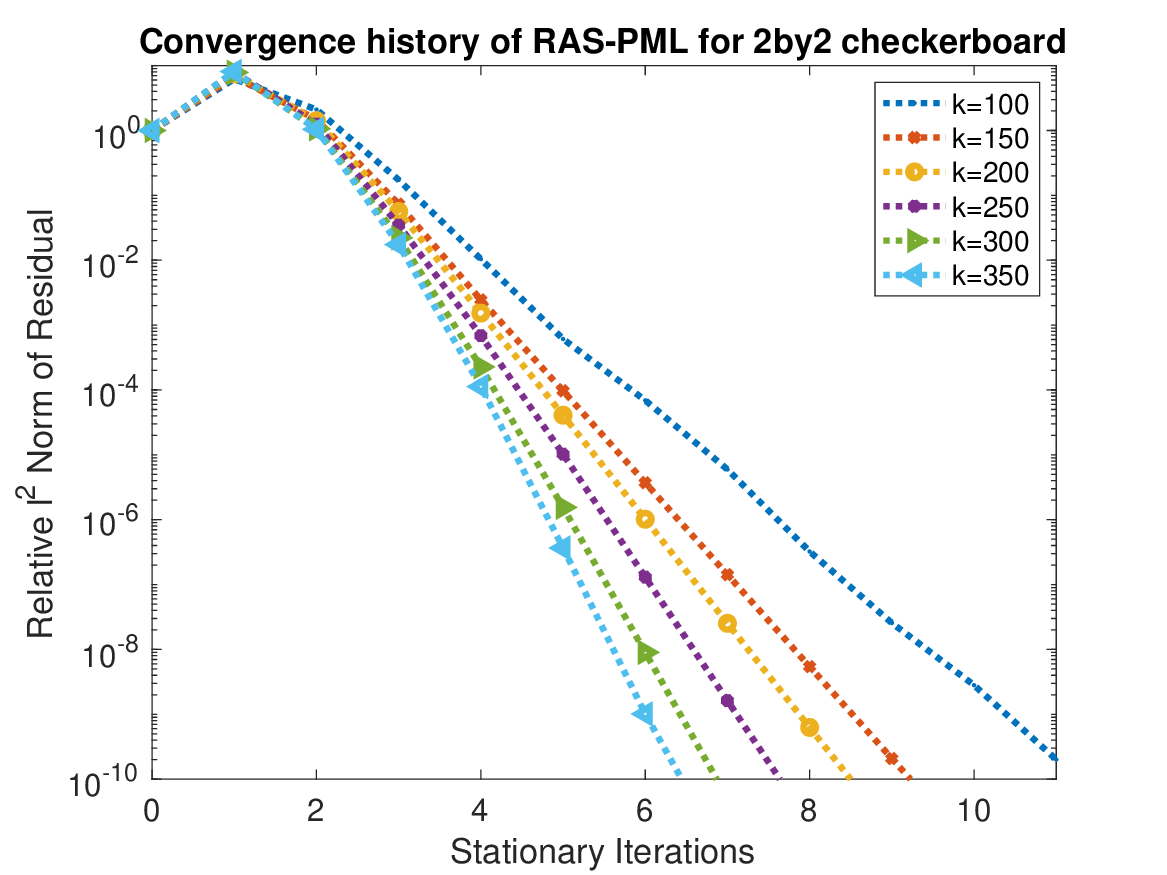}
        \caption{2by2 subdomains}
         \end{subfigure}%
     \hfill
    \begin{subfigure}[t]{0.33\textwidth}
        \centering
        \includegraphics[width=.95\textwidth]{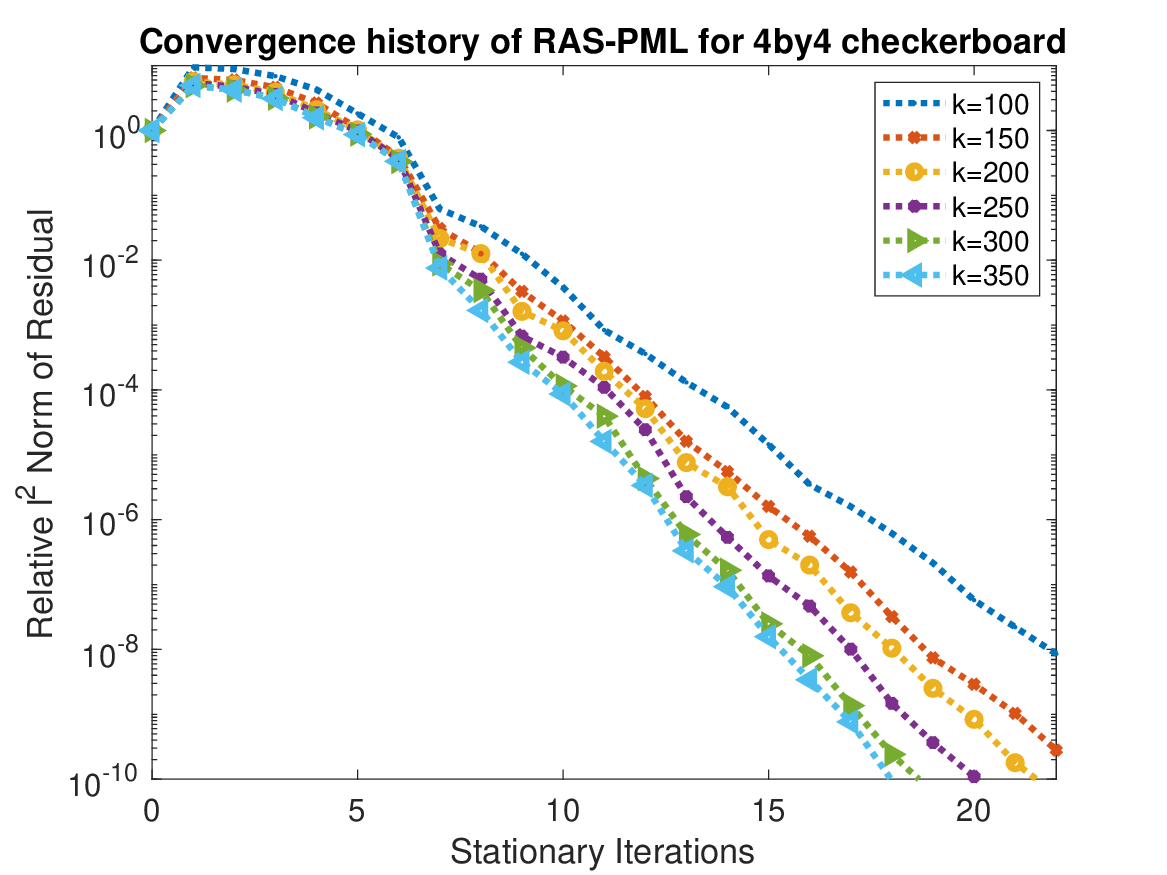}
        \caption{4by4 subdomains}
    \end{subfigure}%
         \hfill
        \begin{subfigure}[t]{0.33\textwidth}
        \centering
        \includegraphics[width=.95\textwidth]{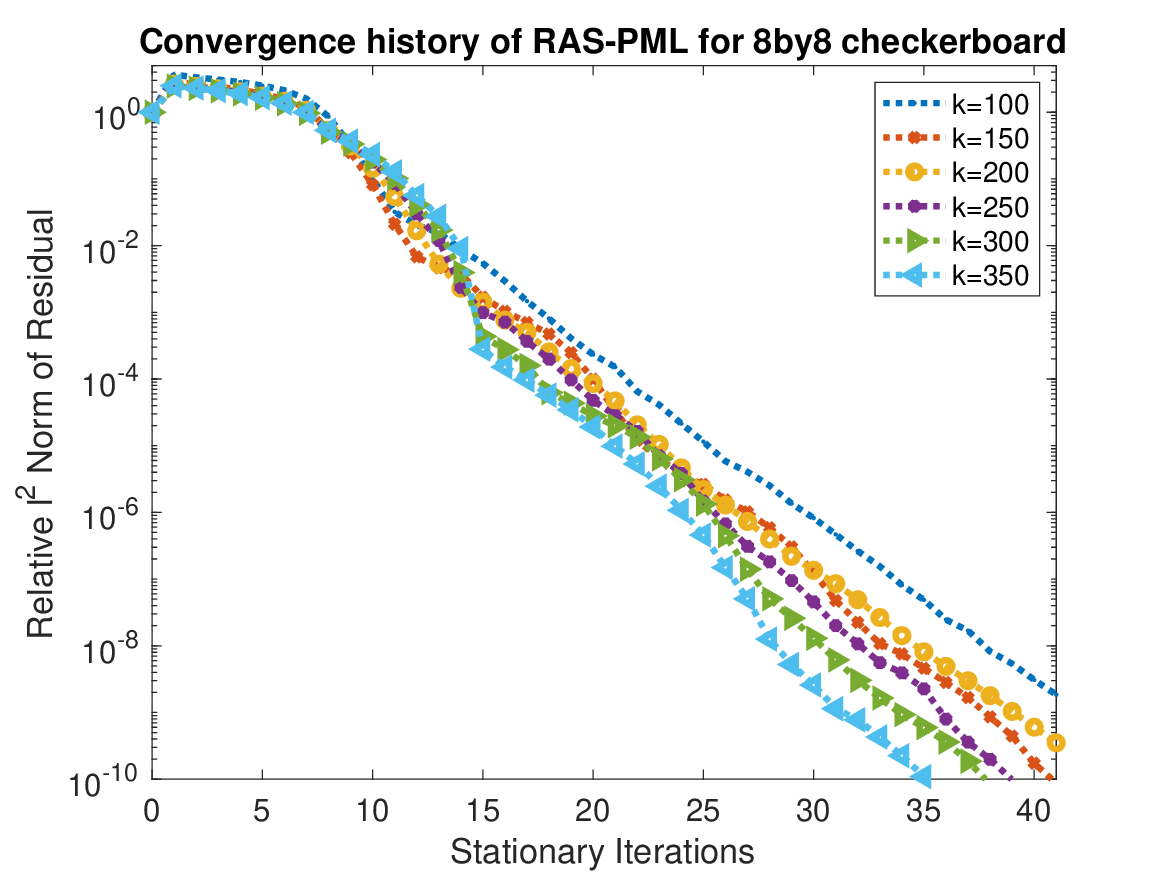}
        \caption{8by8 subdomains}
    \end{subfigure}%
   \caption{ The convergence histories of RAS-PML for checkerboard decompositions}\label{fig:his_ras_checker_centre}
\end{figure}

\begin{figure}[h!]
    \centering
           \begin{subfigure}[t]{0.33\textwidth}
        \centering
        \includegraphics[width=.95\textwidth]{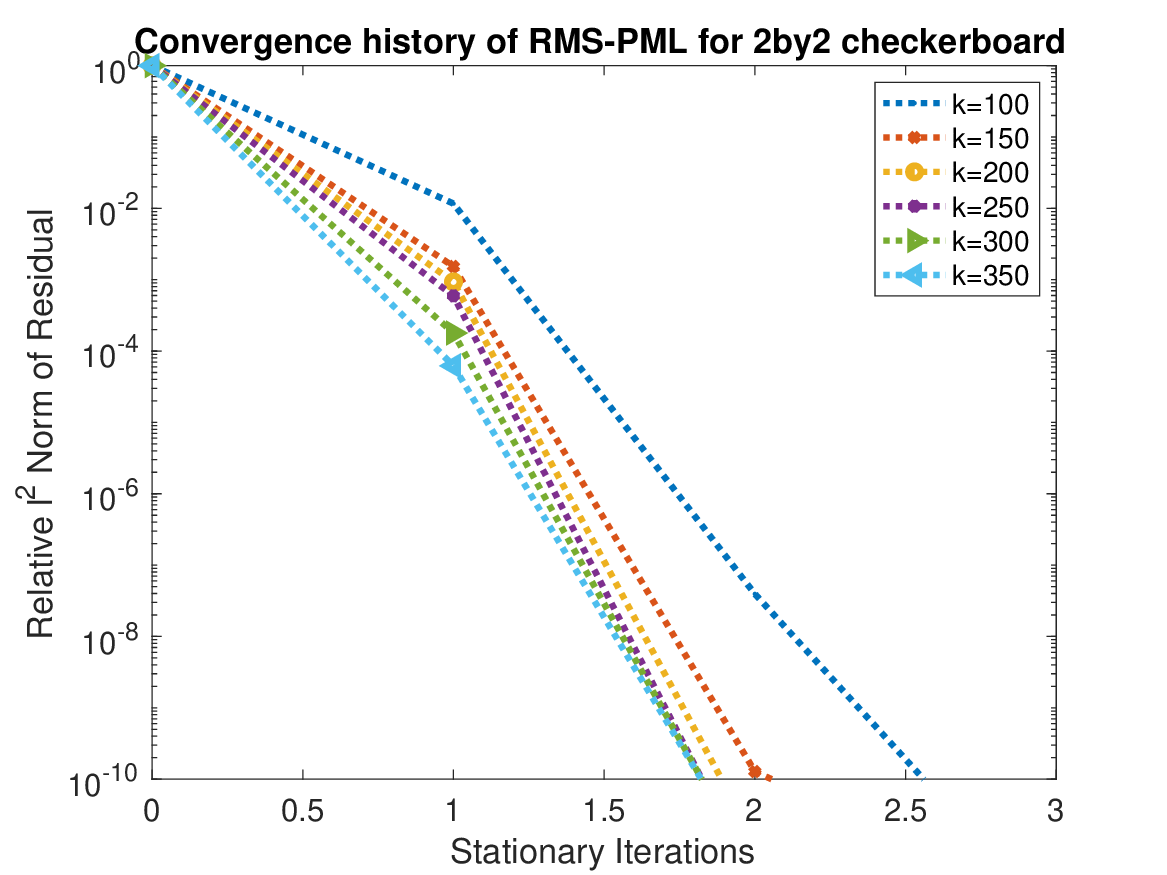}
        \caption{2by2 subdomains}
         \end{subfigure}%
     \hfill
    \begin{subfigure}[t]{0.33\textwidth}
        \centering
        \includegraphics[width=.95\textwidth]{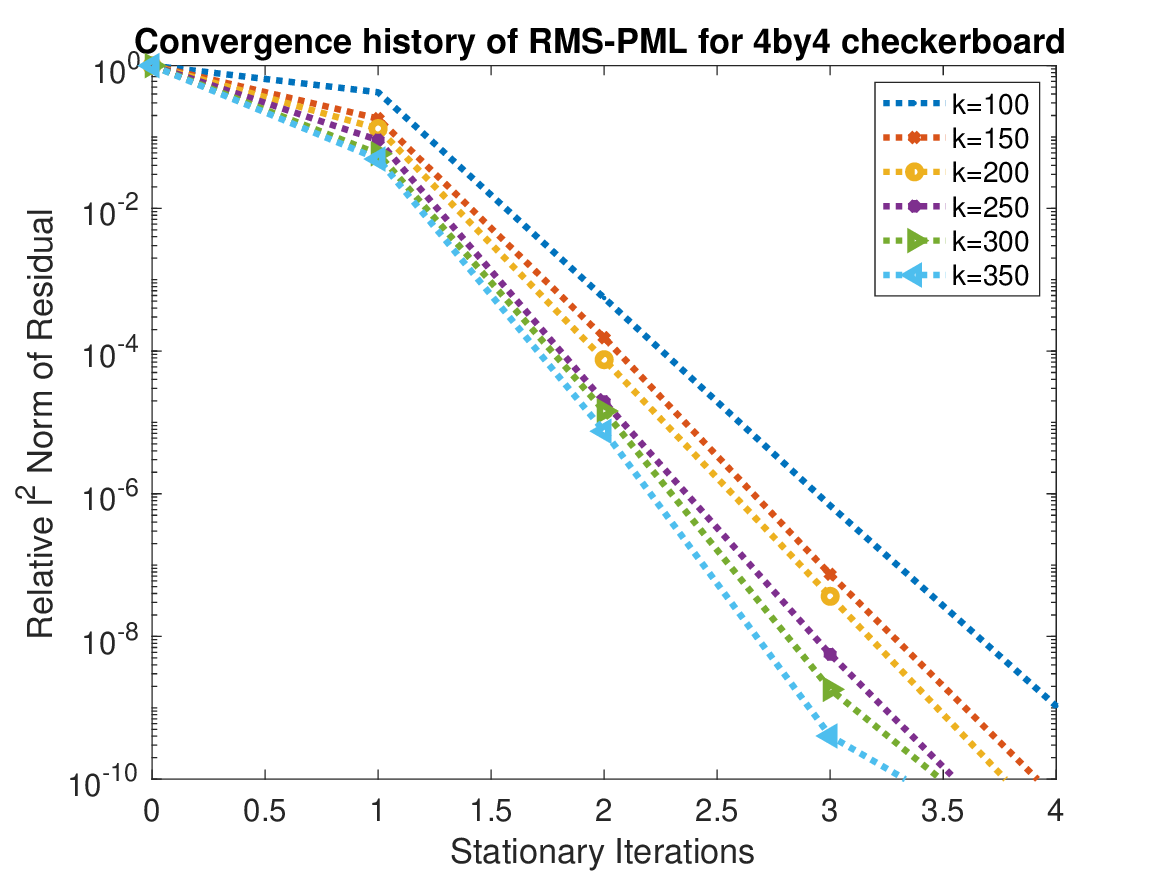}
        \caption{4by4 subdomains}
    \end{subfigure}%
         \hfill
        \begin{subfigure}[t]{0.33\textwidth}
        \centering
        \includegraphics[width=.95\textwidth]{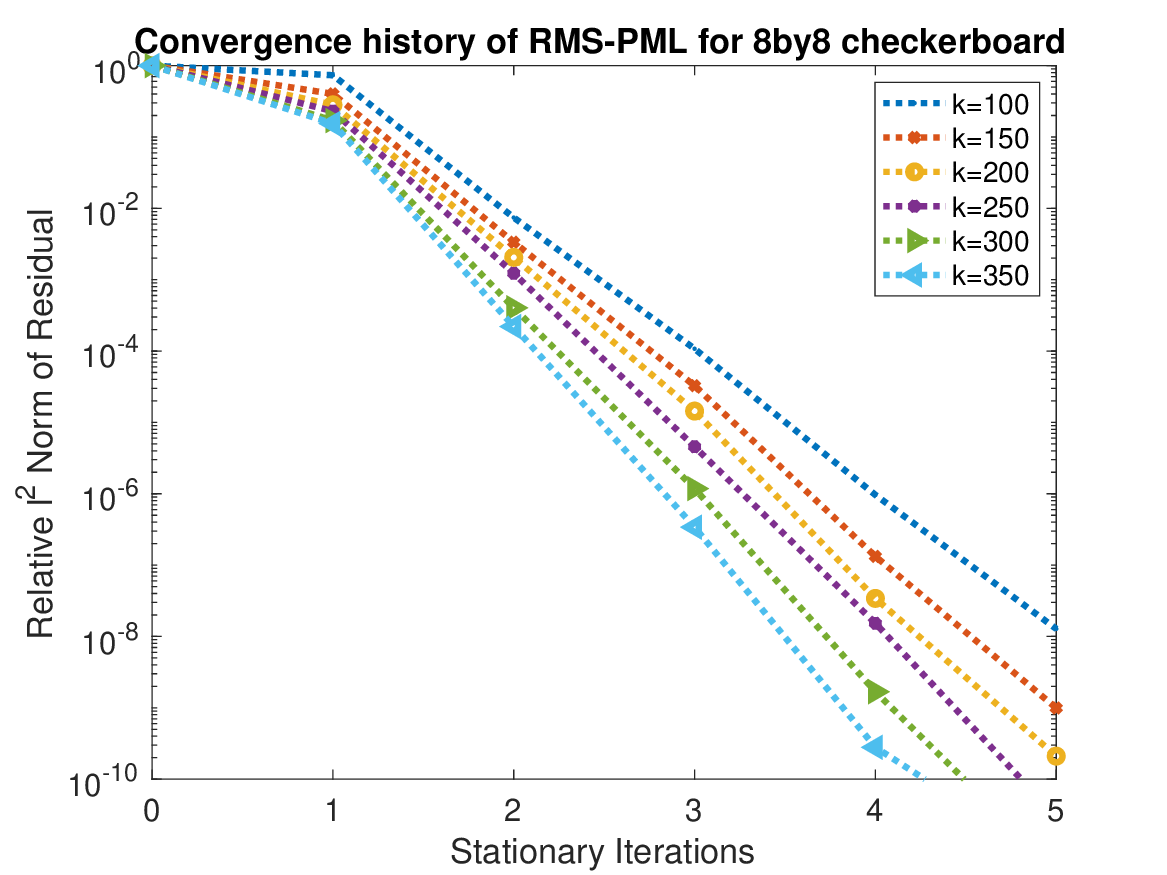}
        \caption{8by8 subdomains}
    \end{subfigure}%
   \caption{ The convergence histories of RMS-PML for checkerboard decompositions}\label{fig:his_rms_checker}
\end{figure}

\newpage

\subsection{Experiment II:~variable $\coeffc$}
\label{subsec:ExptII}

We consider three cases
\begin{enumerate}
\item[Case 1:] $c^{-2}=1$, 
\item[Case 2:]  $c^{-2}\in [2,1]$ decreases linearly from the center to $|x|=0.4$ and equals 1 for $|x|>0.4,$ 
\item[Case 3:]  $c^{-2}\in [0.5,1]$ increases linearly from the center to $|x|=0.4$ and equals 1 for $|x|>0.4.$ 
\end{enumerate}
When $c$ decreases, i.e., $c^{-1}$ increases, as a function of $|x|$, then the problem is nontrapping (see, e.g., \cite[Theorem 2.5]{GrPeSp:19}, \cite[\S7]{GrPeSp:19}). If $c$ increases quickly enough as a function of $|x|$, then the problem is trapping \cite{Ra:71}, but the growth of the solution operator 
through an increasing sequence of $k$s is very sensitive to the particular value of $k$ \cite{LaSpWu:19}, and depends on the data \cite{DaVa:12a}. 
Therefore, in these experiments, we do not seek to solve any trapping problems where the solution operator grows faster with $k$ than for nontrapping problems.

Experiment I showed that, in Case 1 ($c\equiv 1$), the rate of residual reduction  increases with $k$ for all the methods considered. In this experiment, we compare Cases 1-3 by listing 
the fixed-point iteration counts (in brackets, GMRES iterations) to obtain the following relative-residual reduction
  \begin{align} 
  \label{eq:resred} 
\frac  {\|\mathsf{A(u-u^n)}\|_{\ell^2}}{\|\mathsf{A(u-u^0)}\Vert_{\ell^2}}
  < 10^{-6},
  \end{align}  
with Table \ref{tab:variablecRAS-PML1} considering RAS-PML on strips, Table \ref{tab:variablecRMS-PML1} considering RMS-PML on strips, Table \ref{tab:variablecRAS-PML2}
considering RAS-PML on checkerboards, and Table \ref{tab:variablecRMS-PML2} considering RMS-PML on checkerboards.
In all scenarios the number of iterations are very similar for each of Cases 1, 2, and 3.

\begin{table}[h!]
\begin{center}
\begin{tabular}{c|ccc|ccc|ccc}
\hline
Variable $c^{-2}$                         & \multicolumn{3}{c|}{Case 1} & \multicolumn{3}{c|}{Case 2} & \multicolumn{3}{c}{Case 3} \\ \hline
$k$\textbackslash{}$N$ & 2    & 4   & 8   & 2    & 4   & 8   & $2$    & $4$   & $8$  \\ \hline
100                       & 6(6)    & 11(11)       & 23(22)     & 6(6)   & 13(12)& 26(23)& 7(6)& 11(10)& 22(22)\\
150                        & 4(4)& 7(7)& 15(15)& 5(4)& 10(10)& 18(18)& 5(4)&10(8)& 17(15)\\
200                        & 4(4)& 7(7)& 15(15)& 4(4)& 10(10)& 18(15)& 4(4)& 10(8)& 15(15)\\
250                      & 4(4)& 7(7)& 15(15)&4(4)& 10(8)& 18(15)& 4(4)& 8(7)& 14(14)\\
300                        & 4(4)& 7(7)& 15(15)& 4(4)& 10(8)& 15(15)& 4(4)& 8(7)& 15(15)\\
350                      & 4(4)& 5(5)& 8(8)&4(4)& 10(7)& 13(13)& 4(4)& 7(7)& 13(13)\\ \hline
\end{tabular}
\caption{Number of fixed-point iterations (in brackets, GMRES iterations) to obtain the residual reduction \eqref{eq:resred} for  RAS-PML with a strip decomposition}
\label{tab:variablecRAS-PML1}
\end{center}
\end{table}

\begin{table}[h!]
\begin{center}
\begin{tabular}{c|ccc|ccc|ccc}
\hline
Variable $c^{-2}$                         & \multicolumn{3}{c|}{Case 1} & \multicolumn{3}{c|}{Case 2} & \multicolumn{3}{c}{Case 3} \\ \hline
$k$\textbackslash{}$N$ & 2    & 4   & 8   & 2    & 4   & 8   & $2$    & $4$   & $8$  \\ \hline
100                       & 4& 4& 4&4& 3& 3& 4& 4& 4\\
150                        & 3& 3& 3& 3& 3&3& 3&3& 3\\
200                        & 2& 2& 2& 3& 3& 2& 3& 2& 3\\
250                      & 2& 2& 2& 2& 2& 2& 2& 2& 2\\ 300                        & 2& 2& 2& 2& 2& 2& 2& 2& 2\\
350                      & 2& 2& 2& 2& 2& 2& 2& 2& 2\\ \hline
\end{tabular}
\caption{Number of fixed-point iterations to obtain the residual reduction\eqref{eq:resred}
    for  RMS-PML with a strip decomposition, where one iteration equals one forward plus one backward sweep}
\label{tab:variablecRMS-PML1}
\end{center}
\end{table}

\begin{table}[h!]
\begin{center}
\begin{tabular}{c|ccc|ccc|ccc}
\hline
Variable $c^{-2}$                         & \multicolumn{3}{c|}{Case 1} & \multicolumn{3}{c|}{Case 2} & \multicolumn{3}{c}{Case 3} \\ \hline
$k$\textbackslash{}$N$ & $2\times2$    & $4\times4$  & $8\times8$ &$2\times2$    & $4\times4$  & $8\times8$ &$2\times2$    & $4\times4$  & $8\times8$  \\ \hline
100                       & 8(7)& 18(16)& 30(27)& 8(7)& 20(18)& 34(30)& 8(7)& 17(15)& 28(27)\\
150                        & 7(6)& 16(15)& 28(25)& 7(6)& 17(16)& 30(27)& 7(6)&15(14)& 27(24)\\
200                        & 7(6)& 15(14)& 27(25)& 7(6)& 16(15)& 29(26)& 7(6)& 14(14)& 25(24)\\
250                      & 6(5)& 14(13)& 26(24)&6(6)& 15(14)& 28(26)& 7(6)& 13(13)& 25(23)\\
300                        & 6(5)& 13(13)& 26(23)& 6(6)& 15(14)& 26(25)& 6(6)& 13(12)& 25(22)\\
350                      & 6(5)& 13(12)& 25(23)&6(6)& 15(14)& 25(24)& 6(6)& 12(12)& 23(21)\\ \hline
\end{tabular}
\caption{Iteration counts of fixed-point iteration (in brackets, GMRES iterations) to obtain the residual reduction \eqref{eq:resred} for   RAS-PML  with checkerboard decompositions, where one iteration consists of one exhaustive sequence of four orderings.
}
\label{tab:variablecRAS-PML2}
\end{center}
\end{table}

\begin{table}[h!]
\begin{center}
\begin{tabular}{c|ccc|ccc|ccc}
\hline
Variable $c^{-2}$                         & \multicolumn{3}{c|}{Case 1} & \multicolumn{3}{c|}{Case 2} & \multicolumn{3}{c}{Case 3} \\ \hline
$k$\textbackslash{}$N$ & $2\times2$    & $4\times4$  & $8\times8$ &$2\times2$    & $4\times4$  & $8\times8$ &$2\times2$    & $4\times4$  & $8\times8$ \\ \hline
100                       & 2& 3& 4&2& 4& 5& 2& 3& 4\\
150                        & 2& 3& 4& 2& 3&4& 2&3& 4\\
200                        & 2& 3& 4& 2& 3& 4& 2& 3& 4\\
250                      & 2& 3& 4& 2& 3& 4& 2& 3& 4\\ 300                        & 2& 3& 4& 2& 3& 3& 2& 3& 3\\
350                      & 2& 3& 3& 2& 3& 3& 2& 3& 3\\ \hline
\end{tabular}
\caption{Iteration counts of fixed-point iteration  to obtain the residual reduction \eqref{eq:resred} for  RMS-PML with checkerboard decompositions,
where one iteration consists of one exhaustive sequence of four orderings.}
 \label{tab:variablecRMS-PML2}
\end{center}
\end{table}

\subsection{Experiment III: variable coefficient in the differential operator}\label{sec:num_A}

Remark \ref{rem:A} above described how the results in Theorems \ref{thm:strip}-\ref{thm:sweep} and \ref{thm:gen} do not hold for the problem 
\beq\label{eq:APDE2}
k^{-2}\nabla_\newtheta \cdot( A \nabla_\newtheta u) + c^{-2}u = -f \quad\tin \Rea^d, \qquad u=0 \quad \ton \partial\Omega,
\eeq
where 
\beqs
\nabla_\newtheta := 
\left(
\begin{array}{c}
{\displaystyle\frac{1}{1+ig_{1}'(x_1)} \partial_{x_1} }\\
\vdots
\\
{\displaystyle\frac{1}{1+ig_{d}'(x_d)} \partial_{x_d}}
\end{array}
\right)
\eeqs
and $A$ is a general symmetric-positive-definite-matrix-valued function with $\supp(I-A)\subset \Omega_{\rm int}$.
This experiment demonstrates that the parallel Schwarz method diverges when applied to \eqref{eq:APDE2}, where the local problems involve the operator \eqref{eq:APDEj}.

In common with the previous two experiments, $\Omega_{\rm int}$ is the unit square in 2-d, now centred at the origin to make the expressions for $A$ simpler. 
We let $\Omega = (-0.5-\kappa, 0.5+\kappa)\times (-0.5-\kappa, 0.5+\kappa)$ and 
\beq\label{eq:badA}
A(x) = 
	\begin{pmatrix}
	1& 0\\0&1
	\end{pmatrix} +
	\begin{pmatrix}
	0&1\\1&0
	\end{pmatrix} \Big(\big(0.4 - |x|\big) +\big| 0.4 - |x|\big|\Big);
\eeq
i.e., $A$ transitions from being equal to $\begin{pmatrix} 1 & 0.8\\0.8&1\end{pmatrix}$ at the origin to being equal to $\begin{pmatrix} 1 & 0\\0&1\end{pmatrix}$ near the start of the PML.
The variational formulation of \eqref{eq:APDE2} is \eqref{eq:global_vari} with the sesquilinear form now 
\begin{equation*}
a(u,v):=\int_{\Omega_j} \bigg(k^{-2}\big(\widetilde{A} \nabla u) \cdot \nabla \overline{v}  - (\widetilde\beta\cdot \nabla u)\overline{v} \big)-  c^{-2}u \overline{v} \bigg) dx\quad \tfor u,v\in  H_0^1(\Omega), 
\end{equation*}
where 
$$\widetilde{A}(x) =\begin{pmatrix}\frac{1}{\gamma_{1}(x_1)}&0\\
0& \frac{1}{\gamma_{2}(x_2)} \end{pmatrix} 
A(x)
\begin{pmatrix}\frac{1}{\gamma_{1}(x_1)}&0\\
0& \frac{1}{\gamma_{2}(x_2)} \end{pmatrix} 
\quad \text{and}\quad
\widetilde{\beta}(x) =
\begin{pmatrix}\frac{1}{\gamma_{1}(x_1)}&0\\
0& \frac{1}{\gamma_{2}(x_2)} \end{pmatrix} A(x)
\begin{pmatrix}\frac{\gamma_{1}'(x_1)}{\gamma_{1}^2(x_1)}\\
 \frac{\gamma_{2}'(x_2)}{\gamma_{2}^2(x_2)} \end{pmatrix}
$$
with $\gamma_{1}(x_1) := 1+ig_{1}'(x_1)$ and $\gamma_{2}(x_2) := 1+ig_{2}'(x_2)$ as before.
The local sesquilinear forms $a_j(\cdot,\cdot)$ are defined analogously, with $g_\ell'$, $\ell=1,2,$ replaced by $g_{\ell,j}'$.

We consider a strip decomposition with two subdomains and overlap $\delta=1/40$. We display the results for PML width $\kappa =1/40, 1/20, 1/10$.
The number of wavelengths in the PML at $k=350$ are then $1.39, 2.78,$ and $5.56$ respectively.

Table \ref{tab:variable-diffusion} displays 
the number of fixed-point iterations (in brackets, GMRES iterations) required to achieve 
the residual reduction \eqref{eq:resred}, both for the coefficient $A$ given by \eqref{eq:badA} and (for comparison) 
$A\equiv I$. Table \ref{tab:variable-diffusion} shows that the parallel Schwarz method applied to \eqref{eq:APDE2} with $A$ given by \eqref{eq:badA} converges when $\kappa =1/40$ (with iteration counts almost identical to the case $A\equiv I$), but diverges when $\kappa =1/20$ and when $\kappa=1/10$. 
  
When PML is used to approximate the radiation condition, the accuracy increases with the width of the PML (see 
\cite[Theorem 5.5]{KiPa:10}, \cite[Theorem 5.7]{BrPa:13}, \cite[Theorems 1.2 and 1.5]{GLS2}).
The fact the the behaviour of the parallel Schwarz method with $A\not\equiv I$ gets worse as the PML width increases is consistent 
with the fact, highlighted in Remark \ref{rem:A}, that PML is not intended to be used when $A\not\equiv I$ (in the sense that the imaginary part of the principal symbol is not single-signed).

\begin{table}[h!]
\begin{center}
\begin{tabular}{c|ccc|ccc}
\hline
& \multicolumn{3}{c|}{$A$ given by \eqref{eq:badA}}    &\multicolumn{3}{c}{$A\equiv I$} \\ \hline
$k$\textbackslash{}$\kappa$&	\multicolumn{1}{c}{$\frac{1}{40}$}&	\multicolumn{1}{c}{$\frac{1}{20}$}&	\multicolumn{1}{c|}{$\frac{1}{10}$}&	\multicolumn{1}{c}{$\frac{1}{40}$}&	\multicolumn{1}{c}{$\frac{1}{20}$}&	\multicolumn{1}{c}{$\frac{1}{10}$}  \\
\hline
100                        & 7(7) &diverges(14)&diverges(24) &6(6)  &4(4) &7(7)\\
150                        & 5(5)   &diverges($>100$)&diverges($>100$) &4(4)  &3(3) &5(5)  \\
200                       & 5(5)  &diverges($>100$)&diverges($>100$)   &4(4)   &3(3)&4(4)\\
250                        & 4(5) &diverges($>100$)&diverges($>100$)   &4(4)   &3(3)&3(3)  \\ 
300                        & 4(5) &diverges($>100$)&diverges($>100$)    &4(4)  &3(3) &3(3) \\
350                        & 4(4) &diverges($>100$)&diverges($>100$)    &3(4) &3(3)  &3(3)\\ \hline
\end{tabular}
\caption{Iteration counts (GMRES)(tol=1e-6) for RAS-PML using two subdomains for the problem \eqref{eq:APDE2} with $A$ given by \eqref{eq:badA}.}
\label{tab:variable-diffusion}
\end{center}
\end{table}

\appendix

\section{Extending the main results to more general complex absorbing-layer operators}\label{sec:general}

A careful examination of the assumptions used in the proof of Theorems \ref{thm:strip}-\ref{thm:gen} shows that these results hold under the following abstract assumptions, in which $P_{\mathcal A}$ and $P_{\mathcal A}^j$ ($\mathcal A$ standing for ``absorbing'') play the role of (and generalise)  $P_{\newtheta}$ and $P_{\newtheta}^j$, respectively.

\begin{theorem} \label{thm:abstract}
Let $\{\Omega_{j}\}_{j=1}^N$ be an overlapping decomposition of $\Omega$ associated with the
partition of unity $\{ \chi_j \}_{1\leq j \leq N}$. Let $P, P_{\mathcal A}, P_{\mathcal A}^j \in \Psi^2_\hbar(\mathbb R^d)$ be second-order semiclassical pseudodifferential operators (in the sense of \S\ref{sec:pseudo}) with $p, p_{\mathcal A},p_{\mathcal A}^j$ denoting their semiclassical principal symbols and $\Phi_t,
\Phi^{\mathcal A}_t, \Phi_t^{\mathcal A, j}$ denoting the Hamiltonian flows associated to $\operatorname{Re}p$, $\operatorname{Re}p_{\mathcal A}$ and $\operatorname{Re}p_{\mathcal A}^j$ respectively (i.e. defined by \eqref{eq:Hamilton} with $p$ replaced by the appropriate symbol).
Assume further that $P, P_{\mathcal A},$ and $P_{\mathcal A}^j$ satisfy the following.
\begin{enumerate}
\item \label{it:gen_cc} \emph{(Constant coefficients near $\partial \Omega$)} $P_{\mathcal A}$ and $P^j_{\mathcal A}$ can be written near $\partial \Omega$ as second-order differential  operators with constant coefficients and zero first-order coefficients.
\item \label{it:gen_la} 
\emph{(Agreement of $P_{\mathcal A}$ and $P_{\mathcal A}^j$)} $P_{\mathcal A} = P_{\mathcal A}^j$ on $\Omega_{{\rm int}, j}$.

\item \emph{(Ellipticity)}
\begin{enumerate}
\item \label{it:gen_ee} \emph{(Escape to ellipticity (cf.~Lemma \ref{lem:go_to_elliptic}))} For any $\rho \in T^*\widetilde{\Omega}_j$ not trapped backward in time, there exists $\tau \leq 0$ so that
$$
\gamma_{[\tau, 0]}(\rho) \subset T^*\widetilde{\Omega}_j, \quad \text{ and } \quad  \Phi_{\tau}(\rho) \in \big\{ |p_{\mathcal A}^j(x,\xi)|> 0\big\} .
$$
Furthermore, the same property holds replacing $\widetilde{\Omega}_j$ by $\mathbb R^d$.
\item \label{it:gen_ei} \emph{(Ellipticity at infinity (cf.~Lemma \ref{lem:comp_symbol}))} For any compact $K$, there exists $C>0$ such that if $x \in K$ and $|\xi| \geq C$ then $|p_{\mathcal A}^j(x,\xi)|\geq C^{-1}$. Furthermore the same property holds for $p_{\mathcal A}$.
\end{enumerate}
\item \emph{(Respecting propagation)}
\begin{enumerate}
\item \label{it:gen_rp1} On $\{ p_{{\mathcal A}}^j = 0\}$, $\Phi_t^{{\mathcal A}, j}  = \Phi_t$ for any $t\in \mathbb R$; and on $\{ p_{{\mathcal A}} = 0\}$, $\Phi_t^{{\mathcal A}}  = \Phi_t$ for any $t\in \mathbb R$ (cf.~Lemma \ref{lem:traj_energy_surf}).
\item \label{it:gen_rp2} $\operatorname{Im} p_{{\mathcal A}}^j \leq 0$ near $\{ p_{\mathcal A}^j = 0\}$, and $\operatorname{Im} p_{\mathcal A} \leq 0$ near $\{ p_{\mathcal A} = 0\}$.
\end{enumerate}
\end{enumerate}
Then the results of Theorems \ref{thm:strip}, \ref{thm:sweep_strip}, \ref{thm:check}, \ref{thm:sweep}, \ref{thm:gen} hold (replacing $(P_{\newtheta}, P_{\newtheta}^j)$ by $(P_{\mathcal A}, P_{\mathcal A}^j)$ in the definition of the iterates).
\end{theorem}

\begin{proof}[Proof of Theorem \ref{thm:abstract}]
As discussed in \S\ref{sec:3_ingredients},
the proofs of Theorems \ref{thm:strip}, \ref{thm:sweep_strip}, \ref{thm:check}, \ref{thm:sweep}, \ref{thm:gen} are the consequences of 
\ben
\item[(i)] Algebra of the error propagation, showing which words appear in the products $\mathcal T_w$.
\item[(ii)] Semiclassical analysis, showing that $\mathcal T_w = O(\hbar^\infty)$ if $w$ is not allowed (Lemma \ref{lem:notallowed}).
\item[(iii)] Properties of the trajectories of the flow associated with $P$, dictating which $w$ are allowed.
\end{enumerate}
By Point \ref{it:gen_la},
 the regions $P^j_{\mathcal A}- P_{\mathcal A}$ are unchanged, so that, in particular,
 \beq\label{eq:Tuesday1}
\chi_j \equiv 0 \quad\ton \quad \supp(P_{\mathcal A}- P_{\mathcal A}^j).
\eeq
This implies that (i) above remains unchanged. 
We define \emph{both} the notion of following a word \emph{and} the quantity $\mathcal{N}$ 
exactly as before (i.e., by Definition \ref{def:follow} and \eqref{eq:mathcalN}, respectively).

For (ii), we observe that 
the proof of Lemma \ref{lem:notallowed} uses only the following properties of $P$, $P_{\newtheta}$ and $P_{\newtheta}^j$: 
\begin{itemize}
\item Point \ref{it:symb_ell_infnt} of Lemma \ref{lem:comp_symbol} (ellipticity at infinity):~this is Point \ref{it:gen_ei}, and
\item Lemma \ref{lem:key_prop2} (a composite map follows the associated word).
\end{itemize}
The proof of Lemma \ref{lem:key_prop2} in turn uses only:
\begin{itemize}
\item the fact that in the extended region the PML is linear and the PDE coefficients are constant (this is used to obtain the expression \eqref{eq:rhsweird} involving the extension operator $S_j$ of Definition \ref{def:ext}): this is a consequence of Point \ref{it:gen_cc},
\item the property \eqref{eq:Tuesday1}, and
\item  Lemma \ref{lem:res_est} (PML solution-operator bound) and  Lemma \ref{lem:key_prop} (key propagation lemma).
\end{itemize}
The proofs of Lemma \ref{lem:res_est} and Lemma \ref{lem:key_prop} use only the following:~for both 
Lemma \ref{lem:res_est} and Lemma \ref{lem:key_prop}:
\begin{itemize}
\item Lemma \ref{lem:go_to_elliptic} (escape to ellipticity): this is Point \ref{it:gen_ee},
\item Lemma \ref{lem:FPR} (forward propagation of regularity from an elliptic point), the proof of which uses only Points \ref{it:gen_rp1} and
 \ref{it:gen_rp2},
   \item Lemma \ref{lem:ext_h} (the homogeneous Dirichlet extension of a solution is a solution on the extended subdomain), the proof of which only uses Point \ref{it:gen_cc}.  \end{itemize}
\end{proof}

\begin{theorem}\mythmname{Approximation of outgoing Helmholtz solutions}\label{thm:CAP}
Assume that $P$ and $P_{\mathcal A}$ satisfy the assumptions of Theorem \ref{thm:abstract},
and, in addition, that 
\begin{enumerate}
\item $P = -\hbar^2\nabla \cdot A \nabla  - c^{-2}$, with $A$ smooth and positive definite, $c$ smooth and strictly positive, and $(A, c) = (\operatorname{Id}, 1)$ in a neighbourhood of $\mathbb R^d \backslash \Omega_{\rm int}$.
\item $P=P_{\mathcal A}$ in $\Omega_{\rm int}$.
\item \emph{Either} (i) $P$ is non-trapping, \emph{or} (ii) the solution operators for $P_{\mathcal A}$ and $P$, as operators on,  respectively, $H^1_0(\Omega)$ and $H^1(\Rea^d)$, are bounded polynomially in $\hbar^{-1}$.
\end{enumerate}
Let $f\in H^{-1}(\red{\Omega})$ \red{with $\operatorname{supp}f\subset \overline{\Omega_{\rm int}}$}, and let $u$ and $v$ satisfy
$$
\begin{cases}
P_{\mathcal A} u = f \text{ in }\Omega, \\
u = 0  \text{ on }\partial\Omega,
\end{cases}
\hspace{1.5cm}
\begin{cases}
P v = f \text{ in }\mathbb R^d, \\
v \text{ satisfies the Sommerfeld radiation condition}.
\end{cases}
$$
Then
$$
\tfa s>0 \text{ and } \chi \in C^\infty_c(\Omega_{\rm int}), \quad \Vert \chi(u - v) \Vert_{H^s_\hbar} = O(\hbar^\infty) \N{f}_{H^{-1}_\hbar(\red{\Omega})}.
$$
\end{theorem}
\begin{proof}
Recall from the proof of Theorem \ref{thm:abstract} that, under the assumptions of  Theorem \ref{thm:abstract}, 
Lemma \ref{lem:key_prop} and Lemma \ref{lem:res_est} hold with $P_{\rm s}$ replaced with $P_{\mathcal A}$.

By dividing $u$ and $v$ by $\|f\|_{H_\hbar^{-1}(\red{\Omega})}$, we can assume, without loss of generality, that $\|f\|_{H_\hbar^{-1}(\red{\Omega})}=1$. 
When Point 3.(ii) holds, both $u$ and $v$ are tempered directly by assumption. 
When Point 3.(i) holds, $v$ is tempered by the non-trapping solution-operator bound for $P$ 
(from \cite{MeSj:82} together with either \cite[Theorem 3]{Va:75}/ \cite[Chapter 10, Theorem 2]{Va:89} or \cite{LaPh:89}, or \cite[Theorem 1.3 and \S3]{Bu:02})
and $u$ is tempered by Lemma \ref{lem:res_est} (with $P_{\rm s}$ replaced with $P_{\mathcal A}$).

Let $\psi \in C^\infty_c(\Omega)$ be such that $\psi \equiv 1$ near $\Omega_{\rm int}$. Let
$$
z := u - \psi v,
$$
and observe that 
$$
\begin{cases}
P_{\mathcal A} z = - [P_{\mathcal A}, \psi]v - \psi(P_{\mathcal A} - P)v =: F \text{ in }\Omega, \\
z = 0  \text{ on }\partial\Omega.
\end{cases}
$$
Since $v$ satisfies the Sommerfeld radiation condition,
$$
\operatorname{WF}_\hbar v \subset \Big\{ \rho, \; \exists t <0, \Phi_t(\rho) \in T^*\big(\operatorname{supp} f\big)\Big\};
$$
indeed, this is a classical consequence of Rellich's uniqueness theorem, the outgoing properties of the free resolvent, and propagation of singularities; see e.g., \cite[Proof of Lemma 3.4]{GLS1}. 
Therefore
$$
\operatorname{WF}_\hbar F \subset \Big\{ \rho, \; \exists t <0, \Phi_t(\rho) \in T^*\big(\operatorname{supp} f\big)\Big\} \cap \Big(T^*\big(\operatorname{supp} \psi(P_{\mathcal A} - P)\big) \cup T^* \big(\supp \nabla \psi\big) \Big).
$$
Let $\chi \in C^\infty_c(\Omega_{\rm int})$;
we aim to show that 
\beq\label{eq:noWFz}
\operatorname{WF}_\hbar (\chi z) = \emptyset,
\eeq
from which Lemma \ref{lem:osci}, Part \ref{it:osci1}, yields the result.

Lemma \ref{lem:key_prop} with $P_{\rm s}$ replaced with $P_{\mathcal A}$ gives that
$$
\operatorname{WF}_\hbar (\chi z) \subset
\Big\{ \rho \in T^*\big( \operatorname{supp} \chi\big)\hspace{0.3cm}\text{s.t.} \hspace{0.3cm} \exists t \leq 0, \;  \Phi_t(\rho) \in \operatorname{WF}_\hbar (F) \Big\},
$$
hence
\begin{multline*}
\operatorname{WF}_\hbar (\chi z) \subset
\Big\{ \rho \in T^*\big( \operatorname{supp} \chi\big)\hspace{0.3cm}\text{s.t.} \hspace{0.3cm} \exists t_2 \leq t_1 \leq 0, \\  \Phi_{t_1}(\rho) \in T^*\big(\operatorname{supp} \psi(P_{\mathcal A} - P) \big)\cup T^*\big( \supp \nabla \psi \big)\; \text{ and }\; \Phi_{t_2}(\rho) \in T^*\big( \supp f\big)   \Big\}.
\end{multline*}
The trajectories for $P$ are straight lines in a neighbourhood of $\mathbb R^d \backslash\Omega_{\rm int}$, but the support properties of $\chi$, $\psi(P_{\mathcal A} - P)$, and $f$, imply that no smooth curve going from $T^*(\operatorname{supp} \chi)$ to $T^*(\operatorname{supp} \psi(P_{\mathcal A} - P))\cup T^*( \supp \nabla \psi)$ then back to $T^* (\supp f)$ can possibly be straight in a neighbourhood of $\mathbb R^d \backslash\Omega_{\rm int}$. We therefore obtain \eqref{eq:noWFz} and the result follows.
\end{proof}

\bre[Comparing Theorems \ref{thm:outgoing_approx} and \ref{thm:CAP}]
Comparing Theorems \ref{thm:outgoing_approx} and \ref{thm:CAP}, we see that the advantage of PML truncation is that it controls the error on the whole of $\Omega_{\rm int}$, whereas the more general approximation considered in Theorem \ref{thm:CAP} only controls the error in compact subsets of $\Omega_{\rm int}$. This difference is because of the property of PML that the scaled solution without truncation is precisely the scattering solution (i.e., \eqref{eq:agreement} below), thanks to analyticity.
\ere

\begin{example}[Complex absorbing potential]\label{ex:CAP}
Let $A \in C^\infty(\mathbb R^d;\mathbb R^{d\times d}))$ be such that $A(x)$ is symmetric positive definite for every $x\in \Rea^d$, $\coeffc\in C^\infty(\mathbb R^d, \mathbb R)$ be such that $\coeffc>0$, and $(A, \coeffc) = (\operatorname{Id}, 1)$ near $\partial \Omega$. Define
$$
P := -k^{-2} \nabla \cdot (A\nabla ) - \coeffc^{-2},
$$
\beqs
P^j_{\mathcal A} := -k^{-2} \nabla \cdot (A\nabla ) - \coeffc^{-2} - i V_j, \quad P_{\mathcal A} := -k^{-2} \nabla \cdot (A\nabla ) - \coeffc^{-2} - i V,
\eeqs
where $V_j, V \in C^\infty(\mathbb R^d, \mathbb R)$ are such that $V, V_j \geq 0$, $\{V_j > 0 \} = \widetilde{\mathcal L}_j$, $\{V > 0 \} = \widetilde{\mathcal L}$
(in the notation of \S\ref{ss:def_geo}), and $V$ and $V_j$ are constant near $\partial \Omega$. 
Then, the assumptions of Theorem \ref{thm:abstract} are satisfied. In particular, the results of Theorems \ref{thm:strip}, \ref{thm:sweep_strip}, \ref{thm:check}, \ref{thm:sweep}, \ref{thm:gen} hold (replacing $(P_{\newtheta}, P_{\newtheta}^j)$ by $(P_{\mathcal A}, P_{\mathcal A}^j)$ in the definition of the iterates).
\end{example}
\begin{proof}
We check the assumptions of Theorem \ref{thm:abstract}. Point \ref{it:gen_cc} holds because $V$ and $V_j$ are constant near $\partial \Omega$. Point \ref{it:gen_la} holds
by the support properties of $V$ and $V_j$. 
 Point \ref{it:gen_ee} holds because any trajectory starting from $\rho$ goes to $T^*\widetilde{\mathcal L}_j$, but $\{ |p_{\mathcal A}^j| > 0\} \supset \{ |\operatorname{Im}p_{\mathcal A}^j| > 0\} = \{ V_j > 0\} = T^*\widetilde{\mathcal L}_j$. Point \ref{it:gen_ei} holds because   $|p_{\mathcal A}^j| > |\operatorname{Re} p_{\mathcal A}^j|$ and $A$ is symmetric positive definite. Point \ref{it:gen_rp1} holds because $\operatorname{Re} p_{\mathcal A}^j = \operatorname{Re} p_{\mathcal A} = \operatorname{Re} p$, and Point \ref{it:gen_rp2} holds because
$\operatorname{Im} p_{\mathcal A}^j = - V_j \leq 0$ and $\operatorname{Im} p_{\mathcal A} = - V \leq 0$. 
\end{proof}

\begin{example}[PML on the main domain, CAP on subsets of the subdomains]\label{ex:CAP+PML}
Let $A \in C^\infty(\mathbb R^d;\mathbb R^{d\times d}))$ be such that $A(x)$ is symmetric positive definite for every $x\in \Rea^d$, $\coeffc\in C^\infty(\mathbb R^d, \mathbb R)$ be such that $\coeffc>0$, and $(A, \coeffc) = (\operatorname{Id}, 1)$ near $\partial \Omega$. Define
$$
P := -k^{-2} \nabla \cdot (A\nabla ) - \coeffc^{-2},
$$
$$
 P_{\mathcal A} := -k^{-2} \nabla_{\rm s} \cdot (A\nabla_{\rm s} ) - \coeffc^{-2}, \quad (\nabla_{\rm s})_\ell := \frac{1}{1+ig'_{\ell}(x_\ell)}\partial_{x_\ell},
$$
$$
P^j_{\mathcal A} := -k^{-2} \nabla_{\rm s} \cdot (A\nabla_{\rm s} ) - \coeffc^{-2} - i V_j, 
$$
where $V_j\in C^\infty(\mathbb R^d, \mathbb R)$ are such that $V_j\geq 0$ and (in the notation of \S\ref{ss:def_geo})
$$\supp V_j=\bigcup_{\mathfrak e \in \mathcal Z^j \setminus \mathcal Z^j_{\partial \Omega}}
\widetilde{\mathfrak e}
$$
(in the PML setting of the rest of the paper, the region on the right-hand side of this last equality is $\bigcup_{\ell=1}^d \overline{\big\{ x \in \Rea^d\,:\, g_{\ell,j}(x_\ell) \neq g_{\ell}(x_\ell)\big\}}$).

Then, the assumptions of Theorem \ref{thm:abstract} are satisfied. In particular, the results of Theorems \ref{thm:strip}, \ref{thm:sweep_strip}, \ref{thm:check}, \ref{thm:sweep}, \ref{thm:gen} hold (replacing $(P_{\newtheta}, P_{\newtheta}^j)$ by $(P_{\mathcal A}, P_{\mathcal A}^j)$ in the definition of the iterates).
\end{example}
\begin{proof}
The assumptions of Theorem \ref{thm:abstract} are verified for $P_{\mathcal A}$ by remarking that the results of \S\ref{s:preli} and \S\ref{s:sec_prop} (in particular, Lemma \ref{lem:comp_symbol}, Lemma \ref{lem:traj_energy_surf} and Lemma \ref{lem:FPR}),  still hold for $P_{\mathcal A}$ instead of $P_{\rm s}$ 
using the fact that $A = \operatorname{Id}$ near $\widetilde {\mathcal L}$. These results also hold for $P^j_{\mathcal A}$ by this same reasoning combined with the arguments in the proof of Example \ref{ex:CAP}. Observe that the key point is that $A = \operatorname{Id}$ whenever any PML-scaling is active.
\end{proof}
\section{Recap of results from semiclassical analysis} \label{app:sc}

We recap here some standard definitions and results of semiclassical analysis that we use throughout
the paper. The two main tools that we use are, as highlighted in \S\ref{sec:sketch_prop}, semiclassical ellipticity and forward propagation of regularity, in the forms stated below as Lemma \ref{lem:ell_wf}  and Lemma \ref{lem:FPR_app}, respectively.

\subsection{Weighted Sobolev spaces}\label{sec:SC1}

Recall that $\hbar:= k^{-1}$. 
The \emph{semiclassical Fourier transform} is defined by 
$$
(\mathcal F_{\hbar}u)(\xi) := \int_{\mathbb R^d} \exp\big( -i x \cdot \xi/\hbar\big)
u(x) \, d x,
$$
with inverse
\beqs
(\mathcal F^{-1}_{\hbar}u)(x) := (2\pi \hbar)^{-d} \int_{\mathbb R^d} \exp\big( i x \cdot \xi/\hbar\big)
u(\xi)\, d \xi;
\eeqs
see \cite[\S3.3]{Zw:12}; i.e., the semiclassical Fourier transform is just the usual Fourier transform with the transform variable scaled by $\hbar$. 
These definitions imply that, with $D:= -i \partial$,
\beq\label{eq:FTelement}
\cF_h \big( (\hbar D)^\alpha) u\big) = \xi^\alpha \cF_\hbar u \quad \tand\quad \N{u}_{L^2(\Rea^d)} = \frac{1}{(2\pi \hbar)^{d/2}}\N{\cF_\hbar u}_{L^2(\Rea^d)}; 
\eeq
see, e.g., \cite[Theorem 3.8]{Zw:12}.
Let 
\beqs
H_\hbar^s(\Rea^d):= \Big\{ u\in \mathcal{S}'(\Rea^d) \,\text{ such that }\, \langle \xi\rangle^s (\cF_\hbar u) \in L^2(\Rea^d) \Big\},
\eeqs
where $\langle \xi \rangle := (1+|\xi|^2)^{1/2}$, $\mathcal{S}(\Rea^d)$ is the Schwartz space (see, e.g., \cite[Page 72]{Mc:00}), and $\mathcal{S}'(\Rea^d)$ its dual.
Define the norm
\beqs
\vertiii{u}_{H_\hbar^m(\Rea^d)} ^2 = \frac{1}{(2\pi \hbar)^{d}} \int_{\Rea^d} \langle \xi \rangle^{2m}
 |\mathcal F_\hbar u(\xi)|^2 \, d \xi.
\eeqs
The properties \eqref{eq:FTelement} imply that the space $H_h^s(\Rea^d)$ is the standard Sobolev space $H^s(\Rea^d)$ with each derivative in the norm weighted by $\hbar:=k^{-1}$. The norm on $H_\hbar^s(\Omega)$ for $\Omega\subset \Rea^d$ is then defined by 
\beqs
\vertiii{v}_{H^s_\hbar(\Omega)}:= \inf_{V|_\Omega =v, V \in H^s_\hbar(\Rea^d)} 
\vertiii{V}_{H^s_\hbar(\Rea^d)};
\eeqs
see, e.g., \cite[Page 77]{Mc:00}.
When working with the wavenumber $k$ (instead of $\hbar$) we write $H_\hbar^s(\Rea^d)$ as 
$H_k^s(\Rea^d)$ and $H_k^s(\Omega)$ for $H^s_\hbar(\Omega)$, and similarly for the norms, to avoid the cumbersome notation 
$H_{k^{-1}}^s(\Rea^d)$ and $H_{k^{-1}}^s(\Omega)$.

When $\Omega$ is Lipschitz and $s\geq 1$, 
the norm $\vertiii{\cdot}_{H_k^s(\Omega)}$ is equivalent (with the constants in the norm equivalence independent of $k$) to the norm $\N{\cdot}_{H^s_k(\Omega)}$ defined by \eqref{eq:weighted_norm} for $s\in \mathbb{Z}^+$ and by interpolation for $s>1$; see, e.g., \cite[\S4]{ChHeMo:15}. For simplicity, from here on we drop the ``triple bar'' notation and use only $\N{\cdot}_{H^s_\hsc}$ for norms on $H^s_\hsc$. 

\red{Finally, at one point (Lemma \ref{lem:res_est}) we need the norm $\| \cdot\|_{H^{1/2}_\hsc(\partial \Omega)}$; this is defined analogously to the standard definition of $\|\cdot\|_{H^{1/2}(\partial\Omega)}$ (see, e.g., \cite[Pages 98 and 99]{Mc:00}) except now using the norm $\|\cdot\|_{H^{1/2}_\hsc(\Rea^{d-1})}$ instead of $\|\cdot\|_{H^{1/2}(\Rea^{d-1})}$.}

\subsection{Semiclassical pseudodifferential operators}\label{sec:pseudo}

The set of all possible positions $x$ and momenta (i.e.~Fourier variables) $\xi$ is denoted by $T^*\Rea^d$; this is known informally as ``phase space". Strictly, $T^*\Rea^d :=\Rea^d \times (\Rea^d)^*$, but 
for our purposes, we can consider $T^*\Rea^d$ as $\{(x,\xi) : x\in \Rea^d, \xi\in\Rea^d\}$.

\begin{definition}[Symbols]
 We say that $a\in C^\infty(T^*\mathbb{R}^d)$ is a symbol of order $m$ (and write $a\in S^m(T^*\mathbb{R}^d)$) if, given multiindices $\alpha$ and $\beta$, there exists $C_{\alpha,\beta}$ such that, for all $(x,\xi)\in T^*\mathbb{R}^d$, 
$$
|\partial_x^\alpha \partial_\xi^\beta a(x,\xi)|\leq C_{\alpha,\beta}\langle \xi\rangle^m.
$$
\end{definition}

\begin{definition}[Semiclassical pseudodifferential operators]
Fix $\chi_0\in C_c^\infty(\mathbb{R})$ to be identically 1 near 0.  
We say that an operator $A:C_c^\infty(\mathbb{R}^d)\to \mathcal{D}'(\mathbb{R}^d)$ 
is a \emph{semiclassical pseudodifferential operator} of order $m$, and write $A\in \Psi_\hbar^m(\mathbb{R}^d)$, if $A$ can be written as
\begin{equation}
\label{e:basicPseudo}
Au(x)=\frac{1}{(2\pi \hbar)^d}\int_{\Rea^d} e^{\frac{i}{\hbar}\langle x-y,\xi\rangle}a(x,\xi)\chi_0(|x-y|)
u(y)
dyd\xi + O(\hbar^\infty)_{\Psi^{-\infty}},
\end{equation}
where $a\in S^m(T^*\mathbb{R}^d)$ and $E=O(\hbar^\infty)_{\Psi^{-\infty}}$, if for all $N>0$ there exists $C_N>0$ such that
$$
\|E\|_{H_\hbar^{-N}(\mathbb{R}^d)\to H_\hbar^N(\mathbb{R}^d)}\leq C_N\hbar^N. 
$$
We use the notation $\operatorname{Op}_\hbar a$ for the operator $A$ in~\eqref{e:basicPseudo}  with $E=0$, and define
$$
\Psi^{-\infty}:=\bigcap_m \Psi^m,\qquad S^{-\infty}:=\bigcap_m S^m,\qquad  \Psi^\infty:=\bigcup_m \Psi^m, \qquad S^\infty:=\bigcup_m S^m.
$$
\end{definition}

\begin{theorem}\mythmname{\cite[Propositions E.17 and E.19]{DyZw:19}} 
\label{thm:basic}
If $A\in \Psi_{\hbar}^{m_1}$ and $B  \in \Psi_{\hbar}^{m_2}$, then
\begin{itemize}
\item[(i)]  $AB \in \Psi_{\hbar}^{m_1+m_2}$,
\item[(ii)]  For any $s \in \mathbb R$, $A$ is bounded uniformly in $\hbar$ as an operator from $H_\hbar^s$ to $H_\hbar^{s-m_1}$.
\end {itemize}
\end{theorem}

\paragraph{\textbf{Principal symbol.}}  There exists a map 
$$
\sigma^m_\hbar:\Psi^m \to S^m/hS^{m-1}
$$
called the \emph{principal symbol map} such that the sequence 
$$
0\to hS^{m-1}\overset{\operatorname{Op}_\hbar}{\rightarrow} \Psi^{m}\overset{\sigma^m_\hbar}{\rightarrow} S^m/hS^{m-1}\to 0
$$
is exact and, for $a\in S^m$,
\beqs
\sigma_\hbar^m\big(\operatorname{Op}_\hbar(a)\big) = a \quad\text{ mod } \hsc S^{m-1};
\eeqs
see \cite[Page 213]{Zw:12}, \cite[Proposition E.14]{DyZw:19}. 
When applying the map $\sigma^m_{\hsc}$ to 
elements of $\Psi^m_\hsc$, we denote it by $\sigma_{\hsc}$ (i.e.~we omit the $m$ dependence) and we use $\sigma_{\hsc}(A)$ to denote one of the representatives
in $S^m$ (with the results we use then independent of the choice of representative). By \cite[Proposition E.17]{DyZw:19},
\beq\label{eq:mult_symbol}
\sigma_\hsc(A B) = \sigma_\hsc(A) \sigma_\hsc(B).
\eeq

\paragraph{\textbf{Compactification at infinity.}} 
To deal with the behaviour of
functions on phase space uniformly near $\xi=\infty$ (so-called \emph{fiber infinity}), we consider the \emph{radial
  compactification} in the $\xi$ variable of $T^*\Rea^d$. This is defined by
$$
\overline{T ^* \mathbb R^d}:= \mathbb R^d \times B^d,
$$
where $B^d$ denotes the closed unit ball, considered as the closure of the
image of $\mathbb R^d$ under the radial compactification map 
$$\RC: \xi \mapsto \xi/(1+\langle \xi
\rangle);$$
see \cite[\S E.1.3]{DyZw:19}.
Near the boundary of the
ball, $\lvert \xi\rvert^{-1}\circ \RC^{-1}$ is a smooth function, vanishing to
first order at the boundary, with $(\lvert \xi\rvert^{-1}\circ \RC^{-1}, \widehat\xi\circ\RC^{-1})$
thus giving local coordinates on the ball near its boundary.  The boundary of the
ball should be considered as a sphere at infinity consisting of all
possible \emph{directions} of the momentum variable. 
When appropriate (e.g., in dealing with finite values of $\xi$ only), we abuse notation by dropping the composition with $\RC$ from our
notation and simply identifying $\mathbb R^d$ with the interior of $B^d$.

\begin{definition}[Operator wavefront set]
We say that a point $(x_0,\xi_0)\in \overline{T^*\mathbb{R}^d}$ is not in the wavefront set of an operator $A\in \Psi^m$, and write $(x_0,\xi_0)\notin \operatorname{WF}A$, if there exists a neighbourhood $U$ of $(x_0,\xi_0)$ such that $A$ can be written as in~\eqref{e:basicPseudo} with 
$$
\sup_{(x,\xi)\in U} | \partial^\alpha_x \partial_\xi^\beta a(x,\xi)\langle \xi\rangle^N|\leq C_{\alpha, \beta N} \hbar^N.
$$
\end{definition}
By \cite[Equation E.2.5]{DyZw:19},
\beq\label{eq:mult_WF}
\operatorname{WF}_\hsc(AB) \subset \operatorname{WF}_\hsc(A) \cap \operatorname{WF}_\hsc(B).
\eeq

\subsection{Tempered distribution}

\begin{definition}[Tempered distributions]\label{def:tempered}
An $\hbar$-dependent family of distributions $v_\hbar \in \mathcal D'(\mathbb R^2)$ is \emph{tempered} if
$$
\tfa\,  \chi \in C^\infty_c(\mathbb R^2) \text{ there exist } C, K>0 \text{ such that } \Vert \chi v_\hbar \Vert_{H^{-K}_\hbar} \leq C \hbar^{-K}.
$$
\end{definition}

\begin{definition}[Semiclassical wavefront set]\label{def:WF}
Let $v$ be a $\hbar$-family of tempered distributions on $\mathbb R^d$. For any $s, M\in \mathbb R$, the strict semiclassical wavefront set of $v$ at regularity $(s, M)$, denoted $\operatorname{WF}_\hbar^{s, M}v$, is defined as follows:~$\rho \in T^* \mathbb R^d$ is \emph{not} in $\operatorname{WF}_\hbar^{s, M}v$ if for any $A \in \Psi_\hbar^\infty(\mathbb R^d)$ so that $\operatorname{WF}_\hbar A$ is a sufficiently close neighbourhood of $\rho$,
$$
\hbar^{-M} \Vert A v \Vert_{H^s_\hbar} \rightarrow 0.
$$
In addition
$$
\operatorname{WF}_\hbar^{s, \infty}v := \bigcap_{M\geq 0} \operatorname{WF}_\hbar^{s, M}v, \hspace{0.5cm}\operatorname{WF}_\hbar^{\infty, \infty}v :=
\bigcap_{s\geq 0}\operatorname{WF}_\hbar^{s, \infty}v,
$$
and $\operatorname{WF}_\hbar v := \operatorname{WF}_\hbar^{\infty, \infty}v$.
\end{definition}

By \cite[Equation E.2.18]{DyZw:19}, if $v$ is tempered (in the sense of Definition \ref{def:tempered}) then
\beq\label{eq:mult_WF1}
\operatorname{WF}_\hsc(Av) \subset \operatorname{WF}_\hsc(A) \cap \operatorname{WF}_\hsc(v).
\eeq

\subsection{Semiclassical ellipticity}\label{sec:elliptic}

We say that $A\in \Psi^m_\hbar(\Rea^d)$ is \emph{elliptic} at $(x_0,\xi_0)\in \overline{T^*\mathbb{R}^d}$ if there exists a neighbourhood $U$ of $(x_0,\xi_0)$ such that 
$$
\inf_{(x,\xi)\in U} |\sigma_\hsc(A)\langle \xi\rangle^{-m}|\geq c >0.
$$

\begin{lemma}\mythmname{Elliptic parametrix \cite[Proposition E.32]{DyZw:19}} \label{lem:EP}
Let $A \in \Psi_\hsc^{m}$ and $B \in \Psi_\hsc^{\ell}$ be such 
that $B$ is elliptic on $\operatorname{WF_\hsc}A$.
Then there exist $E, E' \in \Psi_\hsc^{m-\ell}$ such that
$$
A = BE + O(\hsc^\infty)_{\Psi^{-\infty}} = E'B + O(\hsc^\infty)_{\Psi^{-\infty}}.
$$ 
\end{lemma}

We use the two following applications of Lemma \ref{lem:EP}.

\begin{lemma}[Ellipticity at infinity in the $\xi$ variable] \label{lem:osci}
Let $Q\in \Psi^m_h(\mathbb R^d)$ be such that,  given a compact set $K\subset \Rea^d$, there exists $C>0$ such that for all $x\in K$
\beq\label{eq:Qsymbol}
\text{ if }\quad |\xi| \geq C, \quad \text{ then }\quad |\sigma_\hbar(Q)(x, \xi)| \geq C^{-1} |\xi|^m.
\eeq
Then, for any $(s, M) \in [0, +\infty]^2$, and any $\chi, \widetilde \chi \in C^\infty_c(\mathbb R^2)$ such that $\widetilde \chi = 1$ near $\operatorname{supp} \chi$, the following holds.
\begin{enumerate}
\item \label{it:osci1} For any $\hbar$-family of tempered distributions ${v_\hbar}$ such that $\operatorname{WF}^{s,M}_\hbar(\chi v) = \emptyset$ and $\Vert \widetilde \chi Q v\Vert_{H^{s-m}_\hbar} = o(\hbar^M)$,
$$
 \Vert \chi v \Vert_{H^s_\hbar} = o(\hbar^M).
$$
\item \label{it:osci2}  Given $s>0$ there exists $C>0$ such that, for any $\hbar$-family of tempered distributions ${v_\hbar}$,
$$
\Vert \chi v \Vert_{H^{s}_\hbar} \leq C\Big( \Vert \chi v \Vert_{L^2} + \Vert \widetilde \chi Q v \Vert_{H^{s-m}_\hbar} + O(\hbar^\infty)\Big).
$$
\end{enumerate}
\end{lemma}
\begin{proof}
We begin with Point \ref{it:osci1}.
We do the proof for $M<+\infty$, the proof for $M = +\infty$ is similar.
Let $\psi \in C^\infty_c(\mathbb R)$ be such that $\psi = 1$ near $B(0,A)$. Then, on the one hand, by definition of the wavefront-set,
$$
\hbar^{-M}\Vert \psi(|hD|) \chi v \Vert_{H^s_\hbar}\rightarrow 0.
$$
On the other hand, by the assumption \eqref{eq:Qsymbol} and the properties \eqref{eq:mult_symbol} and \eqref{eq:mult_WF},  
$\widetilde \chi Q$ is semiclassicaly elliptic near $\operatorname{WF}_\hbar (1-\psi(|\hbar D|)) \chi$; hence
$$
(1-\psi(|\hbar D|))\chi = E\widetilde \chi Q + O(\hbar^\infty)_{\psi^{-\infty}_\hbar},
$$
with $E \in \Psi^{-m}_\hbar(\mathbb R^d)$. Applying $v$ to the above gives
$$
\hbar^{-M} \Vert (1-\psi(|hD|)) \chi v \Vert_{H^s_\hbar} \rightarrow 0,
$$
and the result follows.
The proof of Point \ref{it:osci2} is very similar to the proof of Point \ref{it:osci1} using that  given $s>0$ there exists $C>0$ such that $\Vert \psi(|\hbar D|) \chi v \Vert_{H^s_\hbar} \leq C \Vert \chi v \Vert_{L^2}$; this follows from Part (ii) of Theorem \ref{thm:basic} since $ \psi(|\hbar D|)\in \Psi^{-\infty}_\hbar(\Rea^d)$.
\end{proof}

\begin{lemma}[$\operatorname{WF}$ statement from symbolic ellipticity] \label{lem:ell_wf}
Let $Q\in \Psi^m_h(\mathbb R^d)$. Then, for 
 any $\hbar$-tempered family of distributions ${v_\hbar}$, any $\chi, \widetilde \chi \in C^\infty_c(\mathbb R^d)$ such that $\widetilde \chi = 1$
on $\operatorname{supp} \chi $ and any $(s,M) \in (-\infty, +\infty]^2$,
\begin{equation*}
\operatorname{WF}^{s,M}_\hbar(\chi v) \subset \operatorname{WF}^{s-m,M}_\hbar(\widetilde \chi Q v) \cup \big\{ \sigma_\hbar(Q) = 0\big\}.
\end{equation*}
\end{lemma}

If $Q$ is elliptic, i.e., $\{ \sigma_\hbar(Q) = 0\}=\emptyset$, then Lemma \ref{lem:ell_wf} states that an elliptic pseudodifferential operator does not ``move around" mass in phase space; i.e., in any compact set in space, the wave front set of $v$ is contained inside the wavefront set of $Qv$.

\begin{proof}[Proof of Lemma \ref{lem:ell_wf}]
Let  $\rho \notin \operatorname{WF}^{s,M}_\hbar(\widetilde \chi Q v) \cup \big\{ \sigma_\hbar(Q) = 0\big\}$,
and $b \in C^\infty_c(T^* \mathbb R^2)$ with $b = 1$ near $\rho$ be such that
$$
\operatorname{supp} b \cap \operatorname{WF}^{s-m,M}_\hbar(\widetilde \chi Q v) = \emptyset, \hspace{0.3cm}
\operatorname{supp }b \cap \big\{ \sigma_\hbar(Q) = 0\big\} = \emptyset.
$$ 
Then, by \eqref{eq:mult_symbol}, $\widetilde \chi Q$ is semiclassically elliptic on $\operatorname{supp}b \cap T^* \operatorname{supp} \chi$. Therefore, by the elliptic parametrix construction (Lemma \ref{lem:EP}), 
$$
(\operatorname{Op}_\hbar b)\chi = E \widetilde \chi Q + O(\hbar^\infty)_{\Psi_\hbar^{-\infty}},
$$
with $E \in \Psi_\hbar^{-m}$ and $\operatorname{WF}_\hbar E \subset \operatorname{supp} b.$ 
Therefore $\operatorname{WF}_\hbar E \cap \operatorname{WF}^{s-m, M}_\hbar  (\widetilde \chi Q v) = \emptyset$
and applying $v$ to the last displayed equation we see that $\rho \notin \operatorname{WF}^{s,M}_\hbar(\chi v)$, which implies the result.
\end{proof}

\subsection{Propagation of singularities}
\begin{lemma}[Propagation of singularities estimate, {\cite[Theorem E.47]{DyZw:19}}] \label{lem:FPE}
Let $Q\in\Psi^m_\hbar(\mathbb R^d)$, of semiclassical principal symbol $q$. Let $\varphi_t$ be the Hamiltonian flow associated to $\operatorname{Re} q$ (i.e. defined by \eqref{eq:Hamilton} with $p$ replaced by $\operatorname{Re}q$), and $A, B, B_1 \in \Psi_h^0(\mathbb R^d)$ be so that
$$
\operatorname{Im} q \leq 0 \text{ on } \operatorname{WF}_\hbar B_1,
$$
and for all $(x,\xi) \in \operatorname{WF}_\hbar A$
$$
\text{ there exists } T \geq 0\, \text{ s.t. }\, \varphi_{-T}(x,\xi) \in \operatorname{ell}_\hbar B, \text{ and for all } t \in [0, T], \; \varphi_{-t}(x,\xi) \in \operatorname{ell}_\hbar B_1.
$$
Then, for any $s>0$ and any family ${v_\hbar}$ of $\hbar$-tempered distributions (in the sense of Definition \ref{def:tempered})
$$
\Vert A  v \Vert_{H_\hbar^s} \leq C \Vert B  v \Vert_{H_\hbar^s} + C \hbar^{-1} \Vert
B_1 Q  v \Vert_{H_\hbar^{s-m+1}} +  O(\hbar^\infty). 
$$
\end{lemma}

We use the following an application of  Lemma \ref{lem:FPE}.
\begin{lemma}[Forward propagation of regularity] \label{lem:FPR_app} 
Let $Q \in \Psi^m_\hbar(\mathbb R^d)$ with semiclassical principal symbol $q$. Let $\varphi_t$ denote the Hamiltonian flow associated with $\operatorname{Re}q$ (i.e. defined by \eqref{eq:Hamilton} with $p$ replaced by $\operatorname{Re}q$), and assume that
$\operatorname{Im} q \leq 0$. Let ${v_\hbar}$ be a family of $\hbar$-tempered distributions (in the sense of Definition \ref{def:tempered}). Let 
 $(x_0, \xi_0) \in T^* \mathbb R^d$ and $(s, M) \in [0, +\infty] \times [0, +\infty]$. If $(x_0, \xi_0) \notin \operatorname{WF}^{s,M}_\hbar ( v)$ and there exists
$t_0 > 0$  so that
$$
\Big\{ \varphi_t(x_0, \xi_0), \; t\in [0, t_0] \Big\} \cap \operatorname{WF}^{s-m+1, M}_\hbar (\hbar^{-1}Q v) = \emptyset,
$$
then $\varphi_{t_0}(x_0, \xi_0) \notin \operatorname{WF}^{s,M}_\hbar ( v)$.
\end{lemma}

Lemma \ref{lem:FPR_app} is a rigorous statement of Point 1 in \S\ref{sec:sketch_prop}; i.e., that the absence of mass in phase space of the solution of $Q u=f$ propagates along the flow defined by $Q$ as long as the trajectory does not intersect the data $f$.

\begin{proof}[Proof of Lemma \ref{lem:FPR_app}]
By continuity of the flow and the assumptions in the lemma, there exist $A, B, B_1 \in \Psi^0_\hbar$ such that $A$ is elliptic near $\varphi_{t_0}(x_0, \xi_0)$, $\operatorname{WF}_\hbar B$ is an arbitrarily-small neighbourhood of $(x_0, \xi_0)$, and, for any $(x,\xi) \in \operatorname{WF}_\hbar A$,
$$
\text{ there exists } T \geq 0\, \text{ s.t. }\, \varphi_{-T}(x,\xi) \in \operatorname{ell}_\hbar B, \text{ and for all } t \in [0, T], \; \varphi_{-t}(x,\xi) \in \operatorname{ell}_\hbar B_1,
$$
with 
$$
\operatorname{WF}_\hbar B_1 \cap \operatorname{WF}^{s-m+1, M}_\hbar(\hbar^{-1}Q  v) = \emptyset.
$$
Hence, by the propagation of singularities estimate (Lemma \ref{lem:FPE}), 
for any $s>0$,
$$
\Vert A  v \Vert_{H_\hbar^s} \leq C \Vert B  v \Vert_{H_\hbar^s} + C \hbar^{-1} \Vert
B_1 Q  v \Vert_{H_\hbar^{s-m+1}} +  O(\hbar^\infty), 
$$
and the result follows taking $\operatorname{WF}_\hbar B$ small enough.
\end{proof}

\begin{remark}
Working in $\mathbb R^d$, as opposed to on a general manifold defined by coordinate charts, allows us to remove the proper-support assumption appearing in \cite[Proposition E.32]{DyZw:19} and {\cite[Theorem E.47]{DyZw:19}}.
\end{remark}

\section*{Acknowledgements}

\red{The authors thank the anonymous referee for useful comments.}
JG was supported by EPSRC grant EP/V001760/1.
SG was supported by the National Natural Science Foundation of China (Grant number 12201535) and the Guangdong Basic and Applied Basic Research Foundation (Grant number 2023A1515011651).
DL was supported by INSMI (CNRS) through a PEPS JCJC grant 2023. EAS was supported by EPSRC grant EP/R005591/1.

\footnotesize{
\bibliographystyle{plain}
\bibliography{combined.bib}
}

\end{document}